\documentclass[aop]{imsart}

\RequirePackage{amsthm,amsmath,amsfonts,amssymb,dsfont}
\RequirePackage[numbers]{natbib}
\usepackage{mathrsfs}
\usepackage{color,ulem,mathabx,graphicx}

\startlocaldefs
\theoremstyle{plain}
\newtheorem{theorem}{Theorem}[section]
\newtheorem{lemma}[theorem]{Lemma}
\newtheorem{proposition}[theorem]{Proposition}
\newtheorem{corollary}[theorem]{Corollary}
\newtheorem{remark}[theorem]{Remark}

\newcommand{\N}{\mathbb{N}}

\newcommand{\E}{\mathbb{E}}
\newcommand{\Z}{\mathbb{Z}}
\newcommand{\R}{\mathbb{R}}
\newcommand{\tn}{\mathsf{t}_N}

\def \P {{\mathbb{P}}}

\endlocaldefs

\begin{document}

\begin{frontmatter}
\title{Scaling limit of an adaptive contact process}
%\title{A sample article title with some additional note\thanksref{T1}}
\runtitle{Adaptive contact process}
%\thankstext{T1}{A sample of additional note to the title.}

\begin{aug}
%%%%%%%%%%%%%%%%%%%%%%%%%%%%%%%%%%%%%%%%%%%%%%%
%% Only one address is permitted per author. %%
%% Only division, organization and e-mail is %%
%% included in the address.                  %%
%% Additional information can be included in %%
\author[A]{\fnms{Adri\'an}~\snm{Gonz\'alez Casanova}},
\author[B]{\fnms{Andr\'as}~\snm{T\'obi\'as}}
\and
\author[C]{\fnms{Daniel}~\snm{Valesin}}
%%%%%%%%%%%%%%%%%%%%%%%%%%%%%%%%%%%%%%%%%%%%%%
%% Addresses                                %%
%%%%%%%%%%%%%%%%%%%%%%%%%%%%%%%%%%%%%%%%%%%%%%
\address[A]{Institute of Mathematics, National University of Mexico, adrian.gonzalez@im.unam.mx, and\\  Department of Statistics, University of California at Berkeley, gonzalez.casanova@berkeley.edu }
\address[B]{Department of Computer Science and Information Theory, Budapest University of Technology and Economics, and \\ Alfréd Rényi Institute of Mathematics, tobias@cs.bme.hu}
\address[C]{Department of Statistics, University of Warwick, daniel.valesin@warwick.ac.uk}
\end{aug}

\begin{center} May 3, 2023 \end{center}

\vspace{-6pt}

\begin{abstract}
	We introduce and study an interacting particle system evolving on the~$d$-dimensional torus~$(\Z/N\Z)^d$. Each vertex of the torus can be either empty or occupied by an individual of type~$\lambda \in (0,\infty)$. An individual of type~$\lambda$ dies with rate one and gives birth at each neighboring empty position with rate~$\lambda$; moreover, when the birth takes place, the newborn individual is likely to have the same type as the parent, but has a small probability of being a mutant. A mutant child of an individual of type~$\lambda$ has type chosen according to a probability kernel. We consider the asymptotic behavior of this process when~$N\to \infty$ and { simultaneously, the mutation probability} tends to zero fast enough that mutations are sufficiently separated in time, so that the amount of time spent on configurations with more than one type becomes negligible. We show that, after a suitable time scaling and deletion of the periods of time spent on configurations with more than one type, the process converges to a Markov jump process on~$(0,\infty)$, whose rates we characterize.
\end{abstract}

\begin{keyword}[class=MSC]
\kwd[Primary ]{60F99}
\kwd{60K35}
\kwd[; secondary ]{92D15}
\end{keyword}

\begin{keyword}
\kwd{Adaptive contact process}
\kwd{interacting particle systems}
\kwd{scaling limit} 
\kwd{mutation} 
\kwd{fixation} 
\kwd{trait substitution sequence}
\end{keyword}

\end{frontmatter}
%%%%%%%%%%%%%%%%%%%%%%%%%%%%%%%%%%%%%%%%%%%%%%
%% Please use \tableofcontents for articles %%
%% with 50 pages and more                   %%
%%%%%%%%%%%%%%%%%%%%%%%%%%%%%%%%%%%%%%%%%%%%%%
%\tableofcontents

%%%%%%%%%%%%%%%%%%%%%%%%%%%%%%%%%%%%%%%%%%%%%%
%%%% Main text entry area:

\section{Introduction}\label{s_intro}
From its origins~\cite{H74}, the contact process has been a model to study the propagation of an epidemic { 
on a graph. It has a positive parameter~$\lambda > 0$, the birth rate of the pathogen (also called the infection rate in part of the literature).
For the case of the process on~$\Z^d$, there exists a (dimension-dependent) threshold parameter value~$\lambda_c(\Z^d)$ such that the following holds. Assume that the process starts from a non-zero and finite number of infections. If~$\lambda > \lambda_c(\Z^d)$, then the infection can persist indefinitely, whereas if~$\lambda \le \lambda_c(\Z^d)$, then it is naturally wiped out from the population with probability one.
}
On a large box contained in $\Z^d$ or a large~$d$-dimensional torus, for $\lambda<\lambda_c(\Z^d)$ the contact process dies out rapidly with high probability, whereas for $\lambda>\lambda_c(\Z^d)$,  one observes a metastable behaviour for an amount of time that grows exponentially with $N$
\cite{CGOV84,  DL88, DS88, M93, M99, S85}.
	
A natural extension of the contact process consists in allowing the birth rate of the pathogen to be subject to mutations. That is, in a birth event, the new strain and its direct ancestor will typically have the same birth rate, but with a small probability a mutation can occur, causing the birth rate of the new strain to be different from that of its ancestor. If this perturbation allows for an increase in the birth rate with reasonable probability, then strains with higher reproduction rates will overcome their weaker ancestors, so that we will observe adaptation in the contact process. 
	
This paper is concerned with introducing and studying this \textit{adaptive contact process}. In particular, we will focus on cases in which mutations are very rare, so that when a mutation occurs, the competition between the newborn mutant and the existing population is already resolved (by the extinction of one of the two) by the time that a further mutation takes place. On the other hand, mutations are assumed to be sufficiently frequent such that typically, the first successful mutation occurs before the metastable contact process would die out.

{
The adaptive contact process is denoted~$(X_{N,t})_{t \ge 0}$, where~$N$ parametrizes two things: first, the size of the torus where the process evolves, and second, a parameter~$\delta_N$ controlling the probability that a mutation occurs with each birth event (the larger~$N$, the smaller this probability). We then define a process~$(Z_{N,t})$ (taking values in~$(0,\infty)$), obtained as a pathwise function of~$(X_{N,t})$ (including a time rescaling), and interpreted as a trait substitution process.
}
Roughly speaking,~$Z_{N,t}>0\color{black}$ is the reproduction rate of the last strain that managed to invade the whole population before (rescaled) time $t$. We assume that at time zero all the sites are occupied by a strain which has reproduction rate $\lambda_0>\lambda_c(\Z^d)$, so that $Z_{N,0}=\lambda_0$. Then, perhaps, some mutations occur but they are unsuccessful and do not manage to persist.  Imagine that at time $t_1>0$ a mutation with reproduction rate $\lambda_1  \neq \lambda_0 \color{black}$ successfully reaches fixation, making the last individual of reproduction rate $\lambda_0$ die precisely at this time. In this case, the process  $(Z_{N,t})$ jumps from the state $\lambda_0$ to the state $\lambda_1$ at time $t_1$. At time $t_2$, perhaps after several births of mutants who do not manage to fixate, a new strain vanquishes over the previous one, so at time $t_2$ all the occupied sites have birth rate $\lambda_2$ etc.

Our main result, Theorem \ref{thm_main}, gives formally conditions under which the sequence of trait substitution processes converges to a Markov pure jump process with values in $[0,\infty)$. At least on a heuristic level, this provides an analogue of certain invasion models of adaptive dynamics in the framework of interacting particle systems. In Section~\ref{sec:discussion} below we discuss the relation and main differences between those models and ours, and we mention some related open questions.
	
	\subsection{Model and results}\label{s_model}
	We introduce the \textit{adaptive contact process} on the~$d$-dimensional torus~$\Z^d_N:= (\Z/N\Z)^d$ as a Markov jump process~$(X_{N,t})_{t \ge 0}$ with state space~$\{\chi:\Z^d_N \to [0,\infty)\}$. We regard an element~$\chi$ in this state space as a description of the occupation of space by a population: we say that~$u\in \Z^d_N$ is empty if~$\chi(u) = 0$ and that it is occupied by an individual of type~$\lambda$ if~$\chi(u) = \lambda$. 

	In order to define the dynamics, we fix:
\begin{itemize}
	\item a number~$\delta_N \in (0,1)$;%, which will serve as the scaling parameter of the probability of mutation upon birth\color{black};
	\item a  measurable \color{black} function~$b: (0,\infty) \to (0,\infty)$, called the \textit{mutation rate function};
	\item a probability kernel~$K:(0,\infty) \times \{\text{Borel sets of }(0,\infty)\} \to [0,1]$, called the \textit{mutation kernel}, with~$K(\lambda,\{\lambda\}) = 0$ for all~$\lambda$.
\end{itemize}
From a configuration~$\chi$, the possible jumps and associated jump rates are:
\begin{itemize}
	\item \textit{death:} for each~$u$ such that~$\chi(u) \neq 0$, with rate one, the process jumps from~$\chi$ to the configuration~$\chi'$ defined by~$\chi'(u) = 0$,~$\chi'(v) = \chi(v)$ for all~$v \neq u$;
	\item \textit{birth with no mutation:} for each pair~$u\sim v$ of neighboring vertices such that~$\chi(u) \neq 0$ and~$\chi(v) = 0$, with rate
		\[\chi(u)\cdot \max(0,1-\delta_N\cdot b(\chi(u))),\]
		the process jumps from~$\chi$ to the configuration~$\chi'$ defined by~$\chi'(v) = \chi(u)$,~$\chi'(w) = \chi(w)$ for all~$w \neq v$;
	\item \textit{birth with mutation:} for each pair~$u \sim v$ such that~$\chi(u) \neq 0$ and~$\chi(v) = 0$, with rate
		\[\chi(u)\cdot \min(\delta_N \cdot b(\chi(u)),1),\] 
		the process jumps from~$\chi$ to the configuration~$\chi'$ defined by~$\chi'(v) = U$,~$\chi'(w) = \chi(w)$ for all~$w \neq v$, where~$U$ is a random variable distributed as~$K(\chi(u),\cdot)$.
\end{itemize}

{
This description gives rise to a continuous-time Markov chain on~$[0,\infty)^{\Z^d_N}$. We prove that this chain is non-explosive: 
\begin{proposition}\label{prop_doesnt_explode}
	For any~$N > 0$ and any choice of~$\delta_N$, mutation rate function~$b$ and mutation kernel~$K$, the adaptive contact process is non-explosive, that is, for any~$t > 0$, it only performs finitely many jumps in~$[0,t]$ with probability one.
\end{proposition}
We prove this proposition at the end of this section.
A heuristic explanation for this fact is that in our model the mutation events are contained in reproduction events, and in order to reproduce, individuals need to be adjacent to empty spaces. These lead to the fact that the death rate regulates the speed of adaptation. In turn, the fact that death rate is constant in our model excludes the possibility of an explosion. }

We are interested in studying the adaptive contact process for~$N \to \infty$, with~$b$ and~$K$ fixed and~$\delta_N \to 0$. In fact, we assume throughout that
\begin{equation}\label{eq_assumption_delta}
	\text{there exists $\varepsilon_0 > 0$ such that } \delta_N\cdot N^{d+1+\varepsilon_0} \xrightarrow{N \to \infty} 0.
\end{equation}
As we will detail later, this assumption will guarantee that  mutation events are sufficiently separated in time, so that the competition between a homogeneous host population and a newborn mutant is typically resolved (by the disappearance of one of the competing types) before the next mutant appears. On the other hand, we want to ensure that~$\delta_N$ is not so small that the process dies out altogether before any mutation occurs, so we also assume that
\begin{equation}\label{eq_assumption_delta_2}
	\limsup_{N \to \infty} \frac{-\log \delta_N}{\log N} < \infty,
\end{equation}
that is, there exists~$a > 0$ such that~$\delta_N > N^{-a}$ for~$N$ large enough. We discuss this assumption further in Remark~\ref{rem_assumption2} below.

Our main result, Theorem~\ref{thm_main} below, describes the scaling limit of the sped-up process~$(X_{N,(\delta_NN^d)^{-1}\cdot t})_{t \ge 0}$. { The time scaling is justified as follows. In the original process, the number of births in one unit of time is of order~$N^d$, so (if~$\delta_N$ is small compared to~$N^d$) the amount of time elapsed between successive mutations is of order~$(\delta_NN^d)^{-1}$. Hence, speeding up time by~$(\delta_NN^d)^{-1}$ makes it so that the amount of time between mutations is of order one.
}

\color{black} In order to prepare the ground to state the theorem, we need to make brief expositions on the classical contact process and the two-type contact process, giving the relevant definitions along the way.\smallskip

\noindent \textit{The contact process on~$\Z^d$ and~$\Z^d_N$.}
Harris' contact process~\cite{H74} on $\Z^d$ is defined as the
continuous-time Markov process  $(\zeta)_{t\geq 0}$ on $\{ 0,1\}^{\Z^d}$ with infinitesimal pregenerator
\begin{equation}\label{eq_contactgenerator} \mathcal L f(\zeta) = \sum_{x \in \Z^d} (f(\zeta^{x \leftarrow 0})-f(\zeta)) + \lambda \sum_{\substack{x \in \Z^d:\\ \zeta(x) =1}}\;\sum_{\substack{y \sim x:\\ \zeta(y) =0}}   (f(\zeta^{x\leftarrow 1}) - f(\zeta)), \end{equation}
where 
\begin{equation}\label{eq_xarrowi} \zeta^{x \leftarrow i} (y) = \begin{cases} i, &\text{ if } y =x; \\ \zeta(y), &\text{ otherwise}, \end{cases}\qquad i \in \{0,1\}, \end{equation}
and $f \colon \{ 0,1\}^{\Z^d} \to \R$ is a local function. 
The parameter $\lambda>0$ is usually called the infection rate, but we will call it the birth rate of the process. The configuration that is identically equal to zero is an absorbing state for this dynamics. As already mentioned, there exists a critical birth rate $\lambda_c(\Z^d) \in (0,\infty)$ such that, for the contact process started from a single~$1$ at the origin  $o$ of $\Z^d$ \color{black}, this absorbing state is reached almost surely if and only if~$\lambda \le \lambda_c(\Z^d)$ {(see \cite[p.43]{L99})}.

In Section~\ref{ss_one_type_prelim}, we will give the construction of the contact process on general graphs, and in particular on the torus~$\Z^d_N$.

The so-called \textit{upper stationary distribution} of the contact process on~$\Z^d$ { is
\[\mu_\lambda:= \lim_{\substack{t \to \infty\\ \text{(distr.)}}} \zeta^1_t,\]
where~$(\zeta^1_t)_{t \ge 0}$ is the contact process on~$\Z^d$ started from the identically 1 configuration.
The existence of this limit in distribution is proved with standard monotonicity arguments, using the fact that the contact process is an attractive spin system (see~\cite[p.35]{L99})}. In case~$\lambda > \lambda_c(\Z^d)$, this measure is supported on configurations containing infinitely many~$0$'s and~$1$'s.
{ We define
	\begin{equation} \label{eq_def_of_R} R(\lambda):= \lambda\cdot \int_{\{0,1\}^{\Z^d}} \mu_\lambda(\mathrm{d}\zeta)\; \; (1-\zeta(o))\cdot  \sum_{x \sim o} \zeta(x).\end{equation}
We also define~$\hat{\mu}_\lambda$ as the probability measure on~$\{0,1\}^{\Z^d}$ with Radon-Nikodym derivative with respect to~$\mu_\lambda$ given by
\[\frac{\mathrm{d}\hat{\mu}_\lambda}{\mathrm{d}\mu_\lambda}(\zeta) = \frac{\lambda}{R(\lambda)} \cdot (1-\zeta(o)) \cdot \sum_{x \sim o} \zeta(x),\]
that is,~$\hat{\mu}_\lambda$ is the unique measure that satisfies, for all measurable and bounded functions~$f:\{0,1\}^{\Z^d} \to \R$,
	\begin{equation} \label{eq_def_of_hat_mu}\int f\;\mathrm{d}\hat{\mu}_\lambda = \frac{\lambda}{R(\lambda)}\cdot \int_{\{0,1\}^{\Z^d}} \mu_\lambda(\mathrm{d}\zeta)\;f(\zeta)\cdot  (1-\zeta(o))\cdot  \sum_{x \sim o} \zeta(x).\end{equation}
		The results we prove in Section~\ref{s_appearance} suggest that~$\hat{\mu}_\lambda$ and~$R(\lambda)$ are naturally associated to the contact process on~$\Z^d$ with rate~$\lambda$ started from~$\mu_\lambda$, as follows.  The average number of times per unit interval that the origin switches from state~0 to state~$1$ is~$R(\lambda)$, so we think of~$R(\lambda)$ as an \textit{effective birth rate}. Moreover,~$\hat{\mu}_\lambda$ is the \textit{distribution of the occupied landscape} around the origin (and not including the origin itself) at times when there is a~$0 \to 1$ switch at the origin. (To be completely precise, the statements we prove in Section~\ref{s_appearance} pertain to the contact process on~$\Z^d_N$, not on~$\Z^d$; however, it is not difficult to extrapolate the interpretation given in the last few sentences from~$\Z^d_N$ to~$\Z^d$).
}
\smallskip

\noindent \textit{The two-type contact process on~$\Z^d$.} 
Neuhauser's two-type contact process~\cite{N92} on $\Z^d$ is defined as the continuous-time Markov process $(\xi_t)_{t\geq 0}$ on $\{ 0,1,2\}^{\Z^d}$ with infinitesimal pregenerator 
\begin{equation} \label{eq_twotypecontactgenerator}
	\begin{split} \mathscr L f(\xi) = \sum_{x \in \Z^d} (f(\xi^{x \leftarrow 0})-f(\xi)) &+ \lambda \sum_{\substack{x \in \Z^d:\\ \xi(x) = 1}}\; \sum_{\substack{y \sim x:\\ \xi(y) = 0}}  (f(\xi^{y\leftarrow 1}) - f(\xi)) \\[.2cm] &+ \lambda' \sum_{\substack{x \in \Z^d:\\ \xi(x) = 2}}\; \sum_{\substack{y \sim x:\\ \xi(y) = 0}}  (f(\xi^{y\leftarrow 2}) - f(\xi)),
\end{split} \end{equation}
where we extend the definition of $\xi^{x \leftarrow i}$ to the case $i=2\color{black}$ analogously to~\eqref{eq_xarrowi}. The parameter $\lambda>0$ is called the birth rate of type $1$ and the parameter $\lambda'>0$ the birth rate of type $2$.  We emphasize that each type can only occupy sites that are not occupied by the other type. Also note that a process with more than two types (and rates~$\lambda_1,\lambda_2,\lambda_3,\ldots$) could be defined analogously, but we will only need to consider the two-types case.  \color{black}

	Given~$A \subseteq \Z^d \backslash \{o\}$, let~$S(\lambda,\lambda';A)$ denote the probability that type 2 survives (that is, of the event~$\{\forall t\; \exists x:\xi_t(x) = 2\}$), for the two-type contact process on~$\Z^d$ in which type 1 has birth rate~$\lambda$ and initially occupies (only) the set~$A$, and type~$2$ has birth rate~$\lambda'$ and initially occupies (only) the origin. Theorem~1.1 in~\cite{MPV20} states that
\begin{equation}\label{eq:Sdef} S(\lambda,\lambda';A) > 0 \text{ when } \lambda' > \lambda \text{ and } \lambda' > \lambda_c(\Z^d). \end{equation}
We refer the reader to Section~\ref{ss_multiple} for further details regarding the two-type contact process.\smallskip

\noindent \textit{Projection and time change.}
For each~$\chi \in [0,\infty)^{\Z^d_N}$, define
\begin{equation}\label{eq_def_projection}\Phi(\chi):= \begin{cases}0&\text{if }\chi \equiv 0; \\ \lambda &\text{if } \{\chi(u):u \in \Z^d_N\} \cap (0,\infty) =\{\lambda\};\\{\Asterisk}&\text{otherwise,}\end{cases}\end{equation}
	where~$\Asterisk$ is an auxiliary state. The states~$\chi$ in which there are individuals with different reproduction rates  are sent to $\Asterisk$. Our assumptions will imply that the population has the same infection rates most of the time. We introduce $\Asterisk$ to deal with the atypical moments in which individuals of different types coexist.

Assume that the adaptive contact process~$(X_{N,t})_{t \ge 0}$ on~$\Z^d_N$ is started from $X_{N,0} \equiv \lambda_0$. We define an auxiliary process~$(Z_{N,t}\color{black})_{t \ge 0}$ as follows.
We define the time-changed process 
\begin{equation}\label{eq_def_of_Zt}
	\tilde{Z}_{N,t} := \Phi(X_{N,(\delta_NN^d)^{-1}\cdot t}),\quad t \ge 0.
\end{equation}
Next, since the time spent in the auxiliary state will be negligible in the scaling limit, we define a new process~$(Z_{N,t})_{t \ge 0}$ where this state is artificially removed. Formally, we define~$Z_{N,0} = \lambda_0$ and, for~$t > 0$,
\begin{equation}
	\label{eq_def_of_Z}
	Z_{N,t} := \lim_{s \nearrow \alpha_{N,t}} \tilde{Z}_{N,s},\quad \text{where} \quad \alpha_{N,t}:= \sup\{t' \le t:\;\tilde{Z}_{N,t'} \neq \Asterisk\}.
\end{equation}
In words,~$Z_{N,t}$ is the last value different from~$\Asterisk$ attained by~$\tilde{Z}_{N,\cdot}$ before time~$t$.\smallskip

{
\noindent \textit{Description of scaling limit.} We now describe the scaling limit that will arise in Theorem~\ref{thm_main}. It is a Markov jump process~$(Z_t)_{t \ge 0}$ on~$(0,\infty)$ with dynamics given as follows. When the process is in state~$\lambda > 0$, it stays there for a time that is exponentially distributed with parameter~$b(\lambda)R(\lambda)$, where~$R$ is defined in~\eqref{eq_def_of_R}. It then \textit{attempts} to jump to a prospective location~$\lambda'$ chosen from the kernel~$K(\lambda,\cdot)$. Recall that we assume that~$K(\lambda,\{\lambda\}) = 0$, so~$\lambda' \neq \lambda$.  The jump is accepted with probability
\begin{equation}\label{eq_acceptance_prob} 
\mathds{1}\{\lambda' > \lambda\}\cdot \int_{\{0,1\}^{\Z^d}} \hat{\mu}_\lambda(\mathrm{d}\zeta) \;S(\lambda,\lambda';\{x:\zeta(x) = 1\}), 
\end{equation}
where~$\mathds{1}$ denotes the indicator function; the jump is rejected otherwise (so that the process stays in~$\lambda$ in that case). The formal generator of this process is given by
			\begin{equation}	Lf(\lambda) = b(\lambda)R(\lambda)\int_\lambda^\infty K(\lambda,\mathrm{d}\lambda')\;\int_{\{0,1\}^{\Z^d}} \hat{\mu}_\lambda(\mathrm{d}\zeta) \;  S(\lambda,\lambda';\{x:\zeta(x) = 1\})\cdot (f(\lambda')-f(\lambda))
			\end{equation}
(this is only a formal expression given for the reader's convenience; we do not prove the existence of a generator, as we will not need it).

	In Theorem~\ref{thm_main}, we assume that~$(Z_t)_{t \ge 0}$ is non-explosive, meaning that for any~$t > 0$, it only jumps finitely many times before time~$t$, with probability one. We comment on this assumption in Remark~\ref{rem_explosive_b} below.
}\smallskip

We are finally ready to state our main result.
\begin{theorem} \label{thm_main} Assume that~\eqref{eq_assumption_delta} { and~\eqref{eq_assumption_delta_2} hold and the Markov jump process~$(Z_t)_{t \ge 0}$ is non-explosive. Fix~$\lambda_0 > \lambda_c(\Z^d)$ and assume that for each~$N$, the adaptive contact process~$(X_{N,t})_{t \ge 0}$ is started from~$X_{N,0} \equiv \lambda_0$ and~$(Z_t)_{t \ge 0}$ is started from~$Z_0 = \lambda_0$.  Then,
	\begin{itemize}
		\item[(a)] for any~$t > 0$ we have
			\[\int_0^t \mathds{1}\{\Phi(\tilde{Z}_{N,s}) = \Asterisk\}\;\mathrm{d}s \xrightarrow{N \to \infty} 0 \quad \text{in probability};\]
\item[(b)] as~$N \to \infty$,~$(Z_{N,t})_{t \ge 0}$ converges in distribution  to~$(Z_t)_{t \ge 0}$ with respect to the Skorokhod topology.
\end{itemize}
}
\end{theorem}
	
\begin{remark}[Monotonicity of limiting process] It is worth emphasizing that we do not impose that the measure~$K(\lambda,\cdot)$ is supported on~$(\lambda,\infty)$, so the adaptive contact process may have mutants with lower birth rate than their parents. However, this possibility vanishes in the scaling limit, because mutant populations born in these circumstances are increasingly likely to lose the battle against the host population, becoming extinct in an amount of time that is negligible on the time scale we consider.
\end{remark}

	To give an intuitive explanation for { the dynamics of the limiting process~$(Z_t)_{t \ge 0}$, we observe that:
		 \begin{itemize}
			\item in the microscopic dynamics, an \textit{attempt to jump} from~$\lambda > 0$  corresponds to the appearance of a mutant. The value of the rate~$R(\lambda)b(\lambda)$ with which this happens is justified as follows: each unit time interval of~$(Z_t)$ corresponds to a time interval of length~$(\delta_NN^d)^{-1}$ for~$(X_{N,t})$; in such an interval, the number of births is of order~$(\delta_NN^d)^{-1}$ times~$N^d$ (volume of the torus) times~$R(\lambda)$ (effective birth rate) -- so all in all,~$\delta_N^{-1} R(\lambda)$ -- and each of such births results in a mutation with probability~$\delta_Nb(\lambda)$;
			\item the probability~\eqref{eq_acceptance_prob} of accepting a jump in the limiting dynamics corresponds to the probability of fixation of the mutant in the microscopic dynamics. The idea is that when the mutant is born, it faces a landscape of type-$\lambda$ enemies whose distribution looks like~$\hat{\mu}_\lambda$ (in a local sense). The probability that the mutant wins the competition against these enemies (fixation) is close to the probability that, in a two-type contact process on~$\Z^d$, an individual of type~$\lambda'$ initially occupying the origin manages to survive against a population of type~$\lambda$ initially occupying a set sampled from~$\hat{\mu}_\lambda$.
		\end{itemize}
Summing up, the transitions of the limiting process are given by the mutation kernel biased by the success probability of the mutants. Fitter mutants are more likely to succeed in invading the whole population, and mutants that do not succeed are not visible in the limiting process. The probability of success depends of the environment around the mutant and is increasing as a function of the mutant's birth rate. 
}

We now make a few more remarks about the statement of Theorem~\ref{thm_main}.
\begin{remark}[Non-fixation of weak mutants]
	As described above, when~$\lambda' > \lambda$, we use the values~$S(\lambda,\lambda';\cdot)$ pertaining to the two-type contact process on~$\Z^d$ to approximate the fixation probability on the torus. However, in case~$\lambda' < \lambda$, we have no knowledge about the values~$S(\lambda,\lambda';\cdot)$, and are able to draw the conclusion that fixation is unlikely by other means, which do not require this knowledge.
\end{remark}

\begin{remark}[Positivity of mutation rates]
	Although we assume that our mutation rate function is strictly positive, this is only done for convenience, and to make some of our arguments shorter. We could extend the statement of Theorem~\ref{thm_main} to~$b$ that is allowed to take the value zero, and this would correspond to having absorbing states in the limiting dynamics.
\end{remark}

{
\begin{remark}[Choice of torus]
	We believe that the statement of Theorem~\ref{thm_main} would hold if we had defined the adaptive contact process in a \textit{box} of~$\Z^d$ rather than in a torus. However, some aspects of our analysis  would be significantly complicated by the presence of a boundary (this is especially the case for the analysis of fixation). Since these complications would not add much to the interest of the model, we have chosen to only consider the process in the torus.
\end{remark}

\begin{remark}[Assumption~\eqref{eq_assumption_delta_2}]
	\label{rem_assumption2}
	We have not aimed for sharpness in assumption~\eqref{eq_assumption_delta_2}. The (one-type) contact process with~$\lambda > \lambda_c(\Z^d)$ is known to survive in a box of side length~$N$ (and hence also in a torus, by monotonicity), for at least as long as~$\exp\{c_\lambda N^d\}$, see~\cite{M93}. For this reason, our results should hold under the weaker assumption that~$\delta_N$ tends to zero slower than~$\exp\{-cN^d\}$ for any~$c > 0$. However, in some of our arguments (notably in Section~\ref{s_appendix}), it is convenient for technical reasons to only consider the process up to times that are much smaller than exponential in~$N^d$, so assumption~\eqref{eq_assumption_delta_2} is a straightforward way to ensure this.
\end{remark}

\begin{remark}[Non-explosiveness of limiting process]
	\label{rem_explosive_b}
	Recall that Proposition~\ref{prop_doesnt_explode} states that the adaptive contact process is never explosive (with no assumptions needed on~$\delta_N$,~$b$ or~$K$). Still, it could happen that due to the time rescaling, the number of jumps of the process~$(Z_{N,t})_{t \ge 0}$ in a finite time interval increased to infinity with~$N$ (in a stochastic sense), which is why we need to assume in Theorem~\ref{thm_main} that~$(Z_t)_{t \ge 0}$ is non-explosive.  This assumption  is not too restrictive. Recall that~$(Z_t)$ jumps from a state~$\lambda$ with rate~$b(\lambda)R(\lambda)$. We prove in Section~\ref{s_appearance} that~$R(\lambda) \le 1$ (see Proposition~\ref{prop_leo}), so for instance if~$b$ is bounded, non-explosiveness is guaranteed. Using standard criteria for non-explosiveness of continuous-time Markov chains (as in~\cite[Chapter 2]{norris}), it is also easy to prove that~$(Z_t)$ is non-explosive if~$b(\lambda)\le \lambda$ and there exists~$c > 0$ such that~$K(\lambda,\cdot)$ is supported on~$[0,\lambda+c]$ for all~$\lambda$.
\end{remark}

}

\subsection{Discussion}\label{sec:discussion}
The adaptation of a population to its environment via consecutive invasions of beneficial mutations has been classically studied within the framework of population genetics, population dynamics and adaptive dynamics. %A successful invasion will typically either end with the fixation of mutants and extinction of the former residents, or with a coexistence between residents and mutants. 
In the stochastic but non-spatial case, proof techniques are often based on coupling mutant populations with branching processes as long as their sizes are small and approximating the entire population by suitably chosen deterministic dynamics during phases when all sub-populations are macroscopic.

In the context of adaptive dynamics, a prototypical example of this approach was provided by the seminal paper~\cite{C06} of Champagnat. In his paper, the underlying deterministic dynamics was described by competitive Lotka--Volterra systems where \color{black} populations of different traits are not able to coexist, and hence a successful invasion implies fixation, i.e., elimination of the former resident population. His assumptions on the mutation rate have a similar effect as our assumptions on $\delta_N$ in this paper: Mutations are rare enough so that invasions of different mutant strains do not interfere with each other, but frequent enough so that mutants occur before the metastable population dies out. Among such conditions, after rescaling time suitably and dealing with the negligible time periods when there are multiple traits present in the population, the trait of the resident population converges to a pure jump Markov process. This process, called the \textit{trait substitution  sequence}, originates from earlier, deterministic works in adaptive dynamics (starting from~\cite{DL96} and~\cite{M96}; for further references see~\cite[Section 1]{C06}).  His methods have been adapted to various settings since then, including competitive Lotka--Volterra systems with coexistences between different traits~\cite{CM11}, later also extended to the case of phenotypic plasticity~\cite{BB18}, but also to multiple models with different biological motivation, for instance predator-prey systems~\cite{ManonCostaetal}. \color{black}

%It is natural to ask whether one can find an analogy of these results in the setting of interacting particle systems, especially in the case of 
%The contact process is one of the most classical models of the spatial spread of an epidemic.
In the context of epidemiology, it has been observed that mutations of pathogens can cause successive selective sweeps, in which mutants with higher reproduction rates invade the population and wipe out previous strains. Our adaptive contact process is a straightforward 
generalization of the classical contact process in this direction; in our model the reproduction rate is affected by recurrent mutations. On an intuitive level, Theorem~\ref{thm_main} is an analogue of the main results of~\cite{C06}. Indeed, in our model, a coexistence of different strains is not possible either, and the limiting Markov jump process $(Z_t)_{t\geq 0}$  can be seen as an analogue of the trait substitution sequence. 

Despite these similarities in biological interpretation, our model is mathematically rather different from the usual settings of stochastic adaptive dynamics. The difference is not only that due to spatial constraints, mutant populations tend to spread linearly rather than exponentially and thus a branching process approximation is not applicable. In fact, in Theorem~\ref{thm_main}, geometry plays a prominent role: The success probability of a mutant depends on the local configuration around the position where the mutation occurred. 

%In our setting, geometry plays a prominent role. Indeed, regarding the first stage, while the time until the appearance of the next mutant is asymptotically independent of the local configuration around this mutant at its birth time (cf.~Section~\ref{s_appearance} below), this local configuration is affected by the number of infected neighbours of the site where the mutant is born (cf.~Theorem~\ref{thm_main}). Further, this local configuration affects the probability that the mutant defeats the former resident.

For future work, various directions of inquiry could be pursued; let us provide a few examples. 
\begin{itemize}
\item[(a)] It would be interesting to study the model with a choice of sequence~$(\delta_N)_{N \ge 1}$ that tends to zero slower than what our assumption requires, so that there is a non-vanishing chance that more than two types appear at a given time.
\item[(b)] The dynamics of the adaptive contact process mimics that of Neuhauser's process~\cite{N92}, in the sense that individuals cannot be born on occupied sites. Neuhauser's process has been studied in infinite volume, and one could also study the behaviour of the adaptive contact process in infinite volume.
\item[(c)] { Still pertaining to the previous comment, one could consider models in which there is an alternate, and perhaps richer, description of the behavior of mutants. For instance, different types could be described not only by their birth rates, but also by their range, and possibly invasion probabilities (that is, they could give birth on top of existing individuals). A competition model based on the contact process that incorporates invasions is the \textit{grass-bushes-trees} process~\cite{DM91,DS91}.}
\item[(d)] In the case of the adaptive contact process on the torus $\Z^d_N$, it would be interesting to investigate how an interaction range growing with $N$ could affect the limiting dynamics. 
\end{itemize}
%In this case, instead of the analogue of a trait substitution sequence, we expect an analogue of a \emph{polymorphic evolution sequence} (cf.~\cite{CM11} in the setting of adaptive dynamics) to be the scaling limit, and one could possibly observe a phenomenon reminding of \emph{evolutionary branching}, where coexisting sub-populations start evolving in different directions. 

It is also worth mentioning that, as is the case for the classical contact process, the adaptive contact process has a unique absorbing state: the empty configuration. On finite graphs, the contact process reaches this configuration almost surely. In contrast, the adaptive contact process could have positive probability of avoiding the absorbing state forever even in a finite graph, for certain choices of the parameters (say, with the type increasing dramatically with each mutation). It is easy to give sufficient conditions for absorption:  for instance, if the discrete-time Markov chain with transition probability given by~$K$ remains bounded almost surely, then the adaptive contact process on~$\Z^d_N$ reaches the empty configuration almost surely. Giving a sharp condition seems beyond reach.

{
	Finally, let us say a few words regarding extensions to other graphs. In the present work, we use two main inputs from the existing literature. First,  the fact that the one-type contact process with sufficiently large~$\lambda$ survives for a long time in the torus (``metastability''), see for instance~\cite{M93}. Second, the fact that in the two-type contact process on~$\Z^d$, the stronger type can survive (provided that it is strong enough to survive in the absence of competition); this is proved in~\cite{MPV20}. To extend our results to other sequences of finite graphs (say, converging locally to some infinite graph), one would also need to have metastability for the one-type process along the sequence, and an analysis of the two-type process on the limiting graph. We expect that for instance, other Euclidean lattices should not pose significant problems in these respects. Another setting in which the analysis could be interesting is locally tree-like graphs, for instance random~$d$-regular graphs. For these, there are existing metastability results~\cite{CD21, LS17,MV16}, as well as an analysis of the two-type process in the limiting graph~\cite{CS09}.
}

\subsection{Organization of paper}
The rest of this paper is organized as follows. { In the end of this Introduction, we prove Proposition~\ref{prop_doesnt_explode}.} Section~\ref{s_convergence} contains the skeleton of the proof our main result, Theorem~\ref{thm_main}. This proof is based on the investigation of the rounds of mutation, each consisting of a first stage lasting until the appearance of the next mutant and a second stage of deterministic duration, by the end of which the mutant typically either fixes or goes extinct (while the previous resident state stays in a metastable state). % In more detail, Section~\ref{ss_fake} contains some preliminary definitions, in Section~\ref{ss_round} we analyze one round, and in Section~\ref{ss_multiple} we consider multiple rounds in order to conclude the proof of Theorem~\ref{thm_main}. 
The remaining sections are devoted to the proof of certain auxiliary results appearing in Section~\ref{s_convergence}. In particular, Section~\ref{s_preliminary} contains prerequisites about the classical and the two-type contact process. In Section~\ref{s_appearance}, we analyze the appearance of a mutant, i.e., the first stage of a round. Finally, in Section~\ref{s_fixation} we investigate the second stage of a round, i.e., the fixation (or extinction) of a mutant. Some technical results included in these sections are postponed until the Appendix.

\subsection{Notation}
Generic vertices of~$\Z^d$ will typically be denoted by the letters~$x,y$ and vertices of~$\Z^d_N$ by the letters~$u,v$. Given~$u \in \Z^d_N$, define the shift~$\theta_u:\Z^d_N \to \Z^d_N$ by~$\theta_u(v) := v-u$, where the difference is taken $\mathrm{mod}\; N$. For each~$N$, define a bijection between a subset of~$\Z^d$ (namely a box around the origin) \color{black} and~$\Z_N^d$ by letting
\begin{equation}\label{eq_def_of_psi}
	\begin{array}{ccc}\psi_N(x) = [x]_N, &\text{with}&\psi_N: \{-\tfrac{N}{2},\ldots, o,\ldots, \tfrac{N}{2}-1\}^d \to \Z_N^d \quad \text{if $N$ is even}\\[.2cm]
	&\text{and}&\psi_N: \{-\lfloor \tfrac{N}{2}\rfloor,\ldots, o,\ldots, \lfloor \tfrac{N}{2} \rfloor\}^d \to \Z_N^d\quad \text{ if $N$ is odd}.\end{array}
\end{equation}

For~$x \in \Z^d$ and~$r \ge 0$, define
\[Q_{\Z^d}(x,r):= \{y \in \Z^d:\; \|x-y\|_\infty \le r\};\]
we call this set the   box in~$\Z^d$ with center~$x$ and radius~$r$. For~$u \in \Z^d_N$ and~$0 \le r < N/2$, define
\[Q_{\Z^d_N}(u,r):= \{[y]_N:\;y \in Q_{\Z^d}(x,r)\},\quad \text{with } x \in \Z^d,\; [x]_N = u,\]
the box in~$\Z^d_N$ with center~$u$ and radius~$r$. 
We may omit the subscript~$\Z^d$ from~$Q_{\Z^d}$ or the subscript~$\Z^d_N$ from~$Q_{\Z^d_N}$ when they are clear from the context or unimportant. Additionally, although~$Q_{\Z^d}(x,r)$ (resp.~$Q_{\Z^d_N}(u,r)$) is a set of vertices of~$\Z^d$ (resp.~$\Z^d_N$), we will also  sometimes treat it as the subgraph of~$\Z^d$ (resp.~$\Z^d_N$) induced by this set of vertices.

{
\subsection{Non-explosiveness of the adaptive contact process}
We now prove Proposition~\ref{prop_doesnt_explode}.
\begin{proof}[Proof of Proposition~\ref{prop_doesnt_explode}]
	Fix~$N \in \N$ and~$\delta_N$,~$b$ and~$K$ as in the definition of the adaptive contact process. We will denote the identically-zero element of~$[0,\infty)^{\Z^d_N}$ by~$\underline{0}$ (this is the unique absorbing state for the process).

	We start by constructing the ``skeleton chain'', that is, the discrete-time Markov chain that represents the process at jump times (taking some extra care to deal with the absorbing state). For each~$\chi \in [0,\infty)^{\Z^d_N}$ with~$\chi \neq \underline{0}$, let~$P(\chi,\cdot)$ denote the distribution of the adaptive contact process (immediately) after the first jump, when the state at time~$0$ is~$\chi$. Also let~$P(\underline{0},\cdot) = \delta_{\underline{0}}(\cdot)$, that is, the measure in~$[0,\infty)^{\Z^d_N}$ given by a unit mass on~$\underline{0}$. Then, we fix some initial state~$\chi_0$ and construct a discrete-time Markov chain~$(\mathscr{X}_n)_{n \in \N_0}$ on the (uncountable) state space~$[0,\infty)^{\Z^d_N}$ with~$\mathscr{X}_0 = \chi_0$ and transition function~$P(\cdot,\cdot)$. This is done in a probability space with probability measure~$\P$. Note that if this chain hits~$\underline{0}$, it stays there forever (this is an arbitrary choice that has no impact for the argument).

	Next, let~$r(\chi)$ denote the total rate with which the adaptive contact process jumps away from~$\chi$ (with~$r(\underline{0}) = 0$). We write~$r(\chi) = r_\mathrm{D}(\chi) + r_\mathrm{U}(\chi)$, with~$r_\mathrm{D}(\chi)$ being the rate with which a death occurs in~$\chi$, and~$r_\mathrm{U}(\chi)$ the rate with which a birth occurs in~$\chi$.
	A standard criterion for continuous-time Markov chains (see~\cite[Theorem 2.3.3]{norris}) is that non-explosiveness holds if and only if
	\begin{equation*}
		\P\left(\sum_{n=0}^\infty \frac{1}{r(\mathscr{X}_n)} < \infty\right) = 0
	\end{equation*}
	(where we interpret~$1/0 = \infty$).
	Noting that~$r_\mathrm{D}(\chi)$ equals the number of occupied vertices in~$\chi$, which is at most~$N^d$, and using~$r(\underline{0}) = 0$, we have
	\begin{equation}\label{eq_aux_events_exp}
\left\{\sum_{n=0}^\infty \frac{1}{r(\mathscr{X}_n)} < \infty\right\} = \left\{ \mathscr{X}_n \neq \underline{0} \text{ for all }n,\; \sum_{n=0}^\infty \frac{r_\mathrm{D}(\mathscr{X}_n)}{r(\mathscr{X}_n)} < \infty\right\}.
	\end{equation}
	Let~$D_n$ denote the number of~$k \in \{0,\ldots,n-1\}$ such that the jump from~$\mathscr{X}_k$ to~$\mathscr{X}_{k+1}$ is a death (so~$D_n$ is the number of ``downward steps'' up to time~$n$). Note that, since the process takes place in a finite graph, we have
	\[\{\mathscr{X}_n \neq \underline{0} \text{ for all }n\}= \{\lim_{n\to \infty}D_n = \infty\},\]
	so the right-hand side of~\eqref{eq_aux_events_exp} equals
	\[\left\{ \lim_{n \to \infty} D_n = \infty,\;\sum_{n=0}^\infty \frac{r_\mathrm{D}(\mathscr{X}_n)}{r(\mathscr{X}_n)} < \infty\right\} \subseteq \left\{ \lim_{n \to \infty} \left(D_n - \sum_{k=0}^n \frac{r_\mathrm{D}(\mathscr{X}_k)}{r(\mathscr{X}_k)}\right) = \infty\right\}.\]
	Moreover, for~$\chi \neq \underline{0}$,~$\frac{r_\mathrm{D}(\chi)}{r(\chi)}$ is the probability that the first jump from~$\chi$ is a death; using this, it is easy to check that the process
	\[ \mathds{1}\{\mathscr{X}_n \neq \underline{0}\}\cdot \left( D_n - \sum_{k=0}^{n-1} \frac{r_\mathrm{D}(\mathscr{X}_k)}{r(\mathscr{X}_k)}\right),\quad n \ge 0\]
	is a martingale with bounded increments. By~\cite[Chapter 4, Theorem 3.1]{D19}, a martingale with bounded increments cannot tend to infinity. Putting things together, we have proved that
	\[\P\left(\sum_{n=0}^\infty \frac{1}{r(\mathscr{X}_n)} < \infty\right) \le \P \left(\lim_{n \to \infty} \left(D_n - \sum_{k=0}^n \frac{r_\mathrm{D}(\mathscr{X}_k)}{r(\mathscr{X}_k)}\right) = \infty\right) = 0.\]
\end{proof}
}

\section{Convergence to Markov jump process}\label{s_convergence}

\subsection{Description of limiting Markov jump process}\label{ss_fake}

{
In this section, we will make some observations and introduce some notation concerning the limiting process~$(Z_t)_{t \ge 0}$ defined before the statement of Theorem~\ref{thm_main}. 

For any~$\lambda > \lambda_c(\Z^d)$ and~$\lambda' \neq \lambda$, we introduce the notation~$\mathscr{S}(\lambda,\lambda')$ for the acceptance probability of a jump from~$\lambda$ to~$\lambda'$, given in~\eqref{eq_acceptance_prob}:
\begin{equation}\label{eq_def_acceptance_probability}
	\mathscr{S}(\lambda,\lambda'):= \mathds{1}\{\lambda' > \lambda\}\cdot \int_{\{0,1\}^{\Z^d}} \hat{\mu}_\lambda(\mathrm{d}\zeta)\; S(\lambda,\lambda';\{x:\zeta(x) = 1\}).
\end{equation}
We also define
\begin{equation} \label{eq_def_rejection_probability}
\mathscr{S}(\lambda,\lambda):= \int_{\{\lambda' \neq \lambda\}} (1-\mathscr{S}(\lambda,\lambda'))\; K(\lambda,\mathrm{d}\lambda'),
\end{equation}
that is,~$\mathscr{S}(\lambda,\lambda)$ accumulates the mass of rejection probability of jumping from~$\lambda$ to~$\lambda'$, for~$\lambda'$ sampled from the kernel~$K(\lambda,\cdot)$ (recall that we assume that~$K(\lambda,\{\lambda\}) = 0$).
}

Then define the new probability kernel~$\mathscr{K}$ by
\[{\mathscr{K}}(\lambda,I):= \int_{I\cap (\lambda,\infty)} \mathscr{S}(\lambda,\lambda')\;K(\lambda,\mathrm{d}\lambda')+\mathscr{S}(\lambda,\lambda)\cdot \delta_\lambda(I), \quad I \subseteq (0,\infty) \text{ Borel}.\]
Defining the ``diagonal'' kernel~$D(\lambda,I):=\delta_\lambda(I)$, we can rewrite
\begin{equation}\label{eq_alt_alt_ker}
	{\mathscr{K}}(\lambda,I) = \int_{I} \mathscr{S}(\lambda,\lambda')\;(K+D)(\lambda,\mathrm{d}\lambda'),\quad I \subseteq (0,\infty) \text{ Borel}.
\end{equation}
{
With these definitions, we can give a new (and equivalent) description of~$(Z_t)_{t \ge 0}$, by saying that it jumps from~$\lambda \in (0,\infty)$ with rate~$b(\lambda)R(\lambda)$, to a state~$\lambda'$ chosen from~$\mathscr{K}(\lambda,\cdot)$. This description does not include the acceptance/rejection step, since it is incorporated in the kernel~$\mathscr{K}$. The corresponding (formal) generator is
\begin{equation*}
	Lf(\lambda)=b(\lambda)R(\lambda)\int_0^\infty \mathscr{K}(\lambda,\mathrm{d}\lambda')\;(f(\lambda') -f(\lambda)).
\end{equation*}
}

We let~$\Lambda_0 = \lambda_0$,~$\Lambda_1$,~$\Lambda_2,\ldots$ denote {the locations of~$(Z_t)_{t \ge 0}$ after each attempted jump (so~$\Lambda_i = \Lambda_{i+1}$ in case of a failed jump)}, and let~$\sigma_1,\sigma_2,\ldots$ the holding times of this process, so that we can write
%\[\sigma_1:= \inf\{t \ge 0:Z_t \ne Z_0\},\qquad \sigma_{k+1}:= \inf\left\{t- \sum_{j=1}^k \sigma_j: Z_t \neq Z_{\sum_{j=1}^k \sigma_j}\right\}.\]
\begin{equation}\label{eq_Z_fromjumps}
	Z_t = \Lambda_0\color{black}\cdot \mathds{1}_{[0,\sigma_1)}(t) +  \sum_{k=1}^\infty \Lambda_k \cdot \mathds{1}_{[\sum_{j=1}^k \sigma_j,\;\sum_{j=1}^{k+1} \sigma_j)}(t).
\end{equation}
We note that the distribution of~$(Z_t)_{t \ge 0}$ is characterized by the fact that, for any~$t_1,\ldots,t_k > 0$ and measurable  bounded functions~$g_1,\ldots,g_k:(0,\infty) \to \R$, we have
\begin{equation}\label{eq_original_K}\begin{split}
	\E\left[ \prod_{j=1}^k \mathds{1}\{\sigma_j \le t_j\} \cdot g_j(\Lambda_j)\right] &= \int_{\lambda_0}^\infty \mathscr{K}(\lambda_0,\mathrm{d}\lambda_1) \cdots \int_{\lambda_{k-1}}^\infty \mathscr{K}(\lambda_{k-1},\mathrm{d}\lambda_k)\cdot \\[-.3cm]
&\hspace{1.6cm}\prod_{j=1}^k (1- \mathrm{e}^{-b(\lambda_{j-1})R(\lambda_{j-1})\cdot t_j})\cdot g_j(\lambda_j).\end{split}
\end{equation}
In this expectation the likelihood of the sequence of states visited is pondered by the probability kernels, and the product of exponentials contains the information of the holding times.

\subsection{Round of mutation and competition}\label{ss_round}
 Throughout the rest of Section~\ref{s_convergence}, we will assume that the three assumptions of Theorem~\ref{thm_main} hold. For each~$N$, we take a probability space where the adaptive contact process is defined. It can be constructed with the standard method for continuous-time Markov chains, starting with the discrete-time ``skeleton chain'' and augmenting it with the holding times (Proposition~\ref{prop_doesnt_explode} guarantees no explosion). This construction makes it simple to see that the process has the strong Markov property, which will be exploited in some of our proofs. \color{black}
\subsubsection{Definitions and description of a round}

We now introduce some definitions that will be used throughout the paper.  First, we let
\[\mathscr{G}_N := \{A \subseteq \Z^d_N:\; A \text{ intersects all boxes of radius $N^{1/288}$ in $\Z^d_N$}\}.\]
It will be important for us to argue that the set of occupied vertices in the adaptive contact process at certain moments in time belongs to this class, with high probability. This will allow us to appeal to mixing properties of the process. { For more on the constant~$1/288$, and our overall treatment of constants throughout the paper, see Remark~\ref{rem_constants} below.}

Second, letting~$(X_{N,t})_{t \ge 0}$ denote the adaptive contact process on~$\Z^d_N$ started from an arbitrary, but fixed, \color{black} initial configuration, let
\begin{equation*}T_N := \inf\{ t\ge 0: \text{there is a birth with mutation at time $t$}\}.\end{equation*} We will typically take the rescaling~$\delta_NN^d\cdot T_N$, which is the same as the time scaling seen in Theorem~\ref{thm_main}.

Third, we define
\begin{equation} \label{eq_def_of_tn}\tn := N^{1+\frac{\varepsilon_0}{2}},\end{equation}
{ where~$\varepsilon_0$ is the constant in~\eqref{eq_assumption_delta}.}
This will serve as an amount of time that is long enough for the resolution of a competition between two types, but too short for the appearance of a new mutant. The latter part is guaranteed by the following.

\begin{lemma}\label{lem_bound_Ttn}
	For any~$\varepsilon > 0$ and any finite set~$L \subseteq [0,\infty)$, the following holds for~$N$ large enough. If the adaptive contact process~$(X_{N,t})_{t \ge 0}$ on~$\Z^d_N$ is started from a configuration satisfying~$X_{N,0}(u) \in L$ for all~$u$, then~$\P(T_N \le \tn) < \varepsilon$.
\end{lemma}
\begin{proof}
	 For each~$\chi:\Z^d_N \to [0,\infty)$, let~$r(\chi)$ denote the rate with which a birth with mutation occurs from~$\chi$. We note that the process
	\[ M_t :=\left(\mathds{1}\{T_N \le t\} - \int_0^{t \wedge T_N} r(X_{N,s})\;\mathrm{d}s\right)_{t \ge 0}\] is a martingale and that, assuming that the process has~$X_{N,0}(u) \in L$ for all~$u$,
	\[T_N > t \quad \Longrightarrow \quad r(X_{N,t}) \le N^d \cdot 2d \cdot  \delta_N \cdot  \max_{\lambda \in L} (\lambda \cdot b(\lambda)).\]
	Using the fact that~$\E[M_{\tn}] = \E[M_0] = 0$, we obtain	
	\begin{align*}
		\P(T_N \le \tn) = \E \left[ \int_0^{\tn \wedge T_N} r(X_{N,s})\;\mathrm{d}s\right] 
		&\le \tn \cdot N^d \cdot 2d \cdot  \delta_N \cdot  \max_{\lambda \in L}(\lambda\cdot b(\lambda)) \\
		&= N^{d+1+\frac{\varepsilon_0}{2}} \cdot 2d \cdot  \delta_N \cdot  \max_{\lambda \in L}(\lambda\cdot b(\lambda)).
	\end{align*}
	By~\eqref{eq_assumption_delta}, the right-hand side tends to zero as~$N \to \infty$.
	\end{proof}

	{
\begin{remark}[A word about constants]
	\label{rem_constants}
	Throughout the paper, there are numerous partial results and definitions in which a sufficiently large or sufficiently small constant is needed. To avoid using too much notation (or using the same letter for different constants), we have opted to use numerical values, such as~$1/288$ in the definition of~$\mathscr{G}_N$ above. We do not aim for sharpness with these constants; for instance, when we need a power of~$N$ asymptotically larger than~$N^2$, we may take~$N^3$. Hence, the value $1/288$ is a result of a compound usage of this practice (the last step of which is in Section~\ref{s_fixation}, where a result concerning sub-boxes of radius~$r^{1/24}$ of a box of radius~$r$ is applied with~$r = N^{1/12}$).
\end{remark}
}

Assume that the adaptive contact process is started from a configuration in which a single type  is present (together, possibly, with empty sites). In case~$T_N < \infty$, we refer to the time interval~$[0,T_N+\tn]$ as a \textit{round}, to~$[0,T_N]$ as \textit{stage~1} of the round, and to~$[T_N,T_N + \tn]$ as \textit{stage 2} of the round. Note that stage 2 starts with a configuration with exactly two types, one of which is represented by exactly one (newborn) individual. For the time being, we will focus on a single round, but later, in Section~\ref{ss_multiple}, we will also consider multiple rounds in succession.

\subsubsection{Stage 1: Appearance of a mutant}
For this discussion of stage 1, we focus on the adaptive contact process started from configurations of the form~$X_{N,0} = \lambda \cdot \mathds{1}_A$, where~$\lambda > 0$ and~$A \in \mathscr{G}_N$.

On~$\{T_N < \infty\}$, define~$\mathcal{X}_N$ as the position of the newborn mutant at time~$T_N$ (on~$\{T_N = \infty\}$, we can define~$\mathcal{X}_N$ arbitrarily, for example equal to~$o$).  Note that~$X_{N,T_N} \circ \theta_{\mathcal{X}_N}$ is the configuration at time~$T_N$, shifted so that the position of the mutant becomes the origin of~$\Z^d_N$. We define the occupied landscape around the newborn mutant as
\begin{equation*}
	\mathscr{L}_N:= \{v \in \Z^d_N \backslash \{o\}: (X_{N,T_N} \circ \theta_{\mathcal{X}_N})(v) \neq 0\}
\end{equation*}
(on $\{T_N = \infty\}$, we can define~$\mathscr{L}_N$ arbitrarily, for instance equal to the empty configuration).

For~$\lambda > \lambda_c(\Z^d)$,~$t > 0$ and~$A,B \subseteq \Z^d_N$, define
\[\mathcal{U}_{N,t}(\lambda;A,B):= \P(\delta_N N^d \cdot T_N  \le t,\; \mathscr{L}_N = B \mid X_{N,0} = \lambda\cdot \mathds{1}_A).\]
In words,~$\mathcal{U}_{N,t}$ is the probability, for the adaptive contact process started from the configuration where only type~$\lambda$ is present, occupying the set~$A$, that a birth with mutation occurs before time~$(\delta_N N^d)^{-1} \cdot t$, and that the landscape around this newborn mutant is equal to the set~$B$. 

The following proposition states that, as~$N\to \infty$, the time until the appearance of the first mutant (rescaled by~$\delta_NN^d$) and the landscape around this mutant are asymptotically independent, with the rescaled time converging to an exponential distribution and the landscape converging to~$\hat{\mu}_\lambda$. Moreover, this convergence holds uniformly in the set~$A \in \mathscr{G}_N$ that defines the initial condition. Lastly, the configuration at time~$T_N$ is with high probability in the collection~$\mathscr{G}_N$.  Recall also the definition~\eqref{eq_def_of_psi} of the function $\psi_N$. \color{black} The proposition will be proved in Section~\ref{s_appearance}.
\begin{proposition}\label{prop_appearance_new}
	For any~$\lambda > \lambda_c(\Z^d)$,~$t > 0$ and any local function~$f: \{0,1\}^{\Z^d} \to \mathbb{R}$, we have
	\begin{align*}
		\sup_{A \in \mathscr{G}_N} \left| \sum_{B \in \mathscr{G}_N} \mathcal{U}_{N,t}(\lambda;A,B)\cdot f(\psi_N^{-1}(B)) - (1-\mathrm{e}^{-b(\lambda)R(\lambda)t})\cdot \int f\;\mathrm{d}\hat{\mu}_\lambda\right| \xrightarrow{N \to \infty} 0.
	\end{align*}
\end{proposition}

\subsubsection{Stage 2: Competition}
For this discussion of stage 2, we focus on the adaptive contact process started from configurations of the form~$X_{N,0} = \lambda \cdot \mathds{1}_B+ \lambda' \cdot \mathds{1}_{\{o\}}$, where~$\lambda, \lambda' > 0$ and~$B \in \mathscr{G}_N$ with~$o \notin B$.

Given~$\lambda > \lambda_c(\Z^d)$,~$\lambda' \neq \lambda$ and~$B,B' \subseteq \Z^d_N$ with~$o \notin B$, define
\[\mathcal{V}_N(\lambda,\lambda';B,B'):= \P(T_N > \tn,\; X_{N,\tn} =  \lambda' \cdot  \mathds{1}_{B'} \mid X_{N,0} = \lambda\cdot \mathds{1}_B + \lambda' \cdot \mathds{1}_{\{o\}}).\]
In words,~$\mathcal{V}_N(\lambda,\lambda';B,B')$ is the probability, for the process started from~$\lambda\cdot \mathds{1}_B + \lambda'\cdot \mathds{1}_{\{o\}}$, that by time~$\tn$ there is no mutation, and type~$\lambda$ has disappeared, so only type~$\lambda'$ remains, occupying the set~$B'$.  Also define, for the same collection of~$\lambda,\lambda',B,B'$,
\[\bar{\mathcal{V}}_N(\lambda,\lambda';B,B'):= \P(T_N > \tn,\; X_{N,\tn} =  \lambda \cdot  \mathds{1}_{B'} \mid X_{N,0} = \lambda\cdot \mathds{1}_B + \lambda' \cdot \mathds{1}_{\{o\}}).\]

In Section~\ref{s_fixation}, we will prove the following two propositions. Recall the definition  of~$S(\lambda,\lambda';A)$ above \eqref{eq:Sdef}.\color{black}
\begin{proposition}\label{prop_fixation_one_new}
	If~$\lambda' > \lambda> \lambda_c(\Z^d)$, then we have
	\[ \sup_{\substack{B \in \mathscr{G}_N\\ o \notin B}} \left| \sum_{B' \in \mathscr{G}_N}\mathcal{V}_N(\lambda,\lambda';B,B') - S(\lambda,\lambda';\psi^{-1}_N(B)) \right| \xrightarrow{N \to \infty} 0\]
	and
	\[ \sup_{\substack{B \in \mathscr{G}_N\\ o \notin B}} \left| \sum_{B' \in \mathscr{G}_N}\bar{\mathcal{V}}_N(\lambda,\lambda';B,B') - (1- S(\lambda,\lambda';\psi^{-1}_N(B))) \right| \xrightarrow{N \to \infty} 0.\]
\end{proposition}
\begin{remark}
	It may seem at first glance that one of the assertions in the above statement immediately implies the other, but that is not so, because the sum
	\[\sum_{B' \in \mathscr{G}_N}(\mathcal{V}_N(\lambda,\lambda';B,B') +\bar{\mathcal{V}}_N(\lambda,\lambda';B,B'))\]
	does not equal~$1$.  Indeed, for any fixed $N$, apart from the scenarios whose probabilities are captured by $\sum_{B' \in \mathscr{G}_N}\mathcal{V}_N(\lambda,\lambda';B,B')$ and~$\sum_{B' \in \mathscr{G}_N}\bar{\mathcal{V}}_N(\lambda,\lambda';B,B')$, many other things can happen with positive probability by time~$\tn$. For instance, both types~$\lambda$ and~$\lambda'$ could still be alive at time~$\tn$, or further mutants might have appeared etc. Of course, a consequence of the proposition is that the probability of all these other situations vanishes as~$N \to \infty$. \color{black}
\end{remark}

\begin{proposition}\label{prop_no_fixation}
	If~$\lambda> \lambda_c(\Z^d)$ and~$\lambda > \lambda'$, then we have
		\[ \inf_{\substack{B \in \mathscr{G}_N\\ o \notin B}}  \sum_{B' \in \mathscr{G}_N}\bar{\mathcal{V}}_N(\lambda,\lambda';B,B')  \xrightarrow{N \to \infty} 1.\]

\end{proposition}

The work to prove the above two propositions in Section~\ref{s_fixation} involves, for the most part, studying the two-type contact process on~$\Z^d$, boxes of~$\Z^d$ and~$\Z^d_N$. As a by-product of the analysis, we also obtain the following statement  in Section~\ref{s_fixation}\color{black}, which will also be needed and has independent interest:
\begin{proposition}\label{prop_fixation_three_new}
	If~$\lambda' > \lambda> \lambda_c(\Z^d)$, we have
	\[\lim_{\ell \to \infty} \sup_{A \subseteq \Z^d \backslash \{o\}} \left| S(\lambda,\lambda';A) - S(\lambda,\lambda';A \cap Q_{\Z^d}(o,\ell))\right| = 0.\]
\end{proposition}

\subsubsection{Convergence of round transition probabilities}

Define, for~$\lambda > \lambda_c(\Z^d)$,~$\lambda' \neq \lambda$ and~$A,B' \in \mathscr{G}_N$,
\begin{equation*}
	\mathscr{S}_{N,t}(\lambda,\lambda';A,B'):= \sum_{\substack{B \in \mathscr{G}_N\\ o \notin B}} \mathcal{U}_{N,t}(\lambda;A,B)\cdot \mathcal{V}_{N}(\lambda,\lambda';B,B').
\end{equation*}
Let us motivate this definition in words. Consider the adaptive contact process on~$\Z^d_N$ started from~$X_{N,0} = \lambda\cdot \mathds{1}_A$, and modified in the following way: when (and if) a mutant first appears, we \textit{artificially} replace the type of this mutant by~$\lambda'$, and then continue evolving the process (with no further artificial changes{{; in particular, further mutations can still happen, following the same rules of dynamics as in the adaptive contact process}). Denote this modified process by~$(Y_{N,t})_{t \ge 0}$ (this notation will only be employed for this heuristic explanation). Then,~$\mathscr{S}_{N,t}$ is the probability that~$(Y_{N,t})$ satisfies the following:
\begin{itemize}
	\item the time~$T_N$ is smaller than or equal to~$(\delta_N N^d)^{-1}\cdot t$, and every box of radius~$N^{1/288}$ in~$\Z^d_N$ has an individual of type~$\lambda$ at time~$T_N$;
	\item the competition started at time~$T_N$, between the population of type~$\lambda$ and the individual of type~$\lambda'$, is resolved by time~$T_N + \tn$, with the disappearance of type~$\lambda$, and type~$\lambda'$ occupying the set~$B'$.
\end{itemize}

Also define, for the same collection of~$\lambda,\lambda',A,B'$,
\[\bar{\mathscr{S}}_{N,t}(\lambda,\lambda';A,B'):= \sum_{\substack{B \in \mathscr{G}_N\\ o \notin B}} \mathcal{U}_{N,t}(\lambda;A,B)\cdot \bar{\mathcal{V}}_{N}(\lambda,\lambda';B,B').\]

{We now prove the following result using Propositions \ref{prop_appearance_new}, \ref{prop_fixation_one_new} and \ref{prop_no_fixation}. Recall the definition of~$\mathscr{S}(\lambda,\lambda')$ for~$\lambda \neq \lambda'$ in~\eqref{eq_def_acceptance_probability}.}
\begin{lemma}\label{lem_first_sup}
	If~$\lambda' > \lambda > \lambda_c(\Z^d)$ and~$t > 0$, then
	\[\sup_{A \in \mathscr{G}_N} \left| \sum_{B' \in \mathscr{G}_N}\mathscr{S}_{N,t}(\lambda,\lambda';A,B') - (1-\mathrm{e}^{-b(\lambda)R(\lambda)\cdot t})\cdot \mathscr{S}(\lambda,\lambda')\right| \xrightarrow{N \to \infty} 0\]
	and
	\[\sup_{A \in \mathscr{G}_N} \left| \sum_{B' \in \mathscr{G}_N}\bar{\mathscr{S}}_{N,t}(\lambda,\lambda';A,B') - (1-\mathrm{e}^{-b(\lambda)R(\lambda)\cdot t})\cdot (1-\mathscr{S}(\lambda,\lambda'))\right| \xrightarrow{N \to \infty} 0.\]
\end{lemma}
\begin{proof}
	Let us address the first convergence. Let~$\varepsilon > 0$. Fix~$N \in \N$, whose value will be assumed to be large enough along the proof. Also let~$A \in \mathscr{G}_N$. Using Proposition~\ref{prop_fixation_three_new}, we can find a local function~$f:\{0,1\}^{\Z^d}\to \mathbb{R}$ such that
	\begin{equation} \label{eq_bottomof11}|S(\lambda,\lambda';B) - f(B)| < \frac{\varepsilon}{4} \quad \text{for all } B \subseteq \Z^d,\; o \notin B. \end{equation}
	Additionally, using the first part of Proposition~\ref{prop_fixation_one_new}, if~$N$ is large we have
	\[\left| \sum_{B' \in \mathscr{G}_N} \mathcal{V}_N(\lambda,\lambda';B,B') - S(\lambda,\lambda';\psi^{-1}_N(B))\right| < \frac{\varepsilon}{4} \quad \text{for all } B \in \mathscr{G}_N,\; o \notin B.\]
	Noting that
	\[\sum_{B' \in \mathscr{G}_N} \mathscr{S}_{N,t}(\lambda,\lambda';A,B') = \sum_{\substack{B \in \mathscr{G}_N\\o \notin B}} \mathcal{U}_{N,t}(\lambda;A,B) \cdot \sum_{B' \in \mathscr{G}_N} \mathcal{V}_N(\lambda,\lambda';B,B'),\]
	and that~$\sum_{B \in \mathscr{G}_N,o \notin B} \mathcal{U}_{N,t}(\lambda;A,B) \le 1$, we then obtain
	\[\left| \sum_{B' \in \mathscr{G}_N} \mathscr{S}_{N,t}(\lambda,\lambda';A,B') - \sum_{\substack{B \in \mathscr{G}_N\\ o \notin B}} \mathcal{U}_{N,t}(\lambda;A,B)\cdot f(\psi^{-1}_N(B))\right| < \frac{\varepsilon}{4}+\frac{\varepsilon}{4}=  \frac{\varepsilon}{2}.\]
	Next, Proposition~\ref{prop_appearance_new} implies that, increasing~$N$ if necessary, we have
	\[\left| \sum_{\substack{B \in \mathscr{G}_N\\ o \notin B}} \mathcal{U}_{N,t}(\lambda;A,B)\cdot f(\psi^{-1}_N(B)) - (1-\mathrm{e}^{-b(\lambda)R(\lambda)t})\cdot \int f\;\mathrm{d}\hat{\mu}_\lambda \right| < \frac{\varepsilon}{4}.\]
	Finally, recalling that~$\mathscr{S}(\lambda,\lambda') = \int S(\lambda,\lambda';B)\;\hat{\mu}_\lambda(\mathrm{d}B)$ { and using~\eqref{eq_bottomof11},} we have
	\[\left| (1-\mathrm{e}^{-b(\lambda)R(\lambda)t})\cdot \int f\;\mathrm{d}\hat{\mu}_\lambda - (1-\mathrm{e}^{-b(\lambda)R(\lambda)t})\cdot \mathscr{S}(\lambda,\lambda') \right| < \frac{\varepsilon}{4}.\]
	This concludes the proof of the first convergence, and the second is proved in the same way, using the second part of Proposition~\ref{prop_fixation_one_new}.
\end{proof}

\begin{lemma}\label{lem_second_sup}
	If~$\lambda > \lambda_c(\Z^d)$,~$\lambda' < \lambda$ and~$t> 0$, then
	\[\sup_{A \in \mathscr{G}_N}\left|  \sum_{B' \in \mathscr{G}_N}\bar{\mathscr{S}}_{N,t}(\lambda,\lambda';A,B')   - (1-\mathrm{e}^{-b(\lambda)R(\lambda)\cdot t})\right|\xrightarrow{N \to \infty} 0.\]
\end{lemma}
\begin{proof}
	Let~$\varepsilon > 0$. Fix~$N$ (to be assumed large) and~$A \in \mathscr{G}_N$. Using the definition of~$\bar{\mathscr{S}}_{N,t}(\lambda,\lambda';A,B')$ and Proposition~\ref{prop_no_fixation}, if~$N$ is large we have
	\[\left| \bar{\mathscr{S}}_{N,t}(\lambda,\lambda';A,B') - \sum_{B \in \mathscr{G}_N} \mathcal{U}_{N,t}(\lambda;A,B)\right| < \frac{\varepsilon}{2}.\]
	Next, using Proposition~\ref{prop_appearance_new} with~$f\equiv 1$, if~$N$ is large we have
	\[\left| \sum_{B \in \mathscr{G}_N} \mathcal{U}_{N,t}(\lambda;A,B) - (1-\mathrm{e}^{-b(\lambda)R(\lambda)\cdot t})\right| < \frac{\varepsilon}{2}.\]
The desired result now follows from the triangle inequality.
\end{proof}

Next, by analogy with { the definition of~$\mathscr{S}(\lambda,\lambda)$ in~\eqref{eq_def_rejection_probability}},  for~$\lambda > \lambda_c(\Z^d)$ and~$A,B' \in \mathscr{G}_N$, define
\begin{equation}
	\label{eq_def_of_barS}
\mathscr{S}_{N,t}(\lambda,\lambda;A,B') := \int_{(0,\infty)\backslash \{\lambda\}} \bar{\mathscr{S}}_{N,t}(\lambda,\lambda';A,B')\;K(\lambda,\mathrm{d}\lambda').
\end{equation}
{
The idea is that~$\mathscr{S}_{N,t}(\lambda,\lambda;A,B')$ accumulates the probabilities of scenarios where a mutant appears (with some type sampled from~$K(\lambda,\cdot)$) but loses the battle against the~$\lambda$ population. 
}
\begin{lemma}\label{lem_new_diagonal}
	For any~$\lambda > \lambda_c(\Z^d)$ and~$t > 0$, we have
	\[\sup_{A \in \mathscr{G}_N} \left| \sum_{B' \in \mathscr{G}_N} \mathscr{S}_{N,t}(\lambda,\lambda;A,B') - (1-\mathrm{e}^{-b(\lambda)R(\lambda)\cdot t})\cdot \mathscr{S}(\lambda,\lambda) \right| \xrightarrow{N \to \infty} 0.\]
\end{lemma}
{
\begin{proof}
	Fix~$\lambda > \lambda_c(\Z^d)$ and~$A \in \mathscr{G}_N$.  Using the definition of~$\mathscr{S}(\lambda,\lambda)$ in~\eqref{eq_def_rejection_probability} and the definition of~$\mathscr{S}_{N,t}(\lambda,\lambda;A,B')$ in~\eqref{eq_def_of_barS}, we bound the absolute value that appears in the statement of the lemma by:
	\begin{align*}
		&\int_{(0,\lambda)} \left|\sum_{B' \in \mathscr{G}_N} \bar{\mathscr{S}}_{N,t}(\lambda,\lambda';A,B')-(1-\mathrm{e}^{-b(\lambda)R(\lambda)\cdot t})\right|\;K(\lambda,\mathrm{d}\lambda') \\
		&+\quad  \int_{(\lambda,\infty)} \left|\sum_{B' \in \mathscr{G}_N} \bar{\mathscr{S}}_{N,t}(\lambda,\lambda';A,B') -(1-\mathrm{e}^{-b(\lambda)R(\lambda)\cdot t})\cdot (1-{\mathscr{S}}(\lambda,\lambda')) \right| \;K(\lambda,\mathrm{d}\lambda').
	\end{align*}
	Now, the first integral on the right-hand side tends to zero as~$N \to \infty$ (uniformly in~$A$) by Lemma~\ref{lem_second_sup} and the dominated convergence theorem, and the second integral tends to zero as~$N \to \infty$ (uniformly in~$A$) by the second part of Lemma~\ref{lem_first_sup}  and the dominated convergence theorem.
\end{proof}
}

	Putting together the first part of Lemma~\ref{lem_first_sup} and Lemma~\ref{lem_new_diagonal}, we have that for any~$\lambda > \lambda_c(\Z^d)$ and any~$\lambda' \in [\lambda,\infty)$,
	\[\sup_{A \in \mathscr{G}_N} \left| \sum_{B' \in \mathscr{G}_N} \mathscr{S}_{N,t}(\lambda,\lambda';A,B') - (1-\mathrm{e}^{-b(\lambda)R(\lambda)\cdot t})\cdot \mathscr{S}(\lambda,\lambda') \right| \xrightarrow{N \to \infty} 0.\]
	The following lemma easily follows from this and induction.

\begin{lemma}\label{prop_multiple_new}
	For any~$k \in \N$,~$\lambda_c <\lambda_0\le \lambda_1 \le \cdots  \lambda_k$ and~$t_1,\ldots, t_k> 0 $, letting~$A_0 = \Z^d_N$, we have
	\begin{equation*}\begin{split}
	&\sum_{A_1,\ldots, A_k \in \mathscr{G}_N} \;\prod_{j=1}^k \mathscr{S}_{N,t_j}(\lambda_{j-1},\lambda_j;A_{j-1},A_j) \\ &\hspace{2cm}\xrightarrow{N \to \infty} \;\prod_{j=1}^k (1-\mathrm{e}^{-b(\lambda_{j-1})R(\lambda_{j-1})\cdot t_j})\cdot \mathscr{S}(\lambda_{j-1},\lambda_j). \end{split}
	\end{equation*}
\end{lemma}

Now define, for~$\lambda > \lambda_c(\Z^d)$,~$I \subseteq [\lambda,\infty)$ and~$A,A' \in \mathscr{G}_N$,
\begin{align*}
{\mathscr{K}}_{N,t}(\lambda,I;A,A')& :=\mathscr{S}_{N,t}(\lambda,\lambda;A,A')\cdot  \delta_\lambda(I) + \int_I \mathscr{S}_{N,t}(\lambda,\lambda';A,A')\;K(\lambda,\mathrm{d}\lambda') .\end{align*}
Letting~$D(\lambda,I) = \delta_\lambda(I)$ for all~$\lambda$ and~$I$, we can also write
\begin{equation}\label{eq_new_kd}
	{\mathscr{K}}_{N,t}(\lambda,I;A,A')= \int_I \mathscr{S}_{N,t}(\lambda,\lambda';A,A')\;(K+D)(\lambda,\mathrm{d}\lambda').
\end{equation}

\subsection{Multiple rounds: convergence of holding times and jump locations}\label{ss_multiple}

For the rest of this section, we consider the adaptive contact process~$(X_{N,t})_{t \ge 0}$ started from~$X_{N,0} \equiv \lambda_0$, where~$\lambda_0 > \lambda_c(\Z^d)$.  Define~$T_{N,1}:= T_N$, which as before denotes the time of the first birth with mutation, and recursively define
\[T_{N,k} := \inf\{t > T_{N,k-1}: \text{a birth with mutation occurs at time~$t$}\},\quad k \ge 2.\]
Now, for~$k \in \N$ let
\begin{align*}
&T_{N,k}':= \inf\{t > T_{N,k}: \text{only one type is present at time $t$}\},\\[.2cm]
&T_{N,k}'':= T_{N,k} + \tn. \end{align*}
We then define, for~$k\in \N$, the event
\begin{align*}E_{N,k}&:= \{T_{N,1} < T_{N,1}' < T_{N,1}'' < \cdots < T_{N,k} < T_{N,k}' < T_{N,k}''\}\\
	&\quad \cap \{\{u: X_{N,t}(u) \neq \varnothing\} \in \mathscr{G}_N \text{ for each } t \in \{T_{N,1}'',T_{N,2}'',\ldots,T_{N,k}''\}\}.
\end{align*}
 In words, this event means that for all $l \in \{1,\ldots,k\}$, after the appearance of the $l$-th mutant, within a time interval of length $\tn$ the conflict is resolved, i.e., either only the mutant or the previous resident type remains, and the remaining type is sufficiently densely populated in $\Z^d_N$ (characterized in terms of $\mathscr G_N$). \color{black}
Note that~$E_{N,1}\supseteq E_{N,2}\supseteq \ldots$. 
Next, let
\[\Lambda_{N,k} := \mathds{1}_{E_{N,k}}\cdot \Phi(X_{N,T_{N,k}'}),\quad k \in \N,\]
that is, on the event~$E_{N,k}$,~$\Lambda_{N,k}$ is the (unique) type that is present at time~$T_{N,k}'$; note that this is the type that wins the conflict in round~$k$. 
{ 
\begin{remark}
	\label{rem_weaker}
	There is nothing in the definition of~$E_{N,k}$ preventing the possibility that~$\Lambda_{N,l+1} < \Lambda_{N,l}$ for some~$l$ (that is, fixation of a mutant that is weaker than the population where it appears), but the probability of this situation is small when~$N$ is large and vanishes in the scaling limit, see~\eqref{eq_Enk_tends} below.
\end{remark}
}
Finally, define
\[\sigma_{N,1}:= \mathds{1}_{E_{N,1}}\cdot \delta_N N^d\cdot T_{N,1}'\]
and
\[\sigma_{N,k}:= \mathds{1}_{E_{N,k}}\cdot  \delta_N N^d\cdot (T_{N,k}'-T_{N,k-1}'),\quad k \ge 2.\]
 That is, on the event $E_{N,k}$, $\sigma_{N,k}$ is the rescaled increment between the time $T'_{N,k}$ where either the $k$-th mutant or the previous resident goes extinct and the analogously defined event $T'_{N,k-1}$ corresponding to the previous mutant (with $T'_{N,0}=0$). \color{black}

\begin{lemma} \label{lem_convergence_lv}
	For any~$k \in \N$, we have
	\begin{equation}\label{eq_Enk_tends}
		\P(E_{N,k} \cap \{\lambda_0 \le \Lambda_{N,1} \le \cdots \le \Lambda_{N,k}\}) \xrightarrow{N \to \infty} 1.
	\end{equation}
Moreover,
	\begin{equation}\label{eq_with_tildes}\mathds{1}_{E_{N,k}} \cdot (\sigma_{N,1},\Lambda_{N,1},\ldots,\sigma_{N,k},\Lambda_{N,k}) \xrightarrow{N \to \infty} (\sigma_1,\Lambda_1,\ldots,\sigma_k,\Lambda_k)\; \text{in distribution},\end{equation}
	where~$(\sigma_1,\Lambda_1,\ldots,\sigma_k,\Lambda_k)$ is distributed as in~\eqref{eq_original_K}. 
\end{lemma} 

\begin{proof}
	Define the auxiliary random variables
	\[\tilde\sigma_{N,1}:=   \mathds{1}_{E_{N,1}}\cdot \delta_N N^d\cdot T_{N,1}\]
and
\[\tilde\sigma_{N,j}:= \mathds{1}_{E_{N,j}}\cdot  \delta_N N^d\cdot (T_{N,j}-T_{N,j-1}''),\quad j \ge 2.\]
	Note that, for any~$k \in \N$ and any~$j \in \{1,\ldots k\}$, using the definition of~$E_{N,k}$ we have
	\[\mathds{1}_{E_{N,k}} \cdot  |\tilde{\sigma}_{N,j} - \sigma_{N,j}| \le \delta_NN^d\cdot 2\tn = 2\delta_N N^{d+1+\frac{\varepsilon_0}{2}} \xrightarrow{N\to \infty} 0.\]
	With this observation in mind, we note that in order to prove~\eqref{eq_with_tildes}, it suffices to prove
	\begin{equation}\label{eq_with_tildes1}\mathds{1}_{E_{N,k}} \cdot (\tilde\sigma_{N,1},\Lambda_{N,1},\ldots,\tilde\sigma_{N,k},\Lambda_{N,k}) \xrightarrow{N \to \infty} (\sigma_1,\Lambda_1,\ldots,\sigma_k,\Lambda_k)\; \text{in distribution}.\end{equation}

Fix~$k \in \N$,~$t_1,\ldots, t_k > 0$ and functions~$g_1,\ldots,g_k: (0,\infty) \to \mathbb{R}$ that are measurable and bounded.
	Using the strong Markov property and the definition of~${\mathscr{K}}_{N,t}$, we have
	\begin{align}
		\label{eq_this_exp}&\E\left[ \mathds{1}_{E_{N,k}}\cdot \mathds{1}\{\lambda_0 \le \Lambda_{N,1} \le \cdots \le \Lambda_{N,k}\}\cdot  \prod_{j=1}^k \mathds{1}\{\tilde{\sigma}_{N,j} \le t_j\}\cdot g_j(\Lambda_{N,j})\right]\\
		\nonumber&=\sum_{A_1 \in \mathscr{G}_N}\int_{{ \lambda_0} }^\infty {\mathscr{K}}_{N,t_1}(\lambda_0,\mathrm{d}\lambda_1;\Z^d_N,A_1)\sum_{A_2 \in \mathscr{G}_N} \int_{\lambda_1}^\infty{\mathscr{K}}_{N,t_2}(\lambda_1,\mathrm{d}\lambda_2;A_1,A_2)\\
		\nonumber&\hspace{3cm}\cdots \sum_{A_k \in \mathscr{G}_N} \int_{\lambda_{k-1}}^\infty {\mathscr{K}}_{N,t_k}(\lambda_{k-1},\mathrm{d}\lambda_k; A_{k-1},A_k)\cdot  \prod_{j=1}^k g_j(\lambda_j).
	\end{align}
	Using~\eqref{eq_new_kd}, this equals {
	\begin{align*}
		&\sum_{A_1 \in \mathscr{G}_N}\int_{\lambda_0}^\infty (K+D)(\lambda_0,\mathrm{d}\lambda_1)\; \mathscr{S}_{N,t_1}(\lambda_0,\lambda_1;\Z^d_N,A_1)\\
		&\quad \cdot \sum_{A_2 \in \mathscr{G}_N} \int_{\lambda_1}^\infty (K+D)(\lambda_1,\mathrm{d}\lambda_2)\;\mathscr{S}_{N,t_2}(\lambda_1,\lambda_2;A_1,A_2)\\
		&\quad \cdots \sum_{A_k \in \mathscr{G}_N} \int_{\lambda_{k-1}}^\infty (K+D)(\lambda_{k-1},\mathrm{d}\lambda_k)\;\mathscr{S}_{N,t_k}(\lambda_{k-1},\lambda_k;A_{k-1},A_k)\cdot  \prod_{j=1}^k g_j(\lambda_j).
	\end{align*}
Letting~$A_0 = \Z^d_N$, this is further equal to
	\begin{align*}
		&\int_{\lambda_0}^\infty (K+D)(\lambda_0,\mathrm{d}\lambda_1)\cdots \int_{\lambda_{k-1}}^\infty (K+D)(\lambda_{k-1},\mathrm{d}\lambda_k)\cdot \\ &\hspace{3cm}\left( \sum_{A_1,\ldots, A_k \in \mathscr{G}_N }\;\prod_{j=1}^k \mathscr{S}_{N,t_j}(\lambda_{j-1},\lambda_j;A_{j-1},A_j) \right) \cdot \left( \prod_{j=1}^k g_j(\lambda_j)\right).
	\end{align*}
	}
	By Lemma~\ref{prop_multiple_new} and the dominated convergence theorem, as~$N \to \infty$ this converges to
	\begin{align*}
	&\int_{\lambda_0}^\infty (K+D)(\lambda_0,\mathrm{d}\lambda_1)\cdots \int_{\lambda_{k-1}}^\infty (K+D)(\lambda_{k-1},\mathrm{d}\lambda_k)\cdot \\ &\hspace{3cm} \prod_{j=1}^k(1-\mathrm{e}^{-b(\lambda_{j-1})R(\lambda_{j-1})\cdot t_j})\cdot \mathscr{S}(\lambda_{j-1},\lambda_j)\cdot g_j(\lambda_j).	\end{align*}
	By~\eqref{eq_alt_alt_ker}, this is equal to
	\[ \int_{\lambda_0}^\infty {\mathscr{K}}(\lambda_0,\mathrm{d}\lambda_1)\cdots \int_{\lambda_{k-1}}^\infty {\mathscr{K}}(\lambda_{k-1},\mathrm{d}\lambda_k)\;\prod_{j=1}^k(1-\mathrm{e}^{-b(\lambda_{j-1})R(\lambda_{j-1})\cdot t_j})\cdot g_j(\lambda_j).\]
We have thus proved that
	\begin{equation}\label{eq_almost_tconv} \begin{split} &E_{N,k}\cdot \mathds{1}\{\lambda_0 \le \Lambda_{N,1} \le \cdots \le \Lambda_{N,k}\}\cdot (\tilde{\sigma}_{N,1},\Lambda_{N,1},\ldots,\tilde{\sigma}_{N,k},\Lambda_{N,k})\\
		&\hspace{3cm} \xrightarrow{N \to \infty}(\sigma_1,\Lambda_1,\ldots,\sigma_k,\Lambda_k) \quad \text{in distribution}.
	\end{split}\end{equation}
	Moreover, using the convergence of the expectation in~\eqref{eq_this_exp} with~$t_1=\cdots=t_k = t$ and~$g_1 = \cdots = g_k \equiv 1$, we have
	\begin{align*}
		&\liminf_{N\to \infty}\P\left(E_{N,k} \cap \{\lambda_0 \le \Lambda_{N,1} \le \cdots \le \Lambda_{N,k}\}\right) \\
		&\ge \liminf_{N \to \infty} \P\left(E_{N,k} \cap \{\lambda_0 \le \Lambda_{N,1} \le \cdots \le \Lambda_{N,k}\}\cap\left\{ \max_{1\le j \le k} \tilde{\sigma}_{N,j} \le t\right\}\right) \\
		&= \int_{\lambda_0}^\infty {\mathscr{K}}(\lambda_0,\mathrm{d}\lambda_1)\cdots \int_{\lambda_{k-1}}^\infty {\mathscr{K}}(\lambda_{k-1},\mathrm{d}\lambda_k)\;\prod_{j=1}^k(1-\mathrm{e}^{-b(\lambda_{j-1})R(\lambda_{j-1})\cdot t}).
	\end{align*}
	By dominated convergence, by taking~$t$ large we can make the last expression as close to 1 as desired { (also using the fact that~$b$ is strictly positive)}, proving~\eqref{eq_Enk_tends}. Combining this with~\eqref{eq_almost_tconv}, we also obtain~\eqref{eq_with_tildes1}.
\end{proof}

\begin{proof}[Proof of Theorem~\ref{thm_main}, part (a)]
Since the processes
	\[\int_0^{t} \mathds{1}\{\Phi(X_{N,(\delta_NN^d)^{-1}\cdot s}) = \Asterisk\}\;\mathrm{d}s,\quad t \ge 0\]
	 indexed by $N$ \color{black} are non-negative and increasing  with respect to time\color{black}, the convergence in the statement will follow from proving that, for any~$t > 0$,
	\begin{equation}\label{eq_convp}\int_0^{t} \mathds{1}\{\Phi(X_{N,(\delta_NN^d)^{-1}\cdot s}) = \Asterisk\}\;\mathrm{d}s \xrightarrow{N\to \infty} 0\quad \text{in probability}.\end{equation}
Fix~$t > 0$. Recalling~\eqref{eq_Z_fromjumps} and our non-explosiveness assumption, given~$\varepsilon > 0$ we can choose~$k \in \N$ such that
		\[\P\left(\sum_{j=1}^k \sigma_j > t\right) > 1-\varepsilon.\]
		Then, by Lemma~\ref{lem_convergence_lv}, for~$N$ large enough we have
		\[\P\left(E_{N,k} \cap \left\{\sum_{j=1}^k \sigma_{N,j} > t\right\}\right) > 1-2\varepsilon.\]
		Now, on the event~$E_{N,k} \cap \left\{\sum_{j=1}^k \sigma_{N,j} > t\right\}$, we have
	\[[0, (\delta_NN^d)^{-1}\cdot t] \cap \{t':\;\Phi(X_{N, t'}) = \Asterisk\}  \subseteq \cup_{j=1}^k [T_{N,j},T_{N,j}'] \subseteq \cup_{j=1}^k [T_{N,j},T_{N,j}''],\]
	so the Lebesgue measure of the set on the left-hand side is at most~$k\cdot \tn$, so, after changing variables,
	\[\int_0^{t} \mathds{1}\{\Phi(X_{N,(\delta_NN^d)^{-1}\cdot s}) = \Asterisk\}\;\mathrm{d}s \le \delta_NN^d\cdot k\cdot \tn = \delta_NN^{d+1+\frac{\varepsilon_0}{2}} \xrightarrow{N \to \infty} 0.\]
\end{proof}

Now define, for~$k \in \N$,
	\[Z_t^{(k)} = \lambda_0\cdot \mathds{1}_{[0,\sigma_1)}(t) +  \sum_{j=1}^{k-1} \Lambda_j \cdot \mathds{1}_{[\sum_{i=1}^j \sigma_i,\;\sum_{i=1}^{j+1} \sigma_i)}(t),\quad t \ge 0\]
	and, for each~$N$,
	\[Z_{N,t}^{(k)} = \mathds{1}_{E_{N,k}}\cdot  \left(\lambda_0\cdot \mathds{1}_{[0,\sigma_{N,1})}(t) +  \sum_{j=1}^{k-1} \Lambda_{N,j} \cdot \mathds{1}_{[\sum_{i=1}^j \sigma_{N,i},\;\sum_{i=1}^{j+1} \sigma_{N,i})}(t)\right),\quad t \ge 0.\]
Note that
\begin{equation*}
	\text{on } E_{N,k},\quad Z_{N,t} = Z_{N,t}^{(k)} \quad \text{for } 0 \le t < \sum_{j=1}^k \sigma_{N,j}.
\end{equation*}

\begin{lemma}\label{lem_conv_Skk}
	For each~$k \in \N$, we have that~$(Z_{N,t}^{(k)})_{t \ge 0}$ converges in distribution as~$N \to \infty$ (with respect to the Skorokhod topology) to~$(Z_t^{(k)})_{t \ge 0}$.
\end{lemma}
\begin{proof}
	The statement follows from Lemma~\ref{lem_convergence_lv} and the observation that the mapping from~$(0,\infty)^{2k}$ to the Skorokhod space~$D([0,\infty),(0,\infty))$ given by
	\begin{equation*}(t_1,\lambda_1,\ldots,t_k,\lambda_k) \mapsto \gamma(\cdot) := \lambda_0 \cdot \mathds{1}_{[0,t_1)}(\cdot) + \sum_{j=1}^{k-1} \lambda_j \cdot \mathds{1}_{[\sum_{i=1}^j t_i,\;\sum_{i=1}^{j+1} t_i)}(\cdot) 
	\end{equation*}
	is continuous.	
\end{proof}

\begin{proof}[Proof of Theorem~\ref{thm_main}, part~(b)]
	It suffices to prove that
	\[\E\left[F((Z_{N,t})_{t \ge 0})\right] \xrightarrow{N \to \infty} \E\left[F((Z_{t})_{t \ge 0})\right]\]
	for any function~$F: D([0,\infty),[0,\infty)) \to \mathbb{R}$ that is bounded, measurable, and so that~$F(\gamma)$ only depends on~$\gamma$ through~$(\gamma(t))_{0 \le t \le \bar{t}}$ for some~$\bar{t} > 0$. We note that, for any~$k$,
	\begin{align}
\nonumber		&\left| \E[F((Z_{t})_{t \ge 0})] - \E[F((Z^{(k)}_{t})_{t \ge 0})]\right| \\
		\label{eq_last_triangle1}&\quad\le \|F\|_\infty \cdot  \P((Z_{t})_{0 \le t \le \bar{t}} \neq (Z_{t}^{(k)})_{0 \le t \le \bar{t}}) \le\|F\|_\infty \cdot \P(\sigma_k \le \bar{t}).
\end{align}
Similarly, for any~$k$ and~$N$,
	\begin{align}
		\nonumber&\left| \E[F((Z_{N,t})_{t \ge 0})] - \E[F((Z^{(k)}_{N,t})_{t \ge 0})]\right| \\
		\label{eq_last_triangle2}&\quad  \le \|F\|_\infty\cdot (\P((E_{N,k})^c) + \P(E_{N,k} \cap \{\sigma_{N,k} \le \bar{t}\})).
\end{align}
	Fix~$\varepsilon > 0$. By our non-explosiveness assumption, we can choose~$k$ large enough that the right-hand side of~\eqref{eq_last_triangle1} is smaller than~$\varepsilon/6$. Next, since Lemma~\ref{lem_convergence_lv} implies that
	\[\P((E_{N,k})^c)\xrightarrow{N \to \infty} 0 \quad \text{and}\quad \P(E_{N,k} \cap \{\sigma_{N,k} \le \bar{t}\}))\xrightarrow{N \to \infty} \P(\sigma_k \le \bar{t}),\] 
	we can choose~$N_0$ such that, for all~$N \ge N_0$, the right-hand side of~\eqref{eq_last_triangle2} is smaller than~$\varepsilon/3$. Finally, using Lemma~\ref{lem_conv_Skk}, enlarging~$N_0$ if necessary, we can guarantee that for any~$N \ge N_0$,
	\[\left| \E[F((Z^{(k)}_{N,t})_{t \ge 0})] - \E[F((Z^{(k)}_{t})_{t \ge 0})]\right| < \varepsilon/3.\]
	Now, by the triangle inequality, for any~$N \ge N_0$ we have
	\[\left| \E[F((Z_{N,t})_{t \ge 0})] - \E[F((Z_{t})_{t \ge 0})]\right| < \varepsilon,\]
	completing the proof.
\end{proof}

\section{Preliminaries on the classical and two-type contact process}\label{s_preliminary}
\subsection{Classical contact process}\label{ss_one_type_prelim}
Here we will briefly list some notations and properties involving the contact process that we will need in the next sections. We refer the reader to~\cite{L99} for a complete exposition and proofs of the statements that appear below lacking justification.

We adopt the standard abuse of notation of identifying, for a set~$A$, the configuration~$\zeta \in \{0,1\}^A$ with the set  of points\color{black}~$\{x \in A: \zeta(x) = 1\}$. In particular, the element of~$\{0,1\}^A$ that is identically zero is sometimes denoted by~$\varnothing$. For a contact process configuration~$\zeta \in \{0,1\}^A$, it is common to say that~$x \in A$ is infected if~$\zeta(x) = 1$ and healthy otherwise. While we will maintain some of the infection-related terminology (for instance, we write about infection paths below), we will simply say that~$x$ is \textit{occupied} (by, say, a particle) when~$\zeta(x) = 1$ and \textit{empty} otherwise. 

Let~$G=(V,E)$ be a graph and~$\lambda > 0$. A \textit{graphical construction} for the contact process with parameter~$\lambda$ on~$G$ is a collection of independent Poisson point processes on~$[0,\infty)$: first, processes~$\{R^x: x \in V\}$, each with rate one, and second, processes~$\{R^{(x,y)}:x,y \in V,\;\{x,y\} \in E\}$, each with rate~$\lambda$. In case~$t \in R^x$, we say that~$x$ has a \textit{death mark} at time~$t$, and in case~$t \in R^{(x,y)}$, we say that there is a \textit{birth arrow} from~$x$ to~$y$ at time~$t$. An \textit{infection path} is a function~$\gamma: I \to V$, where~$I \subseteq [0,\infty)$ is an interval, so that~$\gamma$ is right continuous with left limits and satisfies the two requirements that~$t \notin R^{\gamma(t)}$ for all~$t \in I$ and~$t \in R^{(\gamma(t-),\gamma(t))}$ whenever~$\gamma(t-) \neq \gamma(t)$. Given~$(x,s),(y,t) \in V\times [0,\infty)$ with~$s \le t$, we write~$(x,s)\rightsquigarrow (y,t)$ to represent the event that there exists an infection path starting at~$x$ at time~$s$ and ending at~$y$ at time~$t$. Additionally, for~$A \subseteq V$, we write~$A \times \{s\} \rightsquigarrow (y,t)$ to indicate the event that~$(x,s) \rightsquigarrow (y,t)$ for some~$x \in A$. We employ the similar notations~$(x,s) \rightsquigarrow B \times \{t\}$ and~$A \times \{s\} \rightsquigarrow B \times \{t\}$ with obvious meanings. 
We write~$(x,s) \not \rightsquigarrow (y,t)$ (or~$A \times \{s\} \not \rightsquigarrow (y,t)$, etc.) to denote the event that there is no infection path between the starting point (or set) and ending point (or set).
Finally, given~$A' \subseteq A$,~$x,y \in A'$ and~$s \le t$, we say that~$(x,s) \rightsquigarrow (y,t)$ \textit{inside}~$A'$ to mean the event that there is an infection path~$\gamma: [s,t] \to A$ from~$(x,s)$ to~$(y,t)$ so that~$\gamma(r) \in A'$ for all~$r$.

Given a graphical construction and any initial configuration~$\zeta_0 \in \{0,1\}^{V}$, we construct the contact process started from~$\zeta_0$ by letting
\[\zeta_t(x) := \mathds{1}\{\zeta_0 \times \{0\} \rightsquigarrow (x,t)\},\quad x \in V,\; t \ge 0\]
(formally speaking, the contact process is the process on~$\{0,1\}^V$ whose distribution is equal to that of the process obtained from the above construction). The graphical construction has several good features, notably it allows one to obtain processes started from all possible initial configurations in a single probability space. Moreover, if~$(\zeta_t)_{t \ge 0}$ and~$(\zeta_t')_{t \ge 0}$ are obtained with the same graphical construction, then we have~$\zeta_t = \zeta_t'$ for all~$t \ge \inf\{s: \zeta_s = \zeta_s'\}$. In addition, with the aid of the graphical construction, it is easy to verify the so-called \textit{self-duality}  of the contact process, namely, the property that for any~$A,B \subseteq V$,
\[ \P(A \times \{0\} \rightsquigarrow B \times \{t\}) = \P(B \times \{0\} \rightsquigarrow A \times \{t\}).\]

The contact process on~$\Z^d$ exhibits a phase transition, as follows. Let
\[ \lambda_c(\Z^d) := \inf\{\lambda > 0:\; \P(\zeta^{\{o\}}_t \neq \varnothing \; \forall t) > 0\},\]
where~$(\zeta^{\{o\color{black}\}}_t)_{t \ge 0}$ denotes the process started from the configuration with a single~$1$ at the origin. Then, we have that~$\lambda_c(\Z^d) \in (0,\infty)$ {(\cite[p.307]{L85}, \cite[p.43]{L99})}.  Using monotonicity arguments, one can show that the process started from full occupancy (that is, the all-one configuration) converges in distribution, as~$t \to \infty$ {(see~\cite[p.35]{L99})}. The limiting distribution is called the \textit{upper stationary distribution} of the process, denoted~$\mu_\lambda$. In case~$\lambda > \lambda_c(\Z^d)$, this distribution is supported on configurations containing infinitely many~$0$'s and~$1$'s (whereas in case~$\lambda \le \lambda_c(\Z^d)$, it is equal to the unit mass on the all-zero configuration).

For future reference, we 
{
gather together in the following proposition a few properties of the supercritical contact process on $\Z^d$. Given $\zeta \in \{0,1\}^{\Z^d}$, we let $|\zeta|$ be the number of $x$ such that $\zeta(x) = 1$.
\begin{proposition}\label{prop_facts_contact}
    Let $\lambda > \lambda_c(\Z^d)$, let $(\zeta_t)_{t \ge 0}$ be the contact process on $\Z^d$ with rate~$\lambda$ started from an arbitrary initial configuration, and  $(\zeta^{\{o\}}_t)_{t \ge 0}$ be the contact process on $\Z^d$ 
 with rate $\lambda$ started from $\{o\}$. 
    \begin{itemize}
        \item[(a)]
    There exists a constant $c_\mathrm{death} > 0$ such that
    \begin{align} 
       \nonumber & \P(\zeta_t \neq \varnothing,\;\zeta_s = \varnothing \text{ for some } s > t) < \exp\{- c_\mathrm{death} \cdot t\}, \quad t > 0,\\[.2cm] 
       \label{eq_const_death} & \P(\exists t \ge 0:\;\zeta_t = \varnothing) < \exp\{-c_\mathrm{death}\cdot |\zeta_0|\}.
    \end{align}
\item[(b)] There exists a constant $c_\mathrm{speed} > 0$ such that
 \begin{equation}\label{eq_constant_speed}
	\P(\zeta^{\{o\}}_t \not\subseteq Q_{\Z^d}(o,c_{\mathrm{speed}}^{-1}\cdot t)) \le \exp\{-c_{\mathrm{speed}}\cdot t\}.
\end{equation}
\item[(c)] There exists a constant $c_\mathrm{full} >0$ such that 
\begin{equation*}
    \P\big( |\zeta_t^{\{o\}}| > c_\mathrm{full} \cdot t \big| \zeta_t^{\{o\}} \neq \varnothing \big) > 1- \exp\{-c_\mathrm{full} \cdot t\},\quad t > 0.
\end{equation*}
\end{itemize}
\end{proposition}
Item (a) in the proposition is~\cite[Theorem 2.30, pp.57]{L99}. Item (b) follows for instance from a comparison between the contact process and first-passage percolation, and large deviation estimates { (using the amenability of~$\Z^d$)}. We omit the details. Since we have not found item~(c) stated in the literature, we provide a sketch of proof.
\begin{proof}[Sketch of proof of Proposition~\label{prop_facts_contact}(c)]
    First assume that the dimension is one. Assume that $(\zeta^{\{o\}}_t)$ is obtained from a graphical construction, and using this same graphical construction, let~$(\zeta^1_t)_{t \ge 0}$ be the process started from all vertices occupied. Define the left and right edges of~$\zeta^{\{o\}}_t$ as
    \[l_t:= \inf\{x:\zeta^{\{o\}}_t = 1\}, \qquad r_t:= \sup\{x: \zeta^{\{o\}}_t = 1\},\]
    with~$\inf \varnothing = \infty$ and~$\sup \varnothing = -\infty$. Then, there exist~$\alpha > 0$ and~$c > 0$ (depending on~$\lambda$) such that
\begin{equation}\label{eq_full1}
    \P(l_t < -\alpha t < \alpha t < r_t \mid \zeta^{\{o\}}_t \neq \varnothing) > 1- \mathrm{e}^{-ct},\quad t > 0.
\end{equation}
    A statement of this form is proved for supercritical oriented percolation in~\cite[Section 11]{D84} using a contour argument, and it may be translated to the contact process using the standard renormalization technique explained in~\cite[Section VI.3]{L85}.

    Recall that~$\mu_\lambda$ denotes the upper stationary distribution of the process. By the main theorem in~\cite{DS88b}, there exist~$\alpha' > 0$ and~$c' > 0$ such that
    \[\mu_\lambda(\{\zeta:\;|\zeta \cap [-\alpha t, \alpha t]| > \alpha' t\}) > 1- \mathrm{e}^{-c' t},\quad t > 0.\]
    Since the law of~$\zeta^1_t$ stochastically dominates~$\mu_\lambda$, this gives 
    \begin{equation}\label{eq_full2}
        \P(|\zeta^1_t \cap [-\alpha t, \alpha t]| > \alpha' t)> 1- \mathrm{e}^{-c' t},\quad t > 0.
    \end{equation}
        Using the fact that infection paths that cross each other must intersect at some space-time point, it is easy to see that on~$\{l_t < -\alpha t < \alpha t < r_t\}$, we have~$\zeta^{\{o\}}_t(x) = \zeta^1_t(x)$ for all~$x \in [-\alpha t, \alpha t]$. Combining this observation with~\eqref{eq_full1} and~\eqref{eq_full2}, we obtain
        \[
        \P(|\zeta^{\{o\}}_t \cap [-\alpha t, \alpha t]| > \alpha' t \mid \zeta^{\{o\}}_t \neq \varnothing) > 1- \mathrm{e}^{-ct} - \frac{\mathrm{e}^{-c't}}{\P(\zeta^{\{o\}}_t \neq \varnothing)}.\]
The desired bound now follows from noting that~$\P(\zeta^{\{o\}}_t \neq \varnothing)$ is bounded away from zero (it is larger than~$\P(\zeta^{\{o\}}_t \neq \varnothing \;\forall t) > 0$).

In order to extend this to higher dimension, we first take the Bezuidenhout--Grimmett renormalization, which maps the supercritical contact process on~$\Z^d$ into supercritical oriented percolation on dimension~$d$ (plus the time dimension), see~\cite{BG90}, or~\cite[Chapter I.2]{L99}. Restricting the model to a slab, we obtain a supercritical oriented percolation model on dimension $1$ (plus the time dimension). We then carry out the proof in the preceding paragraphs to this oriented percolation model.
\end{proof}

}

The phase transition of the contact process on the infinite lattice~$\Z^d$ has a counterpart on large boxes of~$\Z^d$. In particular, in~\cite{M93} and~\cite{M99}, it is proved that for any~$\lambda > \lambda_c(\Z^d)$,  there exists a constant~$c_\mathrm{long} = c_\mathrm{long}(\lambda) > 0$ such that the following holds. Let~$\tau(Q_{\Z^d}(o,N))$ denote the \textit{extinction time} of the contact process on~$Q_{\Z^d}(o,N)$ with rate~$\lambda$ (the extinction time of the contact process on a graph is defined as the hitting time of the empty configuration, when the process is started from full occupancy). Then, we have that
\begin{equation}\label{eq_from_tom_expectation}
	\frac{\log \E[\tau(Q_{\Z^d}(o,N))]}{N^d} \xrightarrow{N \to \infty} c_\mathrm{long}.
\end{equation}
%Since the box~$Q_{\Z^d}(o,N)$ can be seen as a subgraph of the torus~$\Z^d_N$, we also obtain, by simple monotonicity considerations, that (still assuming~$\lambda > \lambda_c(\Z^d)$)
%\begin{equation}\label{eq_from_tom_expectation}
	%\E[\tau(\Z^d_N)] > \exp\{-c_\mathrm{long}\cdot N^d\}\quad \text{for all } N \ge 1,
%\end{equation}
%where~$\tau(\Z^d_N)$ is the extinction time of the process on~$\Z^d_N$ with rate~$\lambda$.

We will need the following statement concerning the occupation of the supercritical contact process on~$\Z^d_N$.
\begin{proposition}\label{prop_density}
	For any~$\lambda > \lambda_c(\Z^d)$ and~$\varepsilon \in (0,1)$, the following holds for~$N$ large enough. Let~$(\zeta_t^1)_{t \ge 0}$ denote the contact process with rate~$\lambda$ on~$\Z^d_N$ started from full occupancy. Then,
	\begin{equation}
		\mathbb{P}\left(\begin{array}{l}\text{for all }t \le N^{\log N},\; \zeta^1_t \text{ intersects} \\[.1cm] \text{ all boxes of radius $N^\varepsilon$ in $\Z^d_N$}\end{array}\right)> 1- N^{-\log N}.
	\end{equation}
\end{proposition}

Finally, we will need a statement guaranteeing that with high probability, the process started from a sufficiently dense configuration couples in a short time with the process started from full occupancy.

\begin{proposition}\label{prop_coupling}
	For any~$\lambda > \lambda_c(\Z^d)$ and~$\varepsilon \in (0,1)$, the following holds for~$N$ large enough. Let~$(\zeta_t)_{t \ge 0}$ and~$(\zeta^1_t)_{t \ge 0}$ denote contact processes with rate~$\lambda$ on~$\Z^d_N$, both built from a single graphical construction, so that~$\zeta^1_0 = \Z^d_N$ and~$\zeta_0$ is a configuration that intersects all boxes of radius~$N^\varepsilon$ in~$\Z^d_N$. Then,
	\begin{equation*}
	\P(\zeta_t = \zeta^1_t \text{ for all } t \ge N^{12\varepsilon}) > 1- N^{-\log N}.
	\end{equation*}
\end{proposition}

The proofs of Propositions~\ref{prop_density} and \ref{prop_coupling} will be carried out in Section~\ref{ss_proofs_classical} in the Appendix. 
{
	The idea behind Proposition~\ref{prop_coupling} is as follows. Let us say we tile the torus with sub-boxes of side length~$N^\varepsilon$. The assumption on~$\zeta_0$ guarantees that each sub-box~$Q$ has a vertex~$u(Q) \in Q$ such that~$\zeta_0(u)= 1$. For fixed~$T > 0$, in order to have~$\zeta_T = \zeta^1_T$, loosely speaking, we need to ensure that the union of the regions reached by the contact processes started from each~$u(Q)$ and ran up to time~$T$ span the whole torus. This already suggests that~$T$ should be at least~$CN^\varepsilon$ for some large constant~$C$. Since there are many boxes, it could happen that things do not go well in some of them (for instance, some of the contact processes may quickly die, or grow too slowly), so we take~$T$ as a large power of~$N^\varepsilon$  to get concentration (the reason for the specific choice of 12 will become clear throughout the proof). 
}

Combining Propositions~\ref{prop_density} and \ref{prop_coupling} {(with~$\varepsilon = \frac{1}{288}$, so that~$(N^\varepsilon)^{12} = N^{1/24}$) }, we obtain that, for~$\lambda > \lambda_c(\Z^d)$,~$N$ large enough and~$(\zeta_t)_{t \ge 0}$ the contact process on~$\Z^d_N$, we have
\begin{equation}\label{eq_coupling_combining}
	\zeta_0 \in \mathscr{G}_N \quad \Longrightarrow \quad \P(\zeta_t \in \mathscr{G}_N \text{ for all } t \in [N^{1/24},N^{\log N}]) > 1-2N^{-\log N}.
\end{equation}

	\subsection{Two-type contact process}\label{ss_two_type_prelim}
As mentioned in the Introduction, the two-type contact process is a variant of the contact process in which we have two types of occupants, both of which can die and give birth at vacant neighboring sites. We now give a construction of this process.

Let~$G=(V,E)$ be a graph, which we assume to be finite for the moment.  Also let~$\lambda,\lambda' > 0$. Whenever we deal with the two-type contact process,~$\lambda$ will denote the birth rate of type~$1$ and~$\lambda'$ the birth rate of type~$2$ (both types die with rate one). For the rest of this section, also assume that~$\lambda' > \lambda$\color{black}. We take three collections of independent Poisson point processes on~$[0,\infty)$: first, death marks processes~$\{R^x:x \in V\}$ with rate one; second, basic birth arrow processes~$\{R^{(x,y)}:x,y \in V,\;\{x,y\} \in E\}$ with rate~$\lambda$; third, extra birth arrow processes~$\{\tilde{R}^{(x,y)}:x,y \in V,\; \{x,y\} \in E\}$ with rate~$\lambda' - \lambda$. Intuitively, both the basic and the extra birth arrows are usable by type~$2$, but only the basic ones are usable by type~$1$. That is, denoting the process by~$(\xi_t)_{t \ge 0}$, for any~$t > 0$ we have the rules:
\begin{equation*}
	\begin{array}{rll}
		t \in R^{x}& \Longrightarrow & \xi_t = \xi_{t-}^{x \leftarrow 0};\\[.2cm]
		t \in R^{(x,y)},\;\xi_{t-}(x) = 1,\; \xi_{t-}(y) = 0 & \Longrightarrow & \xi_t = \xi_{t-}^{y \leftarrow 1};\\[.2cm]
		t \in R^{(x,y)}\cup \tilde{R}^{(x,y)},\;\xi_{t-}(x) = 2,\; \xi_{t-}(y) = 0 & \Longrightarrow &\xi_t=\xi_{t-}^{y \leftarrow 2}
	\end{array}
\end{equation*}
where~$\xi^{x \leftarrow i}$ is defined  according to \eqref{eq_xarrowi}. %by
%\[\xi^{x \leftarrow i}(y) = \begin{cases} \xi(y),&\text{if } y \neq x;\\ i&\text{if }y = x.\end{cases}\]
Since~$G$ is finite, these rules define the process~$(\xi_t)_{t \ge 0}$ as a continuous-time Markov chain on~$\{0,1,2\}^{V}$.

We can extend this construction to define the process on~$\Z^d$ through a standard limiting procedure, which we now briefly outline. Take a graphical construction as above, with~$G$ equal to~$\Z^d$. Denote the full collection of Poisson processes by~$H$, and let~$H_r$ denote its restriction to the box~$Q_{\Z^d}(o,r)$ (that is,~$H_r$ includes the death marks and birth arrows for which the vertices involved belong to this box). Given~$\xi_0 \in \{0,1,2\}^{\Z^d}$, we define~$\xi_{r,0}$ as the restriction of~$\xi_0$ to~$Q_{\Z^d}(o,r)$, and let~$(\xi_{r,t})_{t \ge 0}$ be the process started from~$\xi_{r,0}$ and obtained from~$H_r$, as in the previous paragraph. A key observation is as follows. For any~$r$ and any~$(x,t) \in Q_{\Z^d}(o,r) \times [0,\infty)$, keeping a realization of~$H$ fixed, we have
\begin{align*}
	&\{y \in Q_{\Z^d}(o,r): \text{changing $\xi_{r,0}(y)$ affects $\xi_{r,t}(x)$}\} \\
	&\subseteq \{y \in Q_{\Z^d}(o,r): \text{there is an infection path in $H_r$ from $(y,0)$ to $(x,t)$}\}
\end{align*}
(here, infection paths are allowed to follow both basic and extra arrows in~$H_r$). As~$r \to \infty$, the set on the right-hand side increases to the (almost-surely finite) set of~$y \in \Z^d$ such that there is an infection path in~$H$ from~$(y,0)$ to~$(x,t)$.
 This proves that the process~$(\xi_t)_{t \ge 0}$ obtained through~$\xi_t(x) = \lim_{r \to \infty} \xi_{r,t}(x)$ is well defined. Moreover, this process has the feature that, for any~$\bar{r}$ and~$\bar{t}$, for a fixed realization of~$H$, the values of~$\{\xi_t(x): 0 \le t \le \bar{t},\; x \in Q_{\Z^d}(o,\bar{r})\}$ only depends on the values of~$\xi_0$ at the sites~$y$ for which there is an infection path in~$H$ from~$(y,0)$ to~$Q_{\Z^d}(o,\bar{r}) \times [0,\bar{t}]$.

 {
Alternatively to the graphical construction outlined above, the multitype contact process can be constructed from its generator with the Hille--Yosida approach, see~\cite[Theorem 3.9, p.27]{L85}, whose assumptions are easy to check in our setting.
 }

\section{Appearance of a mutant}\label{s_appearance}
\subsection{The contact process with inane mutation marks}\label{ss_preparation}
The goal of this section is to prove Proposition~\ref{prop_appearance_new}. Recall that this proposition involves the adaptive contact process started from a configuration containing a single type, and ran until  a mutant appears (if it ever does). With this in mind, we can reformulate the statement of Proposition~\ref{prop_appearance_new} as a statement involving a classical contact process in which the birth arrows have mutation marks that have no effect in the dynamics (which is why we call them inane). This is what we do in this section.

Throughout this section, we fix~$\lambda > \lambda_c(\Z^d)$. { We assume for the rest of the section that~$N$ is large enough (in a way that depends on the value~$\lambda$ that has been fixed) that~$\delta_N\cdot b(\lambda) < 1$. We can do this in this section because we are interested in proving Proposition~\ref{prop_appearance_new}, and the statement of this proposition concerns fixed~$\lambda$ and~$N \to \infty$. We emphasize that~$\delta_N \cdot b(\lambda) < 1$ is \textit{not} a global assumption in the paper.
}

Let~$H_{\Z^d_N}$ be a graphical construction for the contact process on~$\Z^d_N$. After enlarging the probability space, assume that independently, each birth arrow of~$H_{\Z^d_N}$ has a mark with probability~$\delta_N\cdot b(\lambda)$, { which in this section is assumed to be less than or equal to one}. For each~$A \subseteq \Z^d_N$, let~$(\zeta^A_t)_{t \ge 0}$ denote the contact process obtained from this graphical construction (with the marks having no effect) and started from~$\zeta^A_0 = \mathds{1}_A$.

Now, define
\[\tilde{T}_N^A:= \inf\left\{\begin{array}{c} t > 0: \; \exists u: \zeta_{t-}^A(u) = 0,\;\zeta_t^A(u) = 1 \text{ and the birth arrow  }\\ \text{that caused the birth at $u$ at time $t$ has a mark}\end{array} \right\}.\]
	On~$\{\tilde{T}_N^A < \infty\}$, letting~$\tilde{\mathcal{X}}_N^A$ denote the position of the birth at time~$\tilde{T}_N^A$, define
	\[\tilde{\mathscr{L}}_N^A:= \{v \in \Z^d_N \backslash \{o\}:\; (\zeta_{\tilde{T}_N^A}^A \circ \theta_{\tilde{\mathcal{X}}_N^A})(v) = 1\}\] 
	(on~$\{\tilde{T}_N^A = \infty\}$, we can define~$\tilde{\mathscr{L}}_N^A$ in some arbitrary way).  It is now readily seen that~$(\tilde{T}_N^A,\tilde{\mathscr{L}}_N^A)\cdot \mathds{1}\{\tilde{T}_N^A < \infty\}$ has the same distribution as that of~$(T_N,\mathscr{L}_N)\cdot \mathds{1}\{T_N < \infty\}$ for the adaptive contact process on~$\Z^d_N$ started from~$X_{N,0} =\lambda \cdot \mathds{1}_A$. In particular, for any~$A,B,t$,
\begin{equation}\label{eq_reform_ucp}
	\mathcal{U}_{N,t}(\lambda;A,B) = \P(\delta_NN^d\cdot \tilde{T}_N^A\le t,\; \tilde{\mathscr{L}}_N^A = B).
\end{equation}
Due to this equality, we will only deal with the classical contact process (with inane mutation marks) for the rest of this section.

\begin{figure}[htb]
\begin{center}
\setlength\fboxsep{0cm}
\setlength\fboxrule{0cm}
\fbox{\includegraphics[width=\textwidth]{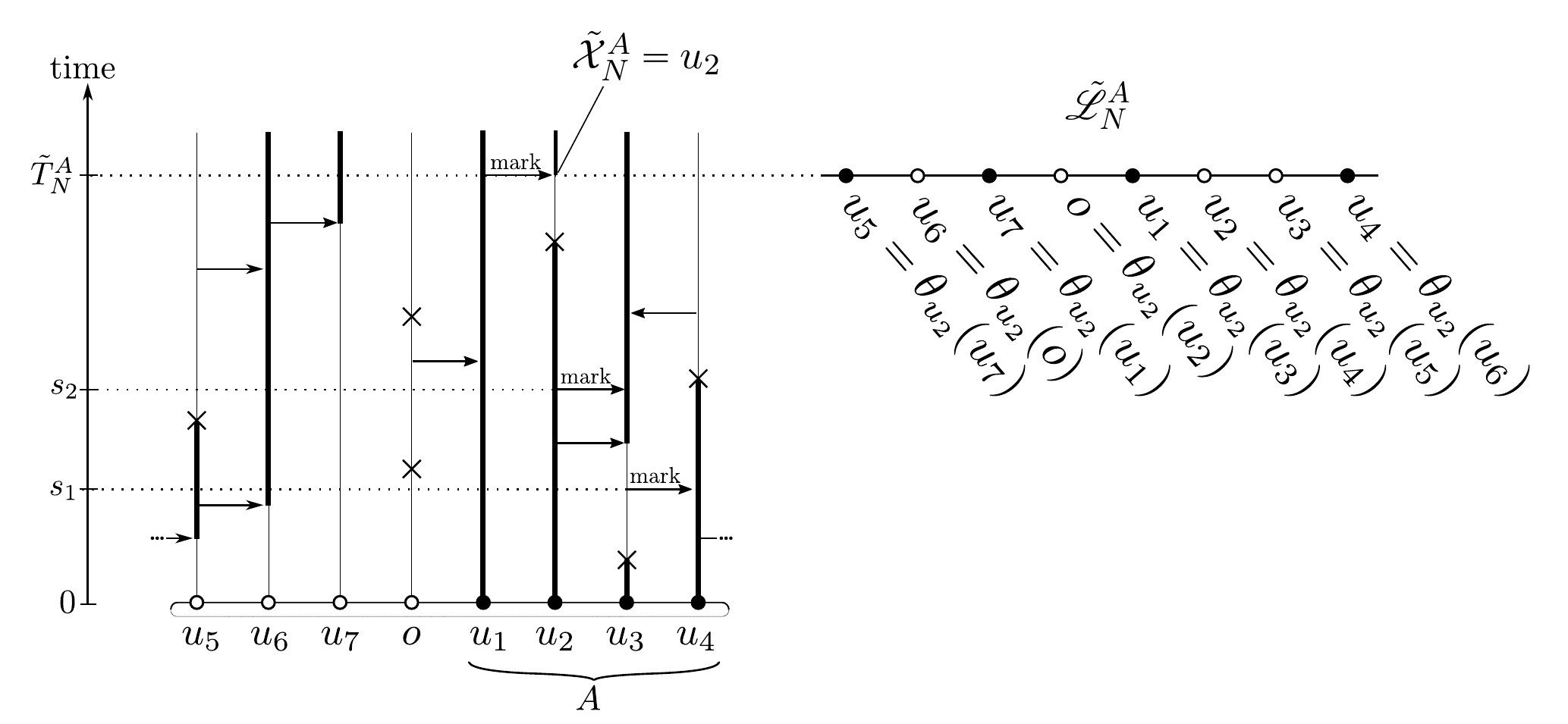}}
\end{center}
	\caption{\label{fig_landscape} { Illustration of $\tilde{T}_N^A$, $\tilde{\mathcal{X}}_N^A$ and $\tilde{\mathscr{L}}_N^A$. Note that $\tilde{T}_N^A$ does not occur at times $s_1$ or $s_2$, as the birth arrows with mutation marks that appear at these times do not produce a shift from state 0 to state 1 at their endpoint.}}
\end{figure}

We now proceed to a few reductions that will make it easier to prove Proposition~\ref{prop_appearance_new}. The first one is the following, which will allow us to focus on the process started from full occupancy.

\begin{lemma}\label{lem_reduction_full_occ}
	For any~$t > 0$ we have
	\[\sup_{A \in \mathscr{G}_N} \sum_{B \in \mathscr{G}_N}| \mathcal{U}_{N,t}(\lambda;A,B) - \mathcal{U}_{N,t}(\lambda;\Z^d_N,B)| \xrightarrow{N \to \infty}0. \]
\end{lemma}
\begin{proof}
	{ Recall the definition of~$\tn$ in~\eqref{eq_def_of_tn}.} By Lemma~\ref{lem_bound_Ttn},  we have  
	\begin{equation}
		\label{eq_T_bound}
		\sup_{A \subseteq \Z^d_N} \mathbb{P}(\tilde{T}_N^A < {\tn}) \xrightarrow{N \to \infty} 0.
	\end{equation}
{ Fix~$t>0$ and assume that~$N$ is large enough that $(\delta_NN^d)^{-1}\cdot t >  \tn$.
This is possible by~\eqref{eq_assumption_delta} and the definition of~$\tn$ in~\eqref{eq_def_of_tn}).} Defining
	\[\mathcal{U}_{N,t}'(\lambda;A,B):= \P({\tn} < \tilde{T}^A_N < (\delta_NN^d)^{-1}\cdot t,\; \tilde{\mathscr{L}}^A_N = B),\]
	we have
	\begin{equation*}
		\sup_{A \subseteq \Z^d_N} \sum_{B \in \mathscr{G}_N}\left| \mathcal{U}_{N,t}(\lambda;A,B) - \mathcal{U}_{N,t}'(\lambda;A,B)\right| \xrightarrow{N\to\infty}0.
	\end{equation*}
In particular, the statement of the lemma will follow from
	\begin{equation}\label{eq_next_Q}
		\sup_{A \in \mathscr{G}_N} \sum_{B \in \mathscr{G}_N} \left| \mathcal{U}'_{N,t}(\lambda;A,B) - \mathcal{U}'_{N,t}(\lambda;\Z^d_N,B)\right| \xrightarrow{N\to \infty}0.
	\end{equation}
	Now,  we have
	\[\mathcal{U}_{N,t}'(\lambda;A,B) = \sum_{A' \subseteq \Z^d_N} \kappa_N(A,A') \cdot \ell_N(A',B),\]
where we define
	\begin{align*}
		&\kappa_N(A, A'):= \mathbb{P}(\tilde{T}^A_N > {\tn},\; \zeta^A_N =  \mathds{1}_{A'} ),\\[.2cm]
		&\ell_{N}(A', B):= \P(\tilde{T}^{A'}_N < (\delta_NN^d)^{-1}\cdot t - {\tn},\;\; \tilde{\mathscr{L}}^{A'}_N = B).
	\end{align*}
	Using the fact that~$\sum_{B} \ell_N(A',B) \in [0,1]$ and the triangle inequality, we see that to prove~\eqref{eq_next_Q} it suffices to prove that
	\begin{equation}\label{eq_next_Q1}
		\sup_{A \in \mathscr{G}_N} \sum_{A'} |\kappa_N(A,A') - \kappa_N(\Z^d_N,A')| \xrightarrow{N \to \infty} 0.
	\end{equation}
	This follows readily from Proposition~\ref{prop_coupling}.
\end{proof}

Our next goal is to replace, in Proposition~\ref{prop_appearance_new}, the sum over~$B \in \mathscr{G}_N$ by a sum over all~$B \subseteq \Z^d_N$. 

\begin{lemma}\label{lem_after_in_G}
	For any~$t > 0$,  we have
	\[\sum_{{B \subseteq \Z^d_N, B \notin \mathscr{G}_N}} \mathcal{U}_{N,t}(\lambda;\Z^d_N,B) \xrightarrow{N \to \infty} 0.\]
\end{lemma}
\begin{proof}
	Let us omit the superscript~$\Z^d_N$, so that
 \[\zeta_t = \zeta^{\Z^d_N}_t,\qquad \tilde{T}_N = \tilde{T}_N^{\Z^d_N},\qquad \tilde{\mathscr{L}}_N =\tilde{\mathscr{L}}_N^{\Z^d_N}.\]
 By~\eqref{eq_reform_ucp}, the sum in the statement of the lemma equals
	\begin{align*}
		\P(\tilde{T}_N < (\delta_NN^d)^{-1}\cdot t,\;\tilde{\mathscr{L}}_N \notin \mathscr{G}_N) \le \P(\tilde{T}_N < (\delta_NN^d)^{-1}\cdot t,\;\zeta_{\tilde{T}_N} \notin \mathscr{G}_N^*),
	\end{align*}
	where~$\mathscr{G}_N^*$ is the set of subsets of~$\Z^d_N$ that intersect all boxes of radius~$N^{1/289}$ in~$\Z^d_N$. (The reason we cannot just repeat~$\mathscr{G}_N$ in the expression on the right-hand side is that one of the 1's of~$\zeta_{\tilde{T}_N}$ is deleted in the definition of~$\tilde{\mathscr{L}}_N$, so it could be that~$\zeta_{\tilde{T}_N} \in \mathscr{G}_N$ but~$\tilde{\mathscr{L}}_N \notin \mathscr{G}_N$). By~\eqref{eq_assumption_delta_2}, we have~$(\delta_NN^d)^{-1}\cdot t \le N^{\log(N)}$ for~$N$ large enough, so
	\[\P(\tilde{T}_N < (\delta_NN^d)^{-1}\cdot t,\;\zeta_{\tilde{T}_N} \notin \mathscr{G}_N^*) \le \P(\exists t \le N^{\log(N)}:\; \zeta_t \notin \mathscr{G}_N^*).\]
	By Proposition~\ref{prop_density}, the right-hand side tends to zero as~$N \to \infty$.
\end{proof}

\subsection{Poisson convergence}\label{s_poisson}
For the rest of Section~\ref{s_appearance}, we continue assuming that the contact process is obtained from a graphical construction with inane marks, and continue omitting the superscript~$\Z^d_N$, so that
 \[\zeta_t = \zeta^{\Z^d_N}_t,\qquad \tilde{T}_N = \tilde{T}_N^{\Z^d_N},\qquad \tilde{\mathscr{L}}_N =\tilde{\mathscr{L}}_N^{\Z^d_N}.\]

In view of Lemma~\ref{lem_reduction_full_occ} and Lemma~\ref{lem_after_in_G}, in order to prove Proposition~\ref{prop_appearance_new}, it suffices to prove that for any~$t > 0$ and any local function~$f:\{0,1\}^{\Z^d} \to \mathbb{R}$, we have
\begin{equation}\label{eq_conv_tilde_new}
	\sum_{B \subseteq \Z^d_N} \mathcal{U}_{N,t}(\lambda;\Z^d_N,B)\cdot f(\psi_N^{-1}(B))  \xrightarrow{N\to \infty} (1-\mathrm{e}^{-b(\lambda)R(\lambda)\cdot t})\cdot \int f\;\mathrm{d}\hat{\mu}_\lambda.
\end{equation} 
 This is the same as
	\begin{equation}\label{eq_conv_tilde}
		\mathbb{E}\left[\mathds{1}\{\delta_NN^d\cdot \tilde{T}_N\le  t\}\cdot f(\psi^{-1}_N(\tilde{\mathscr{L}}_N))\right] \xrightarrow{N \to \infty} (1-\mathrm{e}^{-b(\lambda)R(\lambda)\cdot t})\cdot \int f\;\mathrm{d}\hat{\mu}_\lambda.
	\end{equation}

Define the process~$(\mathcal{N}_t^{(f)})_{t \ge 0}$ by letting~$\mathcal{N}_0^{(f)} = 0$ and
	\begin{equation}\label{eq_def_jumpsum} \mathcal{N}^{(f)}_t := \sum_{u \in \Z^d_N}\sum_{s \le t} \mathds{1}\{\zeta_{s-}(u)=0,\; \zeta_{s}(u)= 1\}\cdot f(\psi_N^{-1}(\zeta_{s-} \circ \theta_u)).\end{equation}
		In words, at every time for which there is a birth in~$(\zeta_t)$, the process~$(\mathcal{N}^{(f)}_t)$ is incremented by the value of~$f$  evaluated at the configuration obtained by taking the contact process configuration just before the birth, shifting it so that the birth location becomes the origin, and then transporting it to~$\Z^d$.

Define~$q:\{0,1\}^{\Z^d} \to \mathbb{R}$ by
\begin{equation*}
	q(\zeta):= \mathds{1}\{\zeta(o) = 0\}\cdot |\{x \sim o:\;\zeta(x) = 1\}|.
\end{equation*}
Note that, recalling~\eqref{eq_def_of_R},\color{black}
\begin{equation}
\label{eq_def_R}
	R(\lambda) = \lambda\cdot \int_{\{0,1\}^{\Z^d}} q(\zeta)\; \mu_\lambda(\mathrm{d}\zeta)
\end{equation}
and, for any~$f:\{0,1\}^{\Z^d} \to \R$ that is measurable and bounded,
\begin{equation}
	\label{eq_radon}
\int_{\{0,1\}^{\Z^d}} f(\zeta)\; \hat{\mu}_\lambda(\mathrm{d}\zeta) = \frac{\lambda}{R(\lambda)} \cdot \int_{\{0,1\}^{\Z^d}} f(\zeta)\cdot q(\zeta)\; {\mu}_\lambda(\mathrm{d}\zeta). \end{equation}

We will need the following lemma, whose proof is postponed to Section~\ref{ss_concentration}.
		\begin{lemma}\label{lem_jump_sum} For any local function~$f:\{0,1\}^{\Z^d} \to \mathbb{R}$ and any sequence~$(t_N)_{N \ge 1}$  such that $  t_N/N^{1/4} \xrightarrow{N \to \infty} \infty$ and~$t_N/N^{\log N} \xrightarrow{N \to \infty} 0$, we have
	\begin{equation*}
		\frac{1}{N^dt_N}\cdot  \mathcal{N}^{(f)}_{t_N} \xrightarrow{N \to \infty} \lambda\int f\cdot q\;\mathrm{d}\mu_\lambda \quad \text{in probability}.
	\end{equation*}
\end{lemma}

\begin{proof}[Proof of Proposition~\ref{prop_appearance_new}]
	As already observed, it  is sufficient \color{black} to prove~\eqref{eq_conv_tilde}.  In order to do so,  it suffices by taking linear combinations \color{black} to consider functions of the form
	\[f(\zeta) = \mathds{1}\{\zeta \cap Q_{\Z^d}(o,\ell) = \bar{B}\},\quad \zeta \in \{0,1\}^{\Z^d}\]
	for fixed~$\ell \in \N$ and~$\bar{B} \subseteq Q_{Z^d}(o,\ell)$. The left-hand side of~\eqref{eq_conv_tilde} then becomes
\begin{equation} \label{eq_conv_tilde2}
	\P(\delta_NN^d\cdot \tilde{T}_N \le  t,\; \psi_N^{-1}(\tilde{\mathscr{L}}_N) \cap Q_{\Z^d}(o,\ell) = \bar{B}).
\end{equation}
	The choice of~$\ell$ and~$\bar{B}$ (and hence the function~$f$) will remain fixed for the rest of this proof. 
We observe that
	\[1-f(\zeta) = \mathds{1}\{\zeta \cap Q_{\Z^d}(o,\ell) \neq \bar{B}\}.\]
Moreover, for each~$(u,t) \in \Z^d_N \times [0,\infty)$ for which~$\zeta_{t-}(u) = 0$ and~$\zeta_t(u) = 1$, we either have~$f(\psi^{-1}_N(\zeta_{t-} \circ \theta_u)) = 1$ (in which case~$\mathcal{N}^{(f)}_t = \mathcal{N}^{(f)}_{t-} + 1$ and~$\mathcal{N}^{(1-f)}_t = \mathcal{N}^{(1-f)}_{t-}$) or~$f(\psi^{-1}_N(\zeta_{t-} \circ \theta_u)) = 0$ (in which case~$\mathcal{N}^{(f)}_t = \mathcal{N}^{(f)}_{t-} $ and~$\mathcal{N}^{(1-f)}_t = \mathcal{N}^{(1-f)}_{t-}+1$). Define the times
\[\mathcal{J}_N:=  \inf\left\{\begin{array}{ll} 
	t \ge 0:\; \mathcal{N}^{(f)}_t = \mathcal{N}^{(f)}_{t-}+1,\text{ the birth arrow }\\ \text{that causes the birth at time $t$ has a mark}
\end{array}\right\}\]
and
\[\mathcal{J}_N':=  \inf\left\{\begin{array}{ll} 
	t \ge 0:\; \mathcal{N}^{(1-f)}_t = \mathcal{N}^{(1-f)}_{t-}+1,\text{ the birth arrow }\\ \text{that causes the birth at time $t$ has a mark}
\end{array}\right\}.\] 
Intuitively speaking, whether an arrow is added to~$\mathcal{J}_N$ or~$\mathcal{J}_N'$ depends on whether at the corresponding time $t$, the intersection of the configuration $\zeta_t$ (translated with respect to the place where the birth happened, transported to~$\Z^d$) with $Q_{\mathbb Z^d}(o,\ell)$ equals $\bar{B}$ or not. \color{black}
Note that~$\tilde{T}_N =  \min(\mathcal{J}_N,\mathcal{J}'_N)$, and the probability in~\eqref{eq_conv_tilde2} becomes
\begin{equation} \label{eq_conv_tilde3}\P(\delta_NN^d\cdot \mathcal{J}_N \le   t,\; \mathcal{J}_N < \mathcal{J}_N').\end{equation}
	Rather than treating this probability directly, we first fix~$s,s' > 0$ and consider
\begin{align} 
\nonumber	&\P(\mathcal{J}_N >(\delta_NN^d)^{-1}\cdot s,\;\mathcal{J}'_N > (\delta_NN^d)^{-1}\cdot s') \\[.2cm]\nonumber&= \E\left[ \P\left( \mathcal{J}_N >(\delta_NN^d)^{-1}\cdot s,\;\mathcal{J}'_N > (\delta_NN^d)^{-1}\cdot s'\;\left|\; (\mathcal{N}^{(f)}_t,\mathcal{N}^{(1-f)}_t)_{t \ge 0}\right.\right)\;\right]\\[.2cm]
	\nonumber&= \E\left[ (1-\delta_N\cdot b(\lambda))^{\mathcal{N}^{(f)}_{s/(\delta_NN^d)}+\mathcal{N}^{(1-f)}_{s'/(\delta_NN^d)}}\right]\\[.2cm]
	\label{eq_mess22}&= \E\left[ \exp\left\{\frac{\log(1-\delta_N\cdot b(\lambda))}{\delta_N}\cdot\delta_N \cdot   \left( \mathcal{N}^{(f)}_{s/(\delta_NN^d)} + \mathcal{N}^{(1-f)}_{s'/(\delta_NN^d)}  \right) \right\}\right].
\end{align}
Now, defining
	\[\alpha := \lambda \cdot \int f\cdot q\;\mathrm{d}\mu_\lambda,\qquad \alpha' := \lambda \cdot \int (1-f)\cdot q\;\mathrm{d}\mu_\lambda,\]
	Lemma~\ref{lem_jump_sum} (with~$t_N = s/(\delta_NN^d)$, which satisfies~$t_N/N^{1/4} \to \infty$ by~\eqref{eq_assumption_delta}) implies that
\[ \delta_N \cdot \mathcal{N}^{(f)}_{s/(\delta_NN^d)} = \frac{\mathcal{N}^{(f)}_{s/(\delta_NN^d)}}{N^d\cdot s/(\delta_NN^d)}\cdot s \xrightarrow{N\to \infty} \alpha\cdot s\quad \text{in probability} \]
and similarly,
\[ \delta_N \cdot \mathcal{N}^{(1-f)}_{s'/(\delta_NN^d)}  \xrightarrow{N\to \infty} \alpha' \cdot s' \quad \text{in probability}. \]
Using these convergences in~\eqref{eq_mess22}, we obtain that
\[\P(\mathcal{J}_N >(\delta_NN^d)^{-1}\cdot s,\;\mathcal{J}'_N > (\delta_NN^d)^{-1}\cdot s')\xrightarrow{N \to \infty} \mathrm{e}^{-\alpha b(\lambda) s - \alpha'b(\lambda)s'}.\]
This implies that
\[(\delta_NN^d\cdot \mathcal{J}_N,\;\delta_NN^d\cdot \mathcal{J}'_N) \xrightarrow{N \to \infty} (Z,Z')\quad \text{in distribution,}\]
	where~$Z,Z'$ are independent random variables, with~$Z\sim \mathrm{Exp}(\alpha b(\lambda))$ and~$Z'\sim \mathrm{Exp}(\alpha'b(\lambda))$. Then, the probability in~\eqref{eq_conv_tilde3} converges, as~$N \to \infty$, to
\[\P(Z \le t,\;Z < Z') = (1- \mathrm{e}^{-b(\lambda)(\alpha+\alpha')\cdot t})\cdot \frac{\alpha}{\alpha + \alpha'}. \]
By~\eqref{eq_def_R} and~\eqref{eq_radon}, we have~$\alpha + \alpha' = R(\lambda)$ and
\[(1- \mathrm{e}^{-b(\lambda)(\alpha+\alpha')\cdot t})\cdot \frac{\alpha}{\alpha + \alpha'} = (1-\mathrm{e}^{-b(\lambda)R(\lambda)\cdot t})\cdot \int f\;\mathrm{d}\hat{\mu}_\lambda,\]
as required.
\end{proof}

{
An interesting consequence of Lemma~\ref{lem_jump_sum} is that~$R(\lambda) \le 1$, which is the content of the next proposition. Although we do not need this fact for the proof of our main result, we find it to be of independent interest. Moreover, recall that the limiting process~$(Z_t)_{t \ge 0}$ of Theorem~\ref{thm_main} jumps from~$\lambda > 0$ with rate~$b(\lambda)R(\lambda)$; knowing that~$R(\lambda) \le 1$ tells us that the assumption that this process is non-explosive is not too restrictive (see Remark~\ref{rem_explosive_b}). 

\begin{proposition}\label{prop_leo}
	The quantity~$R(\lambda) =\lambda \cdot  \int_{\{0,1\}^{\Z^d}} q \;\mathrm{d}\mu_\lambda$ is at most~$1$.
\end{proposition}
\begin{proof}
	We take~$f:\{0,1\}^{\Z^d} \to \R$ to be identically equal to one, so the process~$(\mathcal{N}^{(f)}_t)_{t \ge 0}$ defined in~\eqref{eq_def_jumpsum} simply counts the number of births in~$(\zeta_t)_{t \ge 0}$. Lemma~\ref{lem_jump_sum} with~$t_N = N$ implies that
	\begin{equation}\label{eq_conv_for_N}
		\frac{1}{N^{d+1}} \mathcal{N}^{(f)}_{N} \xrightarrow{N \to \infty} \lambda \int_{\{0,1\}^{\Z^d}} q\;\mathrm{d}\mu_\lambda =  R(\lambda)\quad \text{in probability.}
	\end{equation}
	Next, define
	\[\mathcal{M}_t := \sum_{u \in \Z^d_N} \sum_{s \le t} \mathds{1}\{\zeta_{s-}(u) = 1,\;\zeta_s(u) = 0\},\]
	that is,~$\mathcal{M}_t$ is the number of times that a shift~$1 \to 0$ happens anywhere in the torus up to time~$t$. Clearly, for any~$t$,
	\[ \sum_{u \in \Z^d_N} \zeta_t(u) = \sum_{u \in \Z^d_N} \zeta_0(u) - \mathcal{M}_t + \mathcal{N}_t^{(f)} = N^d - \mathcal{M}_t + \mathcal{N}^{(f)}_t,\]
	so~$|\mathcal{M}_t - \mathcal{N}_t^{(f)}| \le 2N^d$, so~\eqref{eq_conv_for_N} gives
	\begin{equation}\label{eq_conv_for_M}
		\frac{1}{N^{d+1}} \mathcal{M}_{N} \xrightarrow{N \to \infty} \lambda \int_{\{0,1\}^{\Z^d}} q\;\mathrm{d}\mu_\lambda =  R(\lambda)\quad \text{in probability.}
	\end{equation}
	On the other hand,~$\mathcal{M}_N$ is at most the number of death marks anywhere in the torus in the time interval~$[0,N]$. Since death marks at each vertex are given by a Poisson process with rate 1, the Law of Large Numbers gives that, for any~$\varepsilon > 0$,
	\[\P\left( \frac{1}{N^{d+1}}\mathcal{M}_N \le 1 + \varepsilon \right) \xrightarrow{N \to \infty} 1.\]
	Comparing this with~\eqref{eq_conv_for_M} implies that~$R(\lambda) \le 1$.
\end{proof}

}

\subsection{Martingale concentration}\label{ss_concentration}
In this section, we prove Lemma~\ref{lem_jump_sum}. The first ingredient we will need is the following elementary inequality involving continuous-time Markov chains.
\begin{lemma}\label{lem_elem_mc}
	Let~$(Z_t)_{t \ge 0}$ be a continuous-time Markov chain on the finite state space~$S$, with jump rates~$r(\cdot,\cdot):S^2 \to [0,\infty)$. Let~$h: S^2 \to \mathbb{R}$ be a function with~$h(a,a) = 0$ for all~$a \in S$. Define processes~$(A_t)_{t \ge 0}$ and~$(B_t)_{t \ge 0}$ by setting~$A_0 = B_0 = 0$ and
	\[A_t:= \sum_{s \le t} h(Z_{s-},Z_s),\qquad B_t:= \int_0^t \sum_{b \in S} r(Z_s,b)\cdot h(Z_s,b)\;\mathrm{d}s,\quad t > 0.\]
Then, for any~$t > 0$ and~$x > 0$ we have
	\[\mathbb{P}\left(\sup_{s \le t}|A_s - B_s| > x\right) \le 4 \max_{a \in S}\left(\sum_{b \in S} r(a,b)\cdot h(a,b)^2\right)\cdot t \cdot \frac{1}{x^2}.\]
\end{lemma}

\begin{proof}
	{An application of Dynkin's formula [Lemma 17.22 \cite{Kallenberg}] makes it easy to see that ~$(M_t)_{t \ge 0}:= (A_t - B_t)_{t \ge 0}$ is a martingale. It is elementary to check that it has} quadratic variation
	\[\langle M\rangle_t = \int_0^t \sum_{b \in S} r(Z_s,b)\cdot h(Z_s,b)^2\;\mathrm{d}s.\]
By Markov's inequality and Doob's inequality,
	\[ \P\left(\sup_{s \le t} |M_s| > x \right) \le \frac{\E\left[\sup_{s \le t}(M_s)^2\right]}{x^2} \le \frac{4\cdot \E[(M_t)^2]}{x^2}   = \frac{4\cdot \E[\langle M\rangle_t]}{x^2}.\]
Finally, we bound
	\[\E[\langle M \rangle_t] \le \max_{a \in S}\left(\sum_{b \in S} r(a,b)\cdot h(a,b)^2\right)\cdot t.\]
\end{proof}

Our second ingredient is a convergence result for the integral that will appear in the application we have in mind for Lemma~\ref{lem_elem_mc}. The proof is postponed to Section~\ref{ss_compensator}.

\begin{lemma}\label{lem_compensator}  For any local function~$g:\{0,1\}^{\Z^d} \to \mathbb{R}$  and any sequence~$(t_N)_{N \ge 1}$  such that $  t_N/N^{1/4} \xrightarrow{N \to \infty} \infty$ and~$t_N/N^{\log N} \xrightarrow{N \to \infty} 0$, we have
	\begin{equation*}
		\frac{1}{N^dt_N} \sum_{u \in \Z^d_N} \int_0^{t_N} g(\psi_N^{-1}(\zeta_t \circ \theta_u))\;\mathrm{d}t \xrightarrow{N\to \infty} \int g \;\mathrm{d}\mu_\lambda \quad \text{in probability.}
	\end{equation*}
\end{lemma}
{
The reason for the choice~$N^{1/4}$ will become clear in Section~\ref{ss_compensator}; see the discussion preceding Lemma~\ref{lem_first_moment}.
}

\begin{proof}[Proof of Lemma~\ref{lem_jump_sum}]
	Fix~$f$ and~$(t_N)$ as in the statement of the lemma. We will apply Lemma~\ref{lem_elem_mc} to the process~$(\zeta_t)_{t \ge 0}$. We define the function~$h$ of pairs of configurations~$(\zeta,\zeta') \in (\{0,1\}^{\Z^d_N})^2$ as follows. In case~$\zeta$ and~$\zeta'$ only differ at one site~$u \in \Z^d_N$, with~$\zeta(u) = 0$ and~$\zeta'(u)=1$, we set
	\[h(\zeta,\zeta'):= f(\psi_N^{-1}(\zeta \circ \theta_u)).\]
	In any other case, we set~$h$ to be zero. Note that the process which is called~$(A_t)$ in the notation of Lemma~\ref{lem_elem_mc} is
	\[A_t = \sum_{s \le t}h(\zeta_{s-},\zeta_s) = \mathcal{N}^{(f)}_t.\]

	We now examine what we obtain as the process~$(B_t)$, in the notation of Lemma~\ref{lem_elem_mc}.  Let~$r(\zeta,\zeta')$ denote the rate at which~$(\zeta_t)$ jumps from~$\zeta$ to~$\zeta'$.  We have
	\begin{align*}\sum_{\zeta'} r(\zeta,\zeta')\cdot h(\zeta,\zeta') &=\sum_{u \in \Z^d_N}  \mathds{1}\{\zeta(u) = 0\}\cdot  \lambda\cdot |\{v \sim u:\zeta(v) = 1\}| \cdot  f(\psi^{-1}_N(\zeta\circ \theta_u))\\[.2cm]
		&=\lambda\cdot  \sum_{u\in\Z^d_N} q(\psi^{-1}_N(\zeta \circ \theta_u)) \cdot f(\psi^{-1}_N(\zeta \circ \theta_u)).
	\end{align*}
Thus,
	\[B_t =\lambda\cdot \sum_{u \in \Z^d_N} \int_0^t  q(\psi^{-1}_N(\zeta_s \circ \theta_u))\cdot f(\psi^{-1}_N(\zeta_s \circ \theta_u))\;\mathrm{d}s.\]

	As a last step before applying Lemma~\ref{lem_elem_mc}, we use the bounds $r(\zeta,\zeta') \le 2d\lambda$ and $h(\zeta,\zeta') \le \|f\|_\infty$ \color{black} to obtain
	\[\max_\zeta \sum_{\zeta'} r(\zeta,\zeta')\cdot h(\zeta,\zeta')^2 \le N^d\cdot 2d\lambda\cdot  \|f\|_\infty^2.\]

	Now, for any~$\varepsilon > 0$, the bound in Lemma~\ref{lem_elem_mc}  together with the fact that $\tn>N^{1/4}$ \color{black} gives
	\begin{align*}
		\P\left(\frac{1}{N^dt_N}|A_{t_N} - B_{t_N} | > \varepsilon \right)  \le \frac{4\cdot N^d\cdot 2d\lambda \cdot \|f\|_\infty^2}{\varepsilon^2\cdot N^{2d}\cdot t_N^2} \xrightarrow{N\to \infty}0.
	\end{align*}
	Finally, Lemma~\ref{lem_compensator} implies that~$\frac{ B_{t_N}}{N^dt_N}$ converges to~$\lambda\int(f\cdot q)\mathrm{d}\mu_\lambda$ in probability. With the above convergence, we conclude that~$\frac{A_{t_N}}{N^dt_N}$ converges to the same limit in probability.
\end{proof}

\subsection{Convergence of compensator process}\label{ss_compensator}
In this section, we prove Lemma~\ref{lem_compensator}.  It will be convenient to assume that the Poisson processes in the graphical construction~$H_{\Z^d_N}$ are defined also for negative times, and then we extend the definition of infection paths and the event~$(u,s) \rightsquigarrow (v,t)$ for~$s,t \in \mathbb{R}$,~$s \le t$ in the natural way. We then define the \textit{past-truncated contact process} by
\[\eta_t(u):= \mathds{1}\left\{\Z^d_N \times \{t - N^{1/4}\} \rightsquigarrow (u,t)\right\}\]
for~$t \ge 0$ and~$u \in \Z^d_N$. { In words,~$\eta_t(u)$ equals 1 if and only if there exists some infection path starting from somewhere in the torus at time~$t - N^{1/4}$ and ending at~$u$ at time~$t$. Our interest in this process stems from two facts. First, it is stationary in space and time, in the sense that the distribution of~$\eta_t \circ \theta_u$ does not depend on~$t$ or~$u$. Second, as the following lemma will show, this distribution is close to~$\mu_\lambda$ when~$N$ is large. In the sense that it allows us to sample (albeit approximately) from a stationary distribution by going backwards in time, this process is reminiscent of the \textit{coupling from the past} technique~\cite{PW96}.

The proximity between the law of~$\eta_t$ and~$\mu_\lambda$ is intuitively clear, as we now explain. The  occurrence of the event~$\{\Z^d_N \times \{t - N^{1/4}\} \rightsquigarrow (u,t)\}$ can with high probability be decided by only looking at the graphical construction inside~$Q_{\Z^d_N}(u,\sqrt{N})  \times [t-N^{1/4},t]$ (here,~$\sqrt{N}$ is chosen so that this box is of negligible size in comparison to the torus, and~$N^{1/4}$ is chosen so that the time component is negligible compared to the space component). In particular, the definition of the past-truncated process is not much affected by the fact that the process is happening inside a torus, rather than in infinite volume. Moreover, in infinite volume, the event~$\{\Z^d \times \{t - N^{1/4}\} \rightsquigarrow (u,t)\}$ is very close to the event~$\{-\infty \rightsquigarrow (u,t)\}$ of existence of an infection path from time~$-\infty$ ending at~$(u,t)$. Finally,~$(\mathds{1}\{-\infty \rightsquigarrow (u,t)\})_{u \in \Z^d}$ is distributed as~$\mu_\lambda$, by the definition of~$\mu_\lambda$ as the distributional limit of the contact process started from the identically 1 configuration.
}

\begin{lemma}\label{lem_first_moment}
	For any local function~$g:\{0,1\}^{\Z^d} \to \mathbb{R}$, we have
	\begin{equation*}\lim_{N \to \infty} \E[g(\psi_N^{-1}(\eta_0))] = \int g\;\mathrm{d}\mu_\lambda.\end{equation*}
\end{lemma}
\begin{proof} Let~$g:\{0,1\}^{\Z^d} \to \mathbb{R}$ be a local function, and fix~$\ell \in \N$ such that~$g(\zeta)$ only depends on~$\zeta \cap Q_{\Z^d}(o,\ell)$. 
	Using the same graphical construction~$H_{\Z^d_N}$ from which the past-truncated contact process is obtained, we define
	\begin{equation*}
			\tilde{\eta}_t(u):= \mathds{1}\left\{ (v,s) \rightsquigarrow (u,t) \text{ for some } (v,s) \notin Q_{\Z^d_N}(u,\sqrt{N}) \times [t-N^{1/4},t]\right\},\\
	\end{equation*}
	for~$t \ge 0$ and~$u \in \Z^d_N$.  Note that if the event~$\{\tilde{\eta}_t(u) = 1,\; \eta_t(u) = 0\}$ occurs, then~$(v,s) \rightsquigarrow (u,t)$ for some~$(v,s)$ such that~$s \in (t-N^{1/4},t]$ and~$v \notin Q_{\Z^d_N}(u,\sqrt{N})$.  By~\eqref{eq_constant_speed} and duality, we have that
	\[\P(\eta_0(o) \neq \tilde{\eta}_0(o))\xrightarrow{N \to\infty} 0.\]
Hence, also using a union bound and stationarity,
	\begin{equation}\label{eq_newfconv}
		\begin{split} &|\E[g(\psi_N^{-1}(\eta_0))] - \E[g(\psi_N^{-1}(\tilde{\eta}_0))]| \\ &\le 2\|g\|_\infty \cdot |Q_{\Z^d_N}(o,\ell)|\cdot \P(\eta_0(o) \neq \tilde{\eta_0}(o)) \xrightarrow{N \to \infty} 0.
			\end{split}
	\end{equation}

	Next, let~$H_{\Z^d}$ denote a graphical construction for the contact process with rate~$\lambda$ on~$\Z^d$ (assume that it is defined for both positive and negative times). Using~$H_{\Z^d}$, define
	\begin{equation*}
		\xi_t(x) := \mathds{1}\left\{-\infty \rightsquigarrow (x,t)\right\}
	\end{equation*}
	for~$t \ge 0$ and~$x \in \Z^d$, that is,~$\xi_t(x)$ is the indicator function of the event that for any $s<t$ there exists $y \in \Z^d$ such that $(y,s) \rightsquigarrow (x,t)$\color{black}. Note that~$(\xi_t)_{t \ge 0}$ is a stationary process: the distribution of~$\xi_t$ equals~$\mu_\lambda$ for every~$t$. In particular,
	\begin{equation}\label{eq_conv_compl1}
		\E[g(\xi_0)]=\int g\;\mathrm{d}\mu_\lambda.
	\end{equation}
Next, again using~$H_{\Z^d}$, define 
	\begin{equation*}
		\tilde{\xi}_{N,t}(x):= \mathds{1}\left\{(y,s) \rightsquigarrow (x,t) \text{ for some } (y,s) \notin Q_{\Z^d}(x,\sqrt{N}) \times [t-N^{1/4},t]   \right\}
	\end{equation*}
	for~$t \ge 0$ and~$x \in \Z^d$. We claim that, if~$N$ is large enough that~$\ell + \sqrt{N}\le N/4$, we have 
	\begin{equation}\label{eq_conv_compl2}
		\E[g(\psi_N^{-1}(\tilde{\eta}_0))] = \E[g(\tilde{\xi}_{N,0})].
	\end{equation}
	{
	To prove this, we start observing that, since~$g(\zeta)$ only depends on~$\zeta \cap Q_{\Z^d}(o,\ell)$, it suffices to prove that
	\begin{equation}
		\label{eq_Q_same_distr} \psi_N^{-1}(\tilde{\eta}_0) \cap Q_{\Z^d}(o,\ell) \qquad\text{ has same law as }\qquad \tilde{\xi}_{N,0} \cap Q_{\Z^d}(o,\ell).
	\end{equation}
	Note that
	\[\psi_N^{-1}(\tilde{\eta}_0) \cap Q_{\Z^d}(o,\ell)  = \psi_N^{-1}(\tilde{\eta}_0 \cap Q_{\Z^d_N}(o,\ell)).\]
Also observing that
	\[\bigcup_{u \in Q_{\Z^d_N}(o,\ell)}Q_{\Z^d_N}(u,\sqrt{N}) \subset Q_{\Z^d_N}(o,\sqrt{N}+\ell)\]
	and letting~$\mathcal{H}_N$ denote the restriction of~$H_{\Z^d_N}$ to~$Q_{\Z^d_N}(o,\sqrt{N}+\ell)$, we have
	\[\psi_N^{-1}(\tilde{\eta}_0 \cap Q_{\Z^d_N}(o,\ell)) = \left\{
		\begin{array}{l}
			\psi_N^{-1}(u):\; u \in Q_{\Z^d_N}(o,\ell):\; \text{in } \mathcal{H}_N,\; (v,s) \rightsquigarrow (u,0) \\[.2cm]
			\text{ for some } (v,s) \notin Q_{\Z^d_N}(u,\sqrt{N}) \times [-N^{1/4},0]
		\end{array}
		\right\}.\]
	We transport~$\mathcal{H}_N$ from~$Q_{\Z^d_N}(o,\sqrt{N}+\ell)$ to~$Q_{\Z^d}(o,\sqrt{N}+\ell)$ using~$\psi_N^{-1}$, letting~$\psi_N^{-1}(\mathcal{H}_N)$  denote the graphical construction thus obtained. Then, the right-hand side above equals
\[
	\left\{
		\begin{array}{l}
			x \in Q_{\Z^d}(o,\ell): \; \text{in } \psi_N^{-1}(\mathcal{H}_N), \; (y,s) \rightsquigarrow (x,0) \\[.2cm]
			\text{ for some } (y,s) \notin Q_{\Z^d}(x,\sqrt{N}) \times [-N^{1/4},0]
		\end{array}
		\right\}.
\]
	Using the fact that~$\psi_N^{-1}(\mathcal{H}_N)$ has the same law as the restriction of~$H_{\Z^d}$ to~$Q_{\Z^d}(o,\sqrt{N} + \ell)$, we see that the above has the same law as~$\tilde{\xi}_0(o) \cap Q_{\Z^d}(o,\ell)$, completing the proof of~\eqref{eq_Q_same_distr}, hence also of~\eqref{eq_conv_compl2}.
	}

	Now, by~\eqref{eq_conv_compl1} and~\eqref{eq_conv_compl2}, for~$N$ large we have that
	\begin{align*}
		\left| \E[g(\psi_N^{-1}(\tilde{\eta}_0))] - \int g\;\mathrm{d}\mu_\lambda \right|&= | \E[g(\tilde{\xi}_{N,0}) - g(\xi_0)]|\\&\le 2\|g\|_\infty \cdot |Q_{\Z^d}(o,\ell)| \cdot \P(\tilde{\xi}_{N,0}(o) \neq \xi_0(o)), 
	\end{align*}
	where the inequality follows from a union bound and stationarity.  
	{
	Note that the random variables~$\tilde{\xi}_{N,0}(o)$ are decreasing with~$N$,  and moreover,~$\xi_0(o) = 1$ if and only if~$\tilde{\xi}_{N,0}(o) = 1$ for every~$N$ (equivalently:~${\displaystyle \lim_{N \to \infty} \tilde{\xi}_{N,0}(o) = \xi_0(o)}$). This implies that the events~$E_N:= \{\tilde{\xi}_{N,0}(o) =1,\; \xi_0(o) = 0\}$ are decreasing in~$N$ and~$\cap_{N=1}^\infty E_N = \varnothing$. This gives
	}
	\[\P(\tilde{\xi}_{N,0}(o) \neq \xi(o))\xrightarrow{N \to \infty} 0.\]
	We thus have~$\E[g(\psi_N^{-1}(\tilde{\eta}_0))] \xrightarrow{N \to \infty} \int g\;\mathrm{d}\mu_\lambda$. 
	{Finally, we use Equation~\eqref{eq_newfconv} to observe that the result for $\tilde{\eta}_0$ is also true for ${\eta}_0$, which is the desired convergence.}
\end{proof}

\begin{corollary} For any sequence~$(t_N)_{N \ge 1}$  such that 
	\[t_N/N^{1/4} \xrightarrow{N \to \infty} \infty \quad \text{and} \quad  t_N/N^{\log N} \xrightarrow{N \to \infty} 0\]
	and any local function~$g:\{0,1\}^{\Z^d} \to \mathbb{R}$, we have
	\begin{equation*}
		\frac{1}{N^dt_N} \sum_{u \in \Z^d_N} \int_0^{t_N} g(\psi_N^{-1}(\eta_t \circ \theta_u))\;\mathrm{d}t \xrightarrow{N\to \infty} \int g \;\mathrm{d}\mu_\lambda\quad \text{in probability}.
	\end{equation*}
\end{corollary}
\begin{proof}
We abbreviate
	\[W_N(u,t):= g(\psi_N^{-1}(\eta_t \circ \theta_u)), \quad Y_N := \frac{1}{N^dt_N} \sum_{u \in \Z^d_N}\int_0^{t_N}W_N(u,t)\;\mathrm{d}t.\]
	By stationarity of~$(\eta_t)$ and Lemma~\ref{lem_first_moment}, we have that
	\[p_N := \E[W_N(u,t)] = \E[Y_N] = \E[g(\psi_N^{-1}(\eta_0))]\xrightarrow{N \to \infty} \int g\;\mathrm{d}\mu_\lambda.\]
  Moreover, we have
	\begin{align*}
		\mathrm{Var}(Y_N) &= \E[(Y_N)^2] - \E[Y_N]^2 \\[.2cm]
		&= \frac{1}{(N^dt_N)^2}\sum_{u \in \Z^d_N} \sum_{v \in \Z^d_N}\int_0^{t_N}\int_0^{t_N} \E[W_N(u,s)\cdot W_N(v,t)]\;\mathrm{d}s\;\mathrm{d}t - p_N^2\\[.2cm]
		&= \frac{1}{(N^dt_N)^2}\sum_{u \in \Z^d_N} \sum_{v \in \Z^d_N}\int_0^{t_N}\int_0^{t_N} \mathrm{Cov}(W_N(u,s),W_N(v,t))\;\mathrm{d}s\;\mathrm{d}t.
	\end{align*}
	We have  using the construction of~$(\eta_t)_{t \ge 0}$ that~$\mathrm{Cov}(W_N(u,s),W_N(v,t)) = 0$ if~$|s-t| > 2N^{1/4}$, \color{black} and we bound the covariance by~$4\|g_\infty\|^2$ otherwise; the right-hand side above is then at most
		\[ \frac{1}{(N^dt_N)^2}\sum_{u \in \Z^d_N} \int_0^{t_N} N^d \cdot 2N^{1/4}\cdot 4\|g\|_\infty\;\mathrm{d}s = \frac{8 N^{1/4} \cdot \|g\|_\infty}{t_N}\xrightarrow{N \to \infty} 0.\]
	Now, since~$\E[Y_N]$ converges to the desired limiting value and~$\mathrm{Var}(Y_N) \xrightarrow{N \to \infty}0$, the result follows from Chebyshev's inequality.
\end{proof}

Our next step is comparing the past-truncated contact process with the contact process started from full occupancy on~$\Z^d_N$. We start with the following.
\begin{lemma}\label{lem_two_proc_c0}
	For~$N$ large enough, the processes~$(\zeta_t)_{t \ge 0}$ (contact process on~$\Z^d_N$  started from full occupancy) and~$(\eta_t)_{t \ge 0}$ (past-truncated contact process on~$\Z^d_N$), obtained from the same graphical construction, satisfy
	\[\P(\zeta_t(o) \neq \eta_t(o))) < \exp\{-N^{1/80}\} \quad \text{for all } t \in [N^{1/4},N^{2\log N}].\]
\end{lemma}
The proof of this lemma is postponed to the Appendix. {
The constant~$1/80$ is somewhat arbitrary, in the sense that it comes from a sequence of choices of powers of~$N$, each smaller than the previous one, with no attempt of being sharp at each step -- this is done in the proofs of Lemmas~\ref{lem_numbers} and Lemma~\ref{lem_well_behaved} in the Appendix. } We now obtain a consequence of it.
\begin{lemma}\label{eq_two_proc_c}
	For~$N$ large enough, the processes~$(\zeta_t)_{t \ge 0}$ (contact process on~$\Z^d_N$  started from full occupancy) and~$(\eta_t)_{t \ge 0}$ (past-truncated contact process on~$\Z^d_N$), obtained from the same graphical construction, satisfy
	\[\P(\zeta_t = \eta_t \text{ for all } t \in [N^{1/4},N^{\log(N)}]) > 1- N^{-\log N}.\]
\end{lemma}
\begin{proof}
	Let
	\[\tau:= \inf\{t \ge N^{1/4}:\;\eta_t \neq \zeta_t\}.\]
	We observe that
	\[\mathrm{e}^{-(1+2d\lambda)}\cdot \P(\tau \le N^{\log(N)}) \le \P(\tau \le N^{\log(N)},\;\eta_t \neq \zeta_t \text{ for all } t \in [\tau, \tau+1]) ,\]
	where the exponential factor is obtained by prescribing that a vertex where a disparity exists at time~$\tau$ has no updates in the unit time interval~$[\tau,\tau+1]$. Next, note that the right-hand side is bounded from above by
	\begin{equation}\label{eq_com_t}\E\left[\int_{N^{1/4}}^{N^{\log(N)} + 1} \mathds{1}\{\eta_t \neq \zeta_t\}\;\mathrm{d}t\right] = \int_{N^{1/4}}^{N^{\log(N) }+1} \P(\eta_t \neq \zeta_t)\;\mathrm{d}t. \end{equation}
		By Lemma~\ref{lem_two_proc_c0} and duality, we have that
	\[\P(\eta_t(u) \neq \zeta_t(u)) < \exp\{-N^{1/80}\}.\]
	for any~$u \in \Z^d_N$ and~$t \in [N^{1/4},N^{\log(N)}+1]$. Combining this with a union bound, we see that the right-hand side of~\eqref{eq_com_t} is smaller than
	\[(N^{\log(N)}+1 - N^{1/4})\cdot N^d\cdot \mathrm{e}^{-N^{-1/80}}.\]
	We have thus proved that
	\[\P(\tau \le N^{\log(N)}) \le \mathrm{e}^{1+2d\lambda}\cdot (N^{\log(N)}+1 - N^{1/4})\cdot N^d\cdot \mathrm{e}^{-N^{-1/80}},\]
	so the result follows by taking~$N$ large enough.
\end{proof}

\begin{proof}[Proof of Lemma~\ref{lem_compensator}]
	The statement follows from Lemma~\ref{eq_two_proc_c} and the observation that, in the event that~$\zeta_t = \eta_t$  for all $t \in [ N^{1/4},N^{\log(N)}]$, we have
	\begin{align*}
		&\left|\frac{1}{N^dt_N} \sum_{u \in \Z^d_N} \int_0^{t_N} g(\psi_N^{-1}(\zeta_t \circ \theta_u))\;\mathrm{d}t- \frac{1}{N^dt_N} \sum_{u \in \Z^d_N} \int_0^{t_N} g(\psi_N^{-1}(\eta_t \circ \theta_u))\;\mathrm{d}t\right|\\[.2cm]
		&\hspace{7cm}\le \frac{N^d \cdot N^{1/4}\cdot 2\|g\|_\infty}{N^d \cdot t_N} \xrightarrow{N \to \infty}0.
	\end{align*}
 {In the second line we use that by assumption $(t_N)_{N \ge 1}$  is such that $  t_N/N^{1/4} \xrightarrow{N \to \infty} \infty$ }
\end{proof}

\section{Fixation of a mutant}\label{s_fixation}

\subsection{Competition in~$\Z^d$} \label{ss_competition_lattice}
We now start the work towards proving Propositions~\ref{prop_fixation_one_new},~\ref{prop_no_fixation} and~\ref{prop_fixation_three_new}.  We will mostly deal with the two-type contact process (on either~$\Z^d$ or~$\Z^d_N$), which will be denoted by~$(\xi_t)_{t \ge 0}$ throughout. We will generally denote the birth rates of types 1 and 2 by~$\lambda$ and~$\lambda'$, respectively, assuming throughout Section~\ref{ss_competition_lattice} that~$\lambda'>\lambda$\color{black}. Occasionally we will need to assume that the process is obtained from a graphical construction (with basic arrows with rate~$\lambda$ and extra arrows with rate~$\lambda'-\lambda$), as described in Section~\ref{ss_two_type_prelim}. 

We will abbreviate, for~$\xi \in \{0,1,2\}^{\Z^d}$, 
\[\xi^{(i)} := \{x:\xi(x) = i\}, \quad i\in\{1,2\}.\]
Let~$Q$ be a box of~$\Z^d$ or~$\Z^d_N$, and let~$\xi$ be a two-type contact process configuration on a set that contains~$Q$. We say that \textit{$Q$ is good for~$\xi$} if~$\xi^{(1)} \cap Q = \varnothing$ and additionally, letting~$r$ denote the radius of~$Q$, every sub-box of~$Q$ with radius~$r^{1/24}$ intersects~$\xi^{(2)}$. {
The exponent~$1/24$ is chosen for the proof of Lemma~\ref{lem_propagation} below, which is carried out in the Appendix; it is chosen as twice the constant~$12$ that appears in the statement of Proposition~\ref{prop_coupling}. Recall that in that proposition, inside a box of radius~$N$, a certain coupling is said to happen with high probability by time~$N^{12\varepsilon}$, provided that every sub-box of radius~$N^{\varepsilon}$ is occupied at time zero. In the proof of Lemma~\ref{lem_propagation}, we need this coupling in a box of radius of order~${\ell}$ and within time~$\sqrt{\ell}$, which requires sub-boxes of radius~$(\sqrt{\ell})^{1/12} = \ell^{1/24}$ to be occupied at time zero.
}

The following two lemmas can be obtained from the proof of the main result in~\cite{MPV20}; we give details in Section~\ref{s_second_a} in the Appendix.
\begin{lemma}[Appearance of good boxes]\label{lem_appearance} Assume that~$\lambda' > \lambda_c(\Z^d)$ and~$\lambda' > \lambda$. There exist~$\sigma_0 >0$ and~$\ell_0 > 0$ (depending on~$\lambda,\lambda'$) such that the following holds for all~$\ell \ge \ell_0$. For the two-type contact process~$(\xi_t)_{t \ge 0}$ with rates~$\lambda,\lambda'$ and~$\xi_0(o) = 2$, we have
	\[\P(Q_{\Z^d}(o,\ell) \text{ is good for } \xi_{\ell^2}) > \sigma_0.\]
\end{lemma}
{
We emphasize that the only assumption on the initial configuration in the above lemma is that the origin is occupied by a type 2 individual. The statement is the most interesting in the case where~$\xi_0(x) = 1$ for all~$x \in \Z^d \backslash \{o\}$.
}

\begin{lemma}[Propagation of good boxes]\label{lem_propagation} Assume that~$\lambda' > \lambda_c(\Z^d),\;\lambda' > \lambda$ and~$\varepsilon > 0$. There exists~$\ell_1 > 0$ (depending on~$\lambda,\lambda',\varepsilon$) such that the following holds for all~$\ell \ge \ell_1$. Let~$(\xi_t)_{t \ge 0}$ denote the two-type contact process on~$\Z^d$ with rates~$\lambda,\lambda'$. Then, for any~$\bar{t} \ge \ell^{1+\varepsilon}$ we have
	\begin{equation*}
		{ Q_{\Z^d}(o,\ell) \text{ good for } \xi_0 \; \Longrightarrow} \quad \P\left(\begin{array}{c}\text{the boxes }\{Q_{\Z^d}(x,\ell):\;\|x\|_\infty \le \ell\} \\[.1cm]\text{are all good for } \xi_{\bar{t}}\end{array}\right) > 1 -  \exp(-(\log \ell)^2).\end{equation*}	
In particular,
	\begin{equation*}
		{ Q_{\Z^d}(o,\ell) \text{ good for } \xi_0 \; \Longrightarrow} \quad \P(\xi_t^{(2)} \neq \varnothing \; \text{for all } t) \ge 1-\exp\{-(\log \ell)^2\}. 
	\end{equation*}
	\end{lemma}

	As a consequence, we now prove the following.
	\begin{lemma}\label{lem_alternatives} Assume that~$\lambda' > \lambda_c(\Z^d)$ and~$\lambda' > \lambda$. There exists~$\ell_2 > 0$ (depending on~$\lambda,\lambda'$) such that the following holds for all~$\ell \ge \ell_2$. For the two-type contact process~$(\xi_t)_{t \ge 0}$ on~$\Z^d$  with rates~$\lambda,\lambda'$ and started from any initial configuration, we have
	\begin{equation*}
		\P\left(\{\xi^{(2)}_{\ell^3} = \varnothing\}\cup \left\{\begin{array}{c} \exists x \in \Z^d:\; Q_{\Z^d}(x,\ell) \\[.1cm]\text{is good for } \xi_{\ell^3} \end{array} \right\}\right) > 1-2\exp\{-(\log \ell)^2\}.
	\end{equation*}
\end{lemma}
\begin{proof}
	Fix~$\lambda,\lambda'$ as in the statement. Let~$\ell_0=\ell_0(\lambda,\lambda')$ be the constant of Lemma~\ref{lem_appearance}, and, taking~$\varepsilon = 1$, let~$\ell_1=\ell_1(\lambda,\lambda',\varepsilon)$ be the constant of Lemma~\ref{lem_propagation}. Assume that~$\ell\ge \max(\ell_0,\ell_1)$; for now assume that~$\ell \in \N$.

	For~$i = 1,\ldots, \ell$, let~$E_i$ denote the event that there exists~$x \in \Z^d$ such that $Q_{\Z^d}(x,\ell)$ is good for~$\xi_{i\ell^2}$. By Lemma~\ref{lem_appearance}, we have that
	\[\P(E_i\mid \xi^{(2)}_{(i-1)\ell^2} \neq \varnothing) > \sigma_0.\]
	 Since~$\xi^{(2)}_t = \emptyset$ implies that~$\xi^{(2)}_{t'} = \emptyset$ for all~$t' \ge t$, \color{black} gives
	\[\P\left( (\cup_{i=1}^\ell E_i)^c \cap \{\xi^{(2)}_{\ell^3} \neq \varnothing\} \right) < (1-\sigma_0)^\ell.\]
	By Lemma~\ref{lem_propagation}, if~$E_i$ occurs for some~$i\le \ell$, then with high probability there will also be a good box at time~$\ell^3$, that is,
	\[\P\left(  E_\ell^c\cap(\cup_{i=1}^\ell E_i)\right) < \exp\{-(\log \ell)^2\}.\]
	Putting these inequalities together we obtain
	\[\P(\{\xi^{(2)}_{\ell^3} \neq \varnothing \}\cap E_\ell^c) < (1-\sigma_0)^\ell + \exp\{-(\log \ell)^2\}.\]
	By taking~$\ell$ large enough, the desired bound is now easily seen to hold (also for non-integer~$\ell$).
\end{proof}

	\subsection{Box approximation}\label{ss_box_approx}
	Let~$Q = Q(x,r)$ be a box in~$\Z^d$ or~$\Z^d_N$ (with~$N > r$), and~$t > 0$. Also let~$\lambda' > \lambda$ and let~$H$ be a graphical construction in the space-time set~$Q \times [0,t]$ for a two-type contact process with rates~$\lambda$ and~$\lambda'$. We say that~$Q \times [0,t]$ is \textit{insulated in~$H$} if every infection path in~$H$ (where we allow infection paths to follow both basic and extra arrows) is contained in a set of the form~$Q' \times [0,t]$, where~$Q'$ is a sub-box of~$Q$ with radius~$r/10$.
	\begin{remark}\label{rem_insulate}
		{
			The motivation for this definition, and the name `insulated', is as follows. Let~$B, B'$ be disjoint subsets of~$\Z^d$. Also let~$r \in \N$ and~$H$ be a graphical construction for the two-type contact process with rates~$\lambda,\lambda'$ (with~$\lambda' > \lambda$\color{black}) in~$\Z^d$. Define:	\begin{itemize}
				\item $(\xi_t)_{t \ge 0}$, the two-type contact process on~$\Z^d$ constructed from~$H$, started from~$\xi_0$ given by
					\[\xi_0(x) = \begin{cases} 1&\text{if } x \in B;\\ 2&\text{if } x \in B';\\0&\text{otherwise;}\end{cases}\]
					\item $(\tilde{\xi}_t)_{t \ge 0}$, the two-type contact process on~$Q_{\Z^d}(o,r)$ constructed from the restriction of~$H$ to~$Q_{\Z^d}(o,r)$, started from~$\tilde{\xi}_0$ given by
						\[\tilde{\xi}_0(x) = \begin{cases} 1&\text{if } x \in B \cap Q_{\Z^d}(o,r);\\ 2&\text{if } x \in B'\cap Q_{\Z^d}(o,r);\\0&\text{otherwise.}\end{cases}\]
	\end{itemize}
		Now, let~$t > 0$ and assume that~$Q_{\Z^d}(o,r) \times [0,t]$ is insulated in~$H$ (or more precisely, in the restriction of~$H$ to this space-time set). Then, for~$0 \le s \le t$, any discrepancy between~$\xi_s$ and~$\tilde{\xi}_s$ inside~$Q_{\Z^d}(o,r)$ is contained in the set $Q_{\Z^d}(o,r) \backslash Q_{\Z^d}(o,9r/10)$, so:
		\begin{align*} &\xi^{(1)}_s \cap Q_{\Z^d}(o,9r/10) = \tilde{\xi}^{(1)}_s \cap Q_{\Z^d}(o,9r/10),\\
		&\xi^{(2)}_s \cap Q_{\Z^d}(o,9r/10) = \tilde{\xi}^{(2)}_s \cap Q_{\Z^d}(o,9r/10),\quad \text{ for all } s \le t.\end{align*}  
		 The same observation naturally holds for a graphical construction in the torus and its restriction to a torus box. Figure~\ref{fig_insulation} illustrates these processes.
		}
	\end{remark}

\begin{figure}[htb]
\begin{center}
\setlength\fboxsep{0cm}
\setlength\fboxrule{0cm}
\fbox{\includegraphics[width=\textwidth]{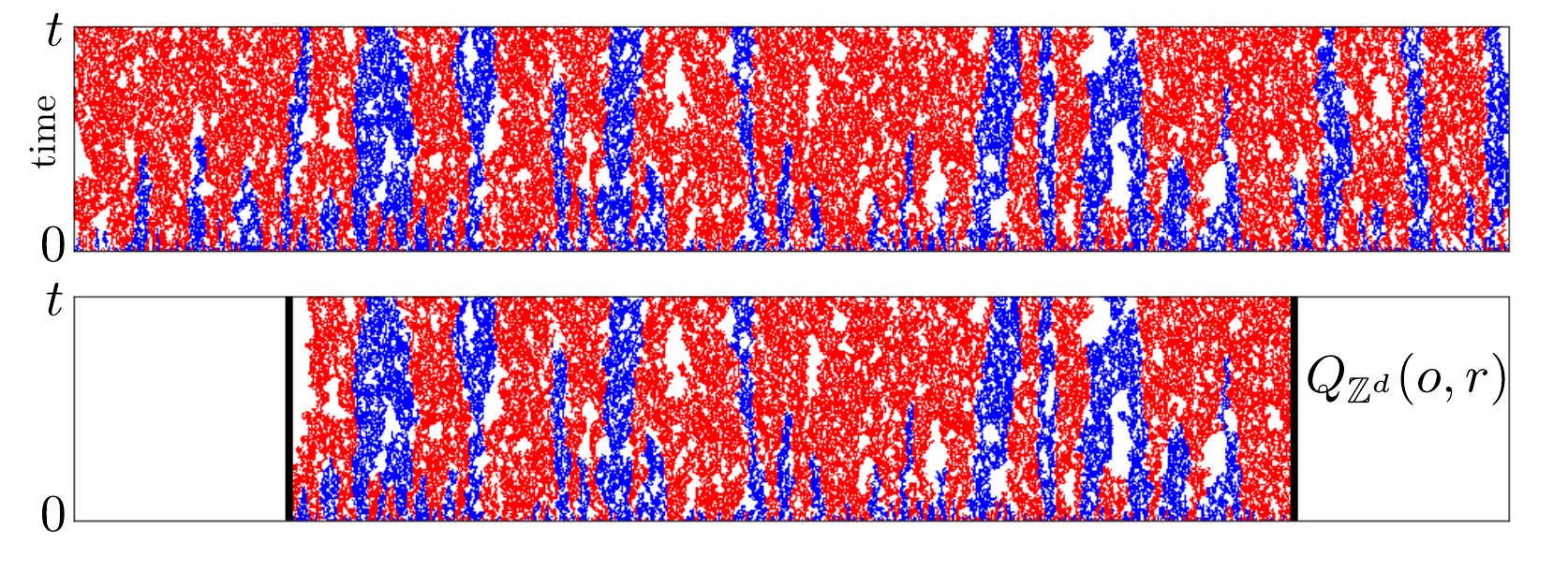}}
\end{center}
	\caption{\label{fig_insulation} {Simulation of the two-type contact process on $\mathbb{Z}^d$ (above) coupled with the two-type contact restricted to a box (below), with same initial configuration inside this box. Type 1 is depicted in blue and type 2 in red.}}
\end{figure}

With the above observation at hand, we will use space-time sets of the form~$Q \times [0,t]$ in order to transfer results we have obtained for the two-type process on~$\Z^d$ (namely, Lemma~\ref{lem_propagation} and Lemma~\ref{lem_alternatives}) to the two-type process on~$\Z^d_N$. This will depend on guaranteeing that insulation in~$Q \times [0,t]$ has high probability if~$t$ is small compared to the radius of~$Q$:
	\begin{lemma}\label{lem_insulation}
		For any~$\lambda,\lambda' \in (0,\infty)$ with~$\lambda \le \lambda'$, there exists~$r_0 > 0$ such that the following holds for any~$r \ge r_0$. Let~$Q = Q(x,r)$ be a box (either in~$\Z^d$ or in~$\Z^d_N$, with~$N > 2r$) and~$H$ be a graphical construction for a two-type contact process with rates~$\lambda,\lambda'$ in~$Q$. Then,~$Q \times [0,\sqrt{r}]$ is insulated in~$H$ with probability higher than~$1 - \exp\{-(\log r)^2\}$.
	\end{lemma}
	\begin{proof}
		For~$(y,s) \in Q \times [0,\sqrt{r}]$, let~$A(y,s)$ denote the event that there exists an infection path (which may use both basic and extra arrows)  that starts at~$(y,s)$ and intersects~$Q(y,r/10)^c \times [s,\sqrt{r}]$. By our speed bound~\eqref{eq_constant_speed}, we have
		\[\P(A(y,s)) < \exp\{-\sqrt{r}\}.\]
		Moreover, for any~$k \in \{0,\ldots, \lfloor \sqrt{r}\rfloor\}$, by imposing that there are no death marks at~$y$ in a time interval of length at most one, we have
		\[\P(A(y,k)) \ge \frac{1}{\mathrm{e}} \cdot \P\left( \bigcup_{s \in [k,(k+1)\wedge \sqrt{r}]} A(y,s)\right),\]
		so the probability on the right-hand side is at most~$\exp\{-\sqrt{r}+1\}$. The result now follows from a union bound over all~$y$ and all time intervals~$[0,1]$, $[1,2]$, $\ldots$, $[\lfloor \sqrt{r}\rfloor,\sqrt{r}]$, and taking~$r$ large enough.
	\end{proof}

	Fix~$\lambda,\lambda'$ with~$\lambda' > \lambda > \lambda_c(\Z^d)$. Also fix~$r \in \N$ and let~$B \subseteq Q_{\Z^d}(o,r)\backslash \{o\}$.  Let~$(\xi_t)_{t \ge 0}$ denote the two-type contact process on~$Q_{\Z^d}(o,r)$ where type 1 has rate~$\lambda$ and type~$2$ has rate~$\lambda'$, and initial configuration~$\xi_0$ given by
	\begin{equation}\label{eq_initial_same}\xi_0(x) = \begin{cases} 1&\text{if } x \in B;\\ 2&\text{if } x = o;\\ 0&\text{otherwise.}\end{cases}\end{equation}
	We define
	\begin{equation} \label{eq_def_of_calSr}
		S^{\mathrm{box}}_r(\lambda,\lambda';B) := \P(\xi^{(2)}_{\sqrt{r}} \neq \varnothing). 
	\end{equation}

\begin{lemma}[Box approximation to survival in~$\Z^d$]\label{lem_box_approx_Zd} Assume that~$\lambda'>\lambda > \lambda_c(\Z^d)$. There exists~$r_1 > 0$ such that for all~$r \ge r_1$, we have	
	\[\sup_{B \subseteq \Z^d\backslash \{o\}}|S^{\mathrm{box}}_r(\lambda,\lambda';B\cap Q_{\Z^d}(o,r)) - S(\lambda,\lambda';B)|< \exp\{-(\log r)^{3/2}\}.\]
\end{lemma}
	\begin{proof}
		Let~$\lambda,\lambda'$ be as in the statement. Fix~$r\ge r_0$, where~$r_0$ is the constant of Lemma~\ref{lem_insulation}; throughout the proof we will assume that~$r$ is large enough that other requirements are satisfied. Also fix~$B \subseteq \Z^d \backslash \{o\}$. Let~$(\xi_t)_{t \ge 0}$ be the two-type contact process on~$\Z^d$ with~$\xi_0$ as in~\eqref{eq_initial_same}, and let~$(\tilde{\xi}_t)_{t \ge 0}$ be the two-type contact process on~$Q_{\Z^d}(o,r)$, with $\tilde{\xi}_0$ equal to the restriction of~$\xi_0$ to~$Q_{\Z^d}(o,r)$,  obtained from the same graphical construction as~$(\xi_t)_{t \ge 0}$ (restricted to the box). {
Hence, these two processes are as in Figure~\ref{fig_insulation}.
  }

		Define the three ``bad events''
		\begin{align*}&\Xi_1:= \{\xi^{(2)}_{\sqrt{r}} \neq \tilde{\xi}^{(2)}_{\sqrt{r}}\},\\[.2cm]
			&\Xi_2 := \{\xi^{(2)}_{\sqrt{r}} \neq \varnothing\} \cap \{\nexists x:\; Q_{\Z^d}(x,r^{1/6})\text{ is good for }\xi_{\sqrt{r}}\},\\[.2cm]
			&\Xi_3 := \{\exists x:\;Q_{\Z^d}(x,r^{1/6}) \text{ is good for } \xi_{\sqrt{r}}\}\cap \{\exists t: \xi^{(2)}_t = \varnothing\}. 
		\end{align*}
Also letting~$\Xi:= \Xi_1 \cup \Xi_2 \cup \Xi_3$, we have the following equalities between events:
		\begin{align*}&\{\tilde{\xi}^{(2)}_{\sqrt{r}} \neq \varnothing\} \;\cap\; \Xi^c =  \{\xi^{(2)}_t \neq \varnothing\; \forall t\}\;\cap\; \Xi^c,\\[.2cm]
		& \{\tilde{\xi}^{(2)}_{\sqrt{r}} = \varnothing\} \;\cap\; \Xi^c = \{\exists t:\;\xi^{(2)}_t = \varnothing\}\;\cap\; \Xi^c.\end{align*}
		This implies that~$\{\tilde{\xi}^{(2)}_{\sqrt{r}} \neq \varnothing\} \triangle \{\xi^{(2)}_t \neq \varnothing \;\forall t\} \subseteq \Xi$, where~$\triangle$ denotes symmetric difference. Hence,
		\[|S(\lambda,\lambda';B) - S^{\mathrm{box}}_r(\lambda,\lambda';B \cap Q_{\Z^d}(o,r))| = |\P(\xi^{(2)}_t \neq \varnothing \;\forall t) - \P(\tilde{\xi}^{(2)}_{\sqrt{r}} \neq \varnothing)|\le \P(\Xi). \]
		We now bound the probabilities of the bad events. First, Lemma~\ref{lem_insulation} implies that~$\P(\Xi_1) < \exp\{-(\log r)^2\}$. Second, Lemma~\ref{lem_alternatives} implies that~$\P(\Xi_2) < 2\exp\{-(\log((r^{1/6}))^2\}$ if~$r^{1/6} \ge \ell_2$. Third, Lemma~\ref{lem_propagation}  implies that~$\P(\Xi_3) \le \exp\{-(\log(r^{1/6}))^2\}$ if~$r^{1/6} \ge \ell_1$. A union bound now concludes the proof.
	\end{proof}

	\begin{proof}[Proof of Proposition~\ref{prop_fixation_three_new}]
		Let~$r_1$ be the constant of Lemma~\ref{lem_box_approx_Zd}, corresponding to~$\lambda,\lambda'$. Fix~$\varepsilon > 0$ and take~$r_1' \ge r_1$ such {that~$\exp\left\{-(\log (r_1'))^{3/2}\right\} < \varepsilon/2$.} Then, for~$\ell \ge r_1'$, noting that~$Q_{\Z^d}(o,\ell) \cap Q_{\Z^d}(o,r_1') = Q_{\Z^d}(o,r_1')$, Lemma~\ref{lem_box_approx_Zd} { applied twice} gives
		\begin{align*}
			&\left| S(\lambda,\lambda';B) - S_r^{\mathrm{box}}(\lambda,\lambda';B \cap Q_{\Z^d}(o,r_1'))\right| < \varepsilon/2,\\
			&\left| S(\lambda,\lambda';B\cap Q_{\Z^d}(o,\ell)) - S_r^{\mathrm{box}}(\lambda,\lambda';B \cap Q_{\Z^d}(o,r_1'))\right| < \varepsilon/2,
		\end{align*}
		so the result readily follows.
	\end{proof}

\subsection{Competition in~$\Z^d_N$}\label{ss_competition_torus}
We now turn to the task of approximating the outcome of a competition in the torus using the survival-in-the-box probability~$S_r^\mathrm{box}(\lambda,\lambda';A)$ introduced earlier. 
{
We will have statements in the torus mimicking statements for~$\Z^d$: Lemma~\ref{lem_propagation_torus} below is the counterpart of Lemma~\ref{lem_propagation}, and Lemma~\ref{lem_alternatives_torus} below is the counterpart of Lemma~\ref{lem_alternatives}. 
In the torus, the radius of the sub-boxes to be deemed good will be~$N^{1/12}$. This is chosen because we want the time length corresponding to~$\ell^3$ in Lemma~\ref{lem_alternatives} to be~$N^{1/4}$ in the torus setting (which yields~$\ell =N^{1/12}$).
}
\begin{lemma}[Propagation of good boxes in~$\Z^d_N$]\label{lem_propagation_torus} For any~$\varepsilon > 0$ and any~$\lambda,\lambda'$ with~$\lambda' > \lambda_c(\Z^d)$ and~$\lambda'>\lambda$, there exists~$N_1 \in \N$ such that the following holds for all~$N \ge N_1$. Let~$(\xi_t)_{t \ge 0}$ denote a two-type contact process on~$\Z^d_N$ with birth rates $\lambda,\lambda'$. If~$Q_{\Z^d_N}(o,N^{1/12})$ is good for~$\xi_0$, letting~$\bar{s}:=N^{\frac{1}{12}+\varepsilon}$, we have 
	\begin{align}\label{eq_prob_prop_tor}
		\P\left( \begin{array}{c} \text{the boxes }\\ \{Q_{\Z^d_N}(u,N^{1/12}):u \in Q_{\Z^d_N}(o,N^{1/12})\}\\
  \text{ are all good for } \xi_{\bar{s}}\end{array}\right) > 1 - \exp\{-(\log N)^{3/2}\}.
	\end{align}
\end{lemma}
\begin{proof}
	%Let us first observe that, by a simple stochastic domination argument, it is sufficient to prove the statement of the lemma with the added condition that the set~$B_0$ is contained in~$Q_{\Z^d_N}(o,N^{1/8})$. 

	Fix~$N \in \N$ and let~$\xi_0 \in \{0,1,2\}^{\Z^d_N}$ be such that~$Q_{\Z^d_N}(o,N^{1/12})$ is good for~$\xi_0$. Let~$H$ be a graphical construction for a two-type contact process with rates~$\lambda,\lambda'$ in~$\Z^d_N$  and~$\tilde{H}$ be the restriction of~$H$ to~$Q_{\Z^d_N}(o,N^{1/2})$. We now define:
	\begin{itemize}
		\item $(\xi_t)_{t \ge 0}$, the two-type contact process on~$\Z^d_N$ constructed from~$H$ and started from~$\xi_0$;
		\item $(\tilde{\xi}_t)_{t \ge 0}$, the two-type contact process on~$Q_{\Z^d_N}(o,N^{1/2})$ constructed from~$\tilde{H}$, started from~$\tilde{\xi}_0$, the restriction of~$\xi_0$ to~$Q_{\Z^d_N}(o,N^{1/2})$.
	\end{itemize}

	We now lift~$\xi_0$ to~$\{0,1,2\}^{\Z^d}$ by defining~$\kappa_0$ by
	\[\kappa_0(x) = \begin{cases}1&\text{if } x \in \psi_N^{-1}(\xi_0^{(1)});\\[.1cm]2&\text{if } x \in \psi_N^{-1}(\xi_0^{(2)});\\[.1cm]0&\text{otherwise}.\end{cases}  \]
		We now take a graphical construction~$\mathcal{H}$ for a two-type contact process with rates~$\lambda,\lambda'$ on~$\Z^d$. Let~$\tilde{\mathcal{H}}$ be the restriction of~$\mathcal{H}$ to~$Q_{\Z^d}(o,N^{1/2})$. Define:
	\begin{itemize}
		\item $(\kappa_t)_{t \ge 0}$, the two-type  contact process on~$\Z^d$ constructed from~$\mathcal{H}$, started from~$\kappa_0$;
		\item $(\tilde{\kappa}_t)_{t \ge 0}$, the two-type contact process on~$Q_{\Z^d}(o,N^{1/2})$ constructed from~$\tilde{\mathcal{H}}$, started from~$\tilde{\kappa}_0$, the restriction of~$\kappa_0$ to~$Q_{\Z^d}(o, N^{1/2})$.
	\end{itemize}

	Now, Remark~\ref{rem_insulate} implies that, letting~$\bar{s} := N^{\frac{1}{12} + \varepsilon}$,
	\begin{align}\nonumber &\left\{ \begin{array}{c} \{Q_{\Z^d}(x,N^{1/12}):x \in Q_{\Z^d}(o,N^{1/12})\} \\ \text{are all good for } \kappa_{\bar{s}}  \end{array} \right\} \cap \left\{\begin{array}{c}Q_{\Z^d}(o,N^{1/2}) \times [0,\bar{s}]\\ \text{ is insulated in }\tilde{\mathcal{H}} \end{array}\right\}\\[.2cm]
		&=\left\{ \begin{array}{c} \{Q_{\Z^d}(x,N^{1/12}):x \in Q_{\Z^d}(o,N^{1/12})\} \\ \text{are all good for } \tilde{\kappa}_{\bar{s}}  \end{array} \right\} \cap \left\{\begin{array}{c}Q_{\Z^d}(o,N^{1/2}) \times [0,\bar{s}]\\ \text{ is insulated in }\tilde{\mathcal{H}} \end{array}\right\}.	\label{eq_inter_event}
\end{align}

This together with Lemma~\ref{lem_propagation} and Lemma~\ref{lem_insulation} gives that, for~$N$ large enough, the probability  of the event in~\eqref{eq_inter_event} is larger than~$1 - \exp\{-(\log N^{1/2})^2\}- \exp\{-(\log N^{1/12})^2\}$. Moreover, this probability only involves the graphical construction inside~$Q_{\Z^d}(o,N^{1/2})\times [0,\bar{s}]$, so it is the same as the probability  of
\[\left\{ \begin{array}{c} \{Q_{\Z^d_N}(u,N^{1/12}):u \in Q_{\Z^d_N}(o,N^{1/12}) \}\\ \text{are all good for } \tilde{\xi}_{\bar{s}} \end{array} \right\} \cap \left\{\begin{array}{c}Q_{\Z^d_N}(o,N^{1/2}) \times [0,\bar{s}]\\ \text{ is insulated in }{\tilde{H}} \end{array}.\right\}\]
	Finally note that, again by Remark~\ref{rem_insulate}, the event above is contained in the event inside the probability in~\eqref{eq_prob_prop_tor}.
\end{proof}

\begin{corollary}\label{cor_propagate_torus}
	For any~$\varepsilon > 0$ and any~$\lambda,\lambda'$ with~$\lambda' > \lambda_c(\Z^d)$ and~$\lambda'>\lambda$, there exists~$N_2 \in \N$ such that the following holds for all~$N \ge N_2$. Let~$(\xi_t)_{t \ge 0}$ denote a two-type contact process on~$\Z^d_N$ with birth rates $\lambda,\lambda'$. If~$Q_{\Z^d_N}(o,N^{1/12})$ is good for~$\xi_0$, then for any~$t \in [\tfrac34 N^{1+\varepsilon},\;N^{\log(N)}]$ we have
	\begin{equation*}\P(\xi^{(1)}_t = \varnothing,\;\xi^{(2)}_t \in \mathscr{G}_N) > 1- \exp\{-(\log N)^{5/4}\}.\end{equation*}
\end{corollary}
\begin{proof}
	Define~$U_0 := Q_{\Z^d_N}(o,N^{1/12})$ and recursively,
	\[U_{k+1}:= \bigcup_{u \in U_{k}} Q_{\Z^d_N}(u,N^{1/12}),\quad k \in \N_0.\]
	Note that~$U_k = \Z^d_N$ for~$k \ge k_\star:=\lfloor \tfrac23 N^{\frac{11}{12}}\rfloor$. Let~$\bar{s}:= N^{\frac{1}{12}+\varepsilon}$, and for~$k \in \N$, define
	\[E_k := \left\{ \text{the boxes } \{Q_{\Z^d_N}(u,N^{1/12}): u \in U_k\} \text{ are good for } \xi_{k\bar{s}}\right\}.\]
	We bound
	\begin{align*} \P((E_{k_\star})^c) &\le \P((E_1)^c) + \sum_{k = 0}^{k_\star - 1} \P((E_{k+1})^c\mid E_k).\end{align*}
		By Lemma~\ref{lem_propagation_torus} (combined with a union bound), we have that
		\[\P((E_1)^c) \le \exp\{-(\log N)^{3/2}\}\qquad \P((E_{k+1})^c|E_k) \le   N^d\cdot \exp\{-(\log N)^{3/2}\},\]
so
	\[\P((E_{k_\star})^c) \le N^{d+1}\cdot \exp\{-(\log N)^{3/2}\}.\]

	Next, let~$\bar{t}:= k_\star\cdot \bar{s}$. Note that~$\bar{t} \le \tfrac23 N^{1+\varepsilon}$. On the event~$E_{k_\star}$, we have that~$\xi_{\bar{t}}^{(1)} = \varnothing$ (so~$\xi_t^{(1)} = \varnothing$ for all~$t \ge \bar{t}$) and~$\xi_{\bar{t}}^{(2)}$ intersects all boxes of radius~$N^{1/288}$($=N^{1/12}N^{1/24}$) \color{black} in~$\Z^d_N$ (so~$\xi_{\bar{t}}^{(2)} \in \mathscr{G}_N$). Hence, on~$E_{k_\star}$,~$(\xi^{(2)}_t)_{t \ge \bar{t}}$ behaves as a one-type contact process whose initial configuration belongs to~$\mathscr{G}_N$. Since~$\tfrac34 N^{1+\varepsilon} - \bar{t} \ge (\tfrac34 - \tfrac23)\cdot N^{1+\varepsilon}$, we can apply~\eqref{eq_coupling_combining}  to obtain, for any~$t \in [\tfrac34 N^{1+\varepsilon}, \;N^{\log(N)}]$,
	\[\P(\xi^{(2)}_t \in \mathscr{G}_N \mid  E_{k_\star}) < 2\exp\{-(\log N)^2\} .\]
\end{proof}

Next, we have the following counterpart of Lemma~\ref{lem_alternatives}.
\begin{lemma}\label{lem_alternatives_torus}
	For any~$\lambda,\lambda' > 0$ with~$\lambda' > \lambda_c(\Z^d)$ and~$\lambda' > \lambda$, there exists~$N_3 \in \N$ such that the following holds for all~$N \ge N_3$. For the two-type contact process~$(\xi_t)_{t \ge 0}$ on~$\Z^d_N$  with rates~$\lambda,\lambda'$ and started from any initial configuration, we have
	\begin{equation*}
		\P\left(\{\xi^{(2)}_{N^{1/4}} = \varnothing\}\cup \left\{\begin{array}{c} \exists u \in \Z^d_N:\; Q_{\Z^d_N}(u,N^{1/12}) \\\text{is good for } \xi_{N^{1/4}} \end{array} \right\}\right) > 1-{ \exp\{-(\log N)^{3/2}\}}.
	\end{equation*}
\end{lemma}
{
This can be proved with the same approach used in the proof of Lemma~\ref{lem_propagation_torus}, so we only sketch the argument.
\begin{proof}[Sketch of proof.]
    We divide the time interval~$[0,N^{1/4}]$ into~$N^{1/12}$ intervals of equal length~$N^{1/6}$.
    For~$i=1,\ldots,N^{1/12}$, let~$E_i$ be the event that there is some~$u\in \Z^d_N$ such that~$Q_{\Z^d_N}(u,N^{1/12})$ is good for~$\xi_{iN^{1/6}}$. Combining Lemma~\ref{lem_appearance} with box insulation, we have
    \[\P(E_{i+1} \mid \xi_{iN^{1/6}}^{(2)} \neq \varnothing) > \sigma_0\]
    if~$N$ is large enough. This gives
    \[\P\left(  (\cup_{i=1}^{N^{1/12}}  E_i)^c  \cap \{\xi^{(2)}_{N^{1/4}} \neq \varnothing\} \right) < (1-\sigma_0)^{N^{1/12}}.\]
    Next, for some~$i \le N^{1/12}$, condition on the event that~$E_i$ occurs and let~$u \in \Z^d_N$ be such that~$Q_{\Z^d_N}(u,N^{1/12})$ is good for~$\xi_{iN^{1/6}}$.
    Combining Lemma~\ref{lem_propagation} with box insulation, we obtain that with (conditional) probability above~$1 - \exp\{-(\log(N^{1/12}))^2\} - \exp\{-(\log(N^{1/2}))^2\}$, there will also be a good box at time~$N^{1/4}$ (the box insulation here involves the box~$Q_{\Z^d_N}(u,\sqrt{N}) \times [iN^{1/6},\;N^{1/4}]$). The argument is then completed as in the proof of Lemma~\ref{lem_alternatives}.
\end{proof}
}

\begin{lemma}\label{lem_box_approx_ZdN} For any~$\lambda,\lambda' > 0$ with~$\lambda' > \lambda > \lambda_c(\Z^d)$, there exists~$N_4 \in \N$ such that the following holds for all~$N \ge N_4$. Let~$(\xi_t)_{t \ge 0}$ denote the two-type contact process on~$\Z^d_N$ with rates~$\lambda,\lambda'$ started from~$\xi_0$ given by
	\[\xi_0(u) = \begin{cases} 1& \text{if } u \in B;\\2 &\text{if } u = o;\\ 0&\text{otherwise,}\end{cases}\]
	where~$B \in \mathscr{G}_N$,~$o \notin B$. Then, 
 { for~$\tn = N^{1+\frac{\varepsilon_0}{2}}$ as in~\eqref{eq_def_of_tn},}
	\[ \left| \P(\xi^{(1)}_{\tn} = \varnothing,\; \xi^{(2)}_{\tn} \in \mathscr{G}_N) - S(\lambda,\lambda';\psi^{-1}_N(B))\right| < 2 \exp\{-(\log N)^{5/4}\}\]
	and
	\[\left| \P(\xi^{(1)}_{\tn} \in \mathscr{G}_N,\;\xi^{(2)}_{\tn} = \varnothing) - (1- S(\lambda,\lambda';\psi^{-1}_N(B) )  )\right|< 2 \exp\{-(\log N)^{5/4}\}.\]
\end{lemma}

\begin{proof}
	Fix~$N$, to be assumed to be large enough, and~$B \in \mathscr{G}_N$,~$o \notin B$.  Let~$(\xi_t)_{t \ge 0}$ denote the two-type process on~$\Z^d_N$, { with rates~$\lambda,\lambda'$} as in the statement, and~$(\tilde{\xi}_t)_{t \ge 0}$ the two-type process on $Q_{\Z^d_N}(o,N^{1/2})$ started from~$\tilde{\xi}_0$ given by
	\[\tilde{\xi}_0(u) = \begin{cases} 1& \text{if } u \in B \cap Q_{\Z^d_N}(o,N^{1/2});\\2 &\text{if } u = o;\\ 0&\text{otherwise,}\end{cases}\]
	obtained from the same graphical construction as~$(\xi_t)_{t \ge 0}$ (restricted to the box). Note that
	\begin{equation} \label{eq_aux_SZdn}  \P(\tilde{\xi}^{(2)}_{N^{1/4}} \neq \varnothing)= S^{\mathrm{box}}_{N^{1/2}}(\lambda,\lambda';B\cap Q_{\Z^d_N}(o,N^{1/2})).\end{equation}

	Define the four ``bad events''
	\begin{align*}
		&\Xi_1:= \{\xi_{N^{1/4}} \text{ and } \tilde{\xi}_{N^{1/4}} \text{ are not equal inside } Q_{\Z^d_N}(o,\tfrac12N^{1/2})\},\\[.2cm]
		&\Xi_2:= \{\xi^{(2)}_{N^{1/4}} \neq \varnothing\} \cap \{\nexists u:\;Q_{\Z^d_N}(u,N^{1/{12}}) \text{ is good for } \xi_{N^{1/4}}\},\\[.2cm]
		&\Xi_3 := \{\exists u:\;Q_{\Z^d_N}(u,N^{1/{12}}) \text{ is good for } \xi_{N^{1/4}}\} \cap \{\xi^{(2)}_{\tn} = \varnothing\},\\[.2cm]
		&\Xi_4 := \{\{u: X_{N,\tn}(u) \neq 0\} \notin \mathscr{G}_N\},
	\end{align*}
 {where~$\tn$ is again as in~\eqref{eq_def_of_tn}.}
Also let~$\Xi := \Xi_1 \cup \Xi_2 \cup \Xi_3 \cup \Xi_4$. We then have
	\begin{align*}
		&\{\tilde{\xi}_{N^{1/4}}^{(2)} \neq \varnothing\} \;\cap\; \Xi^c = \{\xi_{\tn}^{(1)} = \varnothing,\;\xi^{(2)}_{\tn} \in \mathscr{G}_N\}\; \cap\; \Xi^c,\\[.2cm]
		&\{\tilde{\xi}_{N^{1/4}}^{(2)}= \varnothing\}\; \cap\; \Xi^c = \{\xi^{(1)}_{\tn} \in \mathscr{G}_N,\;\xi_{\tn}^{(2)} = \varnothing\}\; \cap\; \Xi^c,	
	\end{align*}
	so, also using~\eqref{eq_aux_SZdn},
	\begin{align*}
		& \left| \P(\xi_{\tn}^{(1)} = \varnothing,\;\xi^{(2)}_{\tn} \in \mathscr{G}_N) - S^{\mathrm{box}}_{N^{1/2}}(\lambda,\lambda';B\cap Q_{\Z^d_N}(o,N^{1/2}))\right| \le \P(\Xi), \\[.2cm]
		& \left| \P(\xi^{(1)}_{\tn} \in \mathscr{G}_N,\;\xi_{\tn}^{(2)}= \varnothing) - (1- S^{\mathrm{box}}_{N^{1/2}}(\lambda,\lambda';B \cap Q_{\Z^d_N}(o,N^{1/2})))  )\right| \le \P(\Xi).
	\end{align*}

		We now bound the probabilities of the bad events. First,~$\Xi_1$ is contained in the event that the space-time set~$Q_{\Z^d_N}(o,N^{1/2}) \times [0,N^{1/4}]$ is not insulated; by Lemma~\ref{lem_insulation}, this has probability smaller than~$\exp\{-(\log N^{1/2})^2 \}$. Next, the probabilities of~$\Xi_2$ and~$\Xi_3$ are bounded using Lemma~\ref{lem_alternatives_torus} (with~$N \ge N_3$) and Corollary~\ref{cor_propagate_torus} (with~$\varepsilon = \varepsilon_0/2$ and~$N \ge N_2$, noting that~$\tn - N^{1/4} \in [\tfrac34 N^{1+\frac{\varepsilon_0}{2}},N^{\log(N)}]$), respectively. For~$\Xi_4$, we observe that, since~$\lambda'>\lambda >\lambda_c(\Z^d)$, the process~$(\{u:\xi_t(u)\neq \varnothing\})_{t \ge 0}$ stochastically dominates a (one-type) contact process with rate~$\lambda > \lambda_c(\Z^d)$ on~$\Z^d_N$, started from a configuration in~$\mathscr{G}_N$, so by~\eqref{eq_coupling_combining}, we have~$\P(\Xi_4) < 2N^{-\log N}$.

		To conclude, note that Lemma~\ref{lem_box_approx_Zd} gives
	\begin{equation*}
		|S^{\textrm{box}}_{N^{1/2}}(\lambda,\lambda';B \cap Q_{\Z^d_N}(o,N^{1/2})) - S(\lambda,\lambda';\psi^{-1}_N(B))| < \exp\{-(\log N^{1/2})^{3/2}\}.
	\end{equation*}
\end{proof}

\subsection{Back to the adaptive contact process: proofs of Propositions~\ref{prop_fixation_one_new} and~\ref{prop_no_fixation}} \label{ss_proofs_adap}
\begin{proof}[Proof of Proposition~\ref{prop_fixation_one_new}]
	Fix~$\lambda' > \lambda > \lambda_c(\Z^d)$ and~$\varepsilon > 0$. Let~$N$ be large (to be chosen later), and fix~$B \subseteq \Z^d_N \backslash \{o\}$. We take a coupling of~$(X_{N,t})_{t \ge 0}$, the adaptive contact process on~$\Z^d_N$, and~$(\xi_t)_{t \ge 0}$, the two-type contact process on~$\Z^d_N$ with rates~$\lambda,\lambda'$, with initial conditions, respectively, given by
	\[X_{N,0}(u) = \begin{cases} \lambda&\text{if } u \in B;\\ \lambda'&\text{if } u = o;\\0&\text{otherwise;}\end{cases}\qquad \xi_0(u) = \begin{cases} 1 &\text{if } u \in B;\\ 2&\text{if } u = o;\\0&\text{otherwise.}\end{cases}\]
	First take a probability space where~$(X_{N,t})_{t \ge 0}$ is defined. Then, for~$t < T_N$ set
	\[\xi^{(1)}_t = \{u:\;X_{N,t}(u) = \lambda\},\qquad \xi^{(2)}_t = \{u:\;X_{N,t}(u) = \lambda'\}.\]
	Next, recall that~$\mathcal{X}_N$ denotes the position of the newborn mutant on~$\{T_N < \infty\}$. In case the mutant is a child of an individual of type~$\lambda$, set
	\[\xi^{(1)}_{T_N} = \xi^{(1)}_{T_N-} \cup \{\mathcal{X}_N\},\qquad \xi^{(2)}_{T_N} = \xi^{(2)}_{T_N-},\]
and in case the mutant is a child of an individual of type~$\lambda'$, set
	\[\xi^{(1)}_{T_N} = \xi^{(1)}_{T_N-} ,\qquad \xi^{(2)}_{T_N} = \xi^{(2)}_{T_N-}\cup \{\mathcal{X}_N\}.\]
	Finally, enlarging the probability space, on the event~$\{T_N < \infty\}$, let~$(\xi_t)_{t > T_N}$ continue evolving as a two-type contact process, independently of~$(X_{N,t})_{t > T_N}$  (conditionally on $(X_{N,t},\xi_t)_{0 \le t \le T_N}$).

With this coupling, we have that
	\[
		\left\{\begin{array}{r} T>\tn,\; \{u:X_{N,\tn}(u) = \lambda'\} \in \mathscr{G}_N,\\[.2cm]\{u: X_{N,\tn}(u) = \lambda\} = \varnothing\end{array}\right\} = \{T > \tn,\; \xi^{(1)}_{\tn} = \varnothing,\;\xi^{(2)}_{\tn} \in \mathscr{G}_N\}
	\]
	and
	\[
		\left\{\begin{array}{r} T>\tn,\; \{u:X_{N,\tn}(u) = \lambda\} \in \mathscr{G}_N,\\[.2cm]\{u: X_{N,\tn}(u) = \lambda'\} = \varnothing\end{array}\right\} = \{T > \tn,\; \xi^{(1)}_{\tn} \in \mathscr{G}_N,\;\xi^{(2)}_{\tn} = \varnothing\}.
	\]
	Hence, both
	\[\left|\sum_{B' \in \mathscr{G}_N} \mathcal{V}(\lambda,\lambda';B,B')- \P(\xi^{(1)}_{\tn} = \varnothing,\;\xi^{(2)}_{\tn} \in \mathscr{G}_N) \right|\]
		and
	\[\left|\sum_{B' \in \mathscr{G}_N} \bar{\mathcal{V}}(\lambda,\lambda';B,B')- \P(\xi^{(1)}_{\tn} \in \mathscr{G}_N,\;\xi^{(2)}_{\tn} = \varnothing) \right|\]
	are bounded from above by~$\P(T \le \tn)$, which can be made as small as desired by taking~$N$ large, by Lemma~\ref{lem_bound_Ttn}. To conclude, Lemma~\ref{lem_box_approx_ZdN} shows that, for~$N$ large, uniformly in~$B$,~$\P(\xi^{(1)}_{\tn} = \varnothing,\;\xi^{(2)}_{\tn} \in \mathscr{G}_N)$ is close to~$S(\lambda,\lambda';\psi_N^{-1}(B))$, and~$\P(\xi^{(1)}_{\tn} \in \mathscr{G}_N,\;\xi^{(2)}_{\tn} = \varnothing)$ is close to~$1-S(\lambda,\lambda';B)$.
\end{proof}

\begin{proof}[Proof of Proposition~\ref{prop_no_fixation}]
	The statement can be proved with the same coupling as in the previous proof, and using Corollary~\ref{cor_propagate_torus}, with the roles of types 1 and 2 reversed. We omit the details.
\end{proof}

\newpage
\section{Appendix}\label{s_appendix}

\subsection{Proofs of results for classical contact process}\label{ss_proofs_classical}
In this section, we prove three results about the contact process on~$\Z^d_N$ that were stated earlier: Proposition~\ref{prop_density}, Proposition~\ref{prop_coupling} and Lemma~\ref{lem_two_proc_c0}. Before we can carry out these proofs, we need some preliminary work.  We start with an easy consequence of~\eqref{eq_from_tom_expectation}.
\begin{lemma}\label{lem_from_tom}
 Let $N\in\N$ and~$\lambda > \lambda_c(\Z^d)$, and let\color{black}~$c_\mathrm{long}$ be the constant from~\eqref{eq_from_tom_expectation}. Then, for the contact process~$(\zeta_t)_{t \ge 0}$ on~$Q_{\Z^d}(o,N)$ started from full occupancy, we have
	\begin{equation}
		\mathbb{P}\big(\zeta_{\exp\big\{\frac{c_\mathrm{long}}{2}\cdot N^d\big\}} \neq \varnothing\big) > 1 - \exp\Big\{-\frac{c_\mathrm{long}}{2}\cdot N^d\Big\}.
	\end{equation}
\end{lemma}
\begin{proof}
	Assume that { the contact process~$(\zeta_t)$ on $Q_{\Z^d}(o,N)$} is obtained from a graphical construction, and let~$\tau := \inf\{s \ge 0:\;\zeta_t = \varnothing\}$.
	The statement follows from~\eqref{eq_from_tom_expectation} and taking~$t = \exp\{\tfrac{c_\mathrm{long}}{2}N^d\}$ in the well-known inequality
	\begin{equation}\label{eq_well_known}
		\mathbb{P}(\tau \le t) \le \frac{t}{\mathbb{E}[\tau]},\qquad t > 0.
	\end{equation}
	For completeness, we now reproduce the proof of~\eqref{eq_well_known}.
	Fix~$t > 0$ and define
	\[\hat{\tau} := \inf\{s \in t\cdot \mathbb{N}:\; Q_{\Z^d}(o,N)\times \{s-t\} \not \rightsquigarrow  Q_{\Z^d}(o,N)\times \{s\}\}.\]
	Note that~$\tau \le \hat{\tau}$ and that~$\hat{\tau}$ is distributed as~$t \cdot X$, where~$X$ is a geometric random variable with parameter~$p:= \mathbb{P}(\tau \le t)$. This gives
	\[\mathbb{E}[\tau] \le \mathbb{E}[\hat{\tau}] = t\cdot \frac{1}{p} = \frac{t}{\mathbb{P}(\tau \le t)},\]
	so~\eqref{eq_well_known} follows.
\end{proof}

\begin{proof}[Proof of Proposition~\ref{prop_density}]
	Let~$Q = Q_{\Z^d_N}(u,N^\varepsilon)$ be a box of radius~$N^\varepsilon$ in~$\Z^d_N$. We bound
	\begin{align*}
		\P(\zeta^1_t \cap Q \neq \varnothing\; \forall t \le N^{\log N}) &= \P(\Z^d_N \times \{0\} \rightsquigarrow Q \times \{t\} \text{ for all } t \le N^{\log N})\\
		&\ge \P(Q \times \{0\} \rightsquigarrow Q \times \{t\} \text{ inside $Q$ for all } t \le N^{\log N})\\
		&\ge 1 - \exp\Big\{-\frac{c_\mathrm{long}}{2}\cdot N^{\varepsilon d}\Big\},
	\end{align*}
	where the last inequality follows from Lemma~\ref{lem_from_tom}. The statement of the proposition now follows from a union bound over all choices of~$Q$.
\end{proof}

The following coupling result is essentially taken from~\cite{M93}.
\begin{lemma}\label{lem_die_or_couple}
	For any~$\lambda > \lambda_c(\Z^d)$, the following holds for~$N$ large enough. Let~$(\zeta_t)_{t \ge 0}$ and~$(\zeta^1_t)_{t \ge 0}$ be two contact processes on~$Q_{\Z^d}(o,N)$ with rate~$\lambda$ obtained with the same graphical construction, with~$\zeta_0^1 \equiv 1$ and~$\zeta_0$ an arbitrary initial configuration. Then, for any~$t \ge N^6$ we have
	\[\mathbb{P}(\zeta_{t} \neq \varnothing,\; \zeta_{t} \neq \zeta^1_{t}) < \mathrm{e}^{-N}.\]
\end{lemma}
\begin{proof}
	Under a single graphical construction, for each~$A \subseteq Q_{\Z^d}(o,N)$, we take a contact process~$(\zeta^A_t)_{t \ge 0}$ started from~$\zeta_0^A = A$. Proposition~3.1 in~\cite{M93} states that there exists a sequence~$(a_N)_{N \ge 1}$ with $a_N/N^5 \xrightarrow{N \to \infty} 0$ such that
	\[\sup_{A \subseteq Q_{\Z^d}(o,N)}\mathbb{P}(\zeta_{a_N}^{A}\neq \varnothing,\;\zeta_{a_N}^A \neq \zeta_{a_N}^1) \xrightarrow{N\to \infty} 0 \]
	(note that the result is stated in that reference for~$d=2$, but as observed there, this is done for simplicity and the same holds in any dimension, {and the choice of $N^5$ does not change with the dimension}). By monotonicity, this also implies that
	\[\sup_{\substack{A,A' \subseteq Q_{\Z^d}(o,N)\\ A \subseteq A'}}\mathbb{P}(\zeta_{a_N}^{A}\neq \varnothing,\;\zeta_{a_N}^A \neq \zeta_{a_N}^{A'}) \xrightarrow{N\to \infty} 0. \]
	Assume that~$N$ is large enough that the supremum above is smaller than~$\mathrm{e}^{-1}$; then, for any~$t > a_N$,
	\begin{align*}&\mathbb{P}(\zeta_{t} \neq \varnothing,\; \zeta_{t} \neq \zeta^1_{t})\\&= \sum_{\substack{A,A' \subseteq Q_{\Z^d}(o,N)\\\varnothing \subsetneq A \subsetneq A'}} \mathbb{P}(\zeta_{t - a_N} = A,\; \zeta_{t - a_N}^1 =  A') \\[-.4cm]
	& \hspace{3.5cm} \cdot \mathbb{P}(\zeta_{t} \neq \varnothing,\; \zeta_{t} \neq \zeta^1_{t} \mid \zeta_{t - a_N} = A,\; \zeta_{t - a_N}^1 =  A') \\[.2cm]
	&\le \mathrm{e}^{-1} \cdot  \P(\zeta_{t - a_N} \neq \varnothing,\; \zeta_{t- a_N} \neq \zeta^1_{t-a_N}).\end{align*}
	Iterating this we get
	\[\mathbb{P}(\zeta_{t} \neq \varnothing,\; \zeta_{t} \neq \zeta^1_{t}) \le \mathrm{e}^{-\lfloor t/a_N\rfloor},\]
	so the desired bound follows, since~$\lfloor N^6/a_N \rfloor > N$ when~$N$ is large.
\end{proof}

\begin{lemma}\label{lem_individual}
	Let~$\lambda > \lambda_c(\Z^d)$ and~$c_\mathrm{long}$ be as in Lemma~\ref{lem_from_tom}. Let~$(\zeta_t)_{t \ge 0}$ be the contact process on~$Q_{\Z^d}(o,N)$ with~$\zeta_0(o) = 1$  and $\zeta_0(x)$ arbitrary for $x \in Q_{\Z^d}(o,N) \setminus \{ o \}$\color{black}. Then,
	\[\mathbb{P}\Big(\zeta_{\exp\big\{\frac{c_\mathrm{long}}{2}\cdot (N/2)^d\big\}} \neq \varnothing\Big) > \frac{1}{N^d}.\]
\end{lemma}
\begin{proof}
	Taking a graphical construction for the process on~$Q_{\Z^d}(o,N)$, by Lemma \ref{lem_from_tom} and a union bound we have
	\begin{equation*} 1-\mathrm{e}^{-\frac{c_\mathrm{long}}{2}\cdot (\frac{N}{2})^d} < \mathbb{P}\left( Q_{\Z^d}(o,\tfrac{N}{2})\times \{0\} \rightsquigarrow Q_{\Z^d}(o,\tfrac{N}{2}) \times \{\mathrm{e}^{c(N/2)^d} \}\text{ inside } Q_{\Z^d}(o,\tfrac{N}{2})\right),
	\end{equation*}
	 The probability is bounded from above by
	\[\sum_{x \in Q_{\Z^d}(o,N/2)} \mathbb{P}\left( (x,0) \rightsquigarrow Q_{\Z^d}(o,\tfrac{N}{2})\times \{\mathrm{e}^{c(N/2)^d} \} \text{ inside } Q_{\Z^d}(o,\tfrac{N}{2})\right).\]
	Since this sum is larger than~$1-\mathrm{e}^{-\frac{c_\mathrm{long}}{2}\cdot (\frac{N}{2})^d}$, there exists~$x^* \in Q_{\Z^d}(o,\tfrac{N}{2})$ for which the term inside the sum is larger than~$(1-\mathrm{e}^{-\frac{c_\mathrm{long}}{2}\cdot (\frac{N}{2})^d} )/|Q_{\Z^d}(o,\tfrac{N}{2})| \ge 1/N^d$. Finally, since~$o \in Q_{\Z^d}(-x^*,\tfrac{N}{2}) \subseteq Q_{\Z^d}(o,N)$, we have
	\begin{align*}&\mathbb{P}\left((o,0) \rightsquigarrow Q_{\Z^d}(o,N) \times \{\mathrm{e}^{c(N/2)^d} \}\right) \\
		&\ge \mathbb{P}\left((o,0) \rightsquigarrow Q_{\Z^d}(-x^*,\tfrac{N}{2}) \times \{\mathrm{e}^{c(N/2)^d}\} \text{ inside }Q_{\Z^d}(-x^*,\tfrac{N}{2})\right)\\
	&= \mathbb{P}\left( (x^*,0) \rightsquigarrow Q_{\Z^d}(o,\tfrac{N}{2})\times \{\mathrm{e}^{c(N/2)^d} \} \text{ inside } Q_{\Z^d}(o,\tfrac{N}{2})\right).\end{align*}
\end{proof}

\begin{lemma}\label{lem_numbers}
	Let~$\lambda > \lambda_c(\Z^d)$. For any~$\varepsilon>0$, the following holds for~$N$ large enough.  Let~$(\zeta_t)_{t \ge 0}$ be the contact process with rate~$\lambda$ on~$Q_{\Z^d}(o,N)$, started from~$\zeta_0$ with~$|\zeta_0\cap Q_{\Z^d}(o,\tfrac{N}{2})| \ge N^\varepsilon$. Then,
	\begin{equation}
		\mathbb{P}(\zeta_{\exp\{N^{\varepsilon/10}\}} \neq \varnothing) > 1-\exp\{-N^{\varepsilon/2}\}.
	\end{equation}
\end{lemma}
\begin{proof} We assume given a graphical construction for the process on~$Q_{\Z^d}(o,N)$. 
	Let~$\alpha := \frac{\varepsilon}{8d}$. For each~$x \in \zeta_0\cap Q_{\Z^d}(o,\tfrac{N}{2})$, there are fewer than~$(5N^\alpha)^d$ elements~$y \in \zeta_0\cap Q_{\Z^d}(o,\tfrac{N}{2})$ such that~$Q_{\Z^d}(x,N^\alpha) \cap Q_{\Z^d}(y,N^\alpha) \neq \varnothing$. Hence, letting~$k:= \lfloor N^\varepsilon/(5N^\alpha)^d \rfloor$, there exist~$x_1,\ldots, x_k \in \zeta_0\cap Q_{\Z^d}(o,\tfrac{N}{2})$ such that the boxes~$Q_{\Z^d}(x_i,N^\alpha)$, for~$1 \le i \le k$, are all disjoint. For any~$t \ge 0$ we have	
	\[\{\zeta_{t} \neq \varnothing\} \supseteq \bigcup_{i=1}^k \{(x_i,o) \rightsquigarrow Q_{\Z^d}(x_i,N^\alpha) \times \{t\} \text{ inside } Q_{\Z^d}(x_i,N^\alpha)\}.\]
	Let~$T := \exp\big\{\frac{c_\mathrm{long}}{2}\cdot (\frac{N^\alpha}{2})^d\big\}$. By Lemma~\ref{lem_individual} applied  to each of the boxes $Q_{\Z^d}(x_i,N^\alpha)$ we obtain
	\begin{align*}
	\mathbb{P}(\zeta_{T} \neq \varnothing) \ge 1-\left(1 -\frac{1}{(N^\alpha)^d} \right)^k \ge 1-\exp\left\{-\frac{k}{N^{\alpha d}}\right\} > 1- \exp\{-N^{\varepsilon/2}\}, 	\end{align*}
	where the last inequality follows  from the definition of~$k$ and the choice of~$\alpha$. To conclude, note that~$T > \exp\{N^{\varepsilon/10}\}$ if~$N$ is large enough.
\end{proof}

\begin{lemma}\label{lem_well_behaved}
	For any~$\lambda > \lambda_c(\Z^d)$ and~$\varepsilon > 0$, the following holds for~$N$ large enough. Given a graphical construction for the contact process with rate~$\lambda$ on~$\mathbb{Z}^d_N$, we have
	\begin{equation}\label{eq_two_linew} \begin{split}
		&\mathbb{P}\left(\begin{array}{l}\{(o,0) \not \rightsquigarrow \mathbb{Z}^d_N \times \{N^{\varepsilon}\}\} \;\cup \\[.1cm] \{ (o,0) \rightsquigarrow Q_{\Z^d_N}(o,N^\varepsilon) \times \{N^{2\log N}\} \text{ inside } Q_{\Z^d_N}(o,N^\varepsilon)\}\end{array}\right) \\
	&\hspace{6.5cm}> 1-\exp\{-N^{\varepsilon/20}\}.\end{split}	\end{equation}
\end{lemma}
\begin{proof}
	Let us define the events
	\begin{align*} &E_1:=\{(o,0) \not \rightsquigarrow \mathbb{Z}^d_N \times \{N^{\varepsilon}\}\},\\
		&E_2:= \{ (o,0) \rightsquigarrow Q_{\Z^d_N}(o,N^\varepsilon) \times \{N^{\log N}\} \text{ inside } Q_{\Z^d_N}(o,N^\varepsilon)\}.
	\end{align*}
	We let~$\zeta_t:= \{x:(o,0) \rightsquigarrow (x,t)\}$, so that~$(\zeta_t)_{t \ge 0}$ is a contact process on~$\mathbb{Z}^d_N$ started from~$\zeta_0 = \{o\}$. Let~$S:= N^{\varepsilon/2}/2$ and define the events
	\begin{align}\label{eq_events_referee} E_1':= \{\zeta_S = \varnothing\},\qquad E_2':=\big\{|\zeta_S| \ge \sqrt{S},\; \cup_{t \le S} \zeta_t \subseteq Q_{\Z^d_N}(o,S^2)\big\}.\end{align}
	{ Using Proposition~\ref{prop_facts_contact}, we have that}
	\begin{equation} \label{eq_bound_prim_prim}\mathbb{P}(E_1'\cup E_2') > 1-\exp\{-\sqrt{S}\}\end{equation}
	if~$N$ is large enough.  

	If~$N$ is large enough, on the event~$E_2'$ we have
	\begin{equation}
		|\zeta_S| > N^{\varepsilon/8},\quad \zeta_S \subseteq Q_{\Z^d_N}(o,N^{\varepsilon}/2),\quad (o,0) \rightsquigarrow \zeta_S \times \{S\} \text{ inside } Q_{\Z^d_N}(o,N^{\varepsilon}/2).
	\end{equation}
	Then, by Lemma~\ref{lem_numbers} applied to the restriction of the graphical construction to the space-time set~$Q_{\Z^d_N}(o,N^{\varepsilon}) \times [S,\infty)$, we have that
	\[\mathbb{P}(E_2'' \mid E_2') > 1- \exp\{-N^{\varepsilon{/16}}\},\]
	where
	\[E_2'':=  \left\{ \zeta_S \times \{S\} \rightsquigarrow Q_{\Z^d_N}(o,N^\varepsilon) \times \{S+\exp\{N^{\varepsilon/80} \} \} \text{ inside } Q_{\Z^d_N}(o,N^\varepsilon)  \right\}.\]
	Using this bound and~\eqref{eq_bound_prim_prim}, we obtain  for~$N$ sufficiently large \color{black}
	\begin{equation*}
		\P(E_1' \cup E_2'') > 1- \exp\{-\sqrt{S}\} - \exp\{-N^{\varepsilon/16}\} > 1 - \exp\{-N^{\varepsilon/20}\}.
	\end{equation*}
	Finally, note that~$E_1' \subseteq E_1$ and~$E_2'' \subseteq E_2$ { by the definition of~$E_1'$ and~$E_2'$ in~\eqref{eq_events_referee}}, so~$\P(E_1 \cup E_2) \ge \P(E_1' \cup E_2'')$.
\end{proof}

We are now ready to give the proofs of the two remaining results from Sections 3 and 4.
\begin{proof}[Proof of Lemma~\ref{lem_two_proc_c0}]
	Using translation invariance and duality, we bound, for any~$(u,t) \in \Z^d_N \times [N^{1/4},N^{2\log N}]$,
	\begin{align*}
		\P(\eta_t(u) \neq \zeta_t(u)) &= \P(\Z^d_N \times \{t-N^{1/4}\} \rightsquigarrow (u,t),\; \Z^d_N \times \{0\} \not \rightsquigarrow (u,t))\\
		&= \P((o,0) \rightsquigarrow \Z^d_N \times \{N^{1/4}\},\; (o,0) \not \rightsquigarrow \Z^d_N \times \{t\}).
	\end{align*}
	By Lemma~\ref{lem_well_behaved} with~$\varepsilon = \frac{1}{4}$, the right-hand side is smaller than~$ \exp\{-N^{1/80}\}.$
\end{proof}

	\begin{proof}[Proof of Proposition~\ref{prop_coupling}]
		Since~$\{\zeta_s = \zeta^1_s\} \subseteq \{\zeta_t = \zeta^1_t\}$ when~$s \le t$, it suffices to prove that
		\[
			\mathbb{P}(\zeta_{N^{12\varepsilon}} \neq \zeta^1_{N^{12\varepsilon}})< \exp\{-(\log N)^2\}.\]
		Let~$\mathcal{Q}$ denote the set of all boxes of radius~$N^{2\varepsilon}$ contained in~$\mathbb{Z}^d_N$. For~$Q \in \mathcal{Q}$ and~$A \subseteq Q$, we define
		\[\zeta^{A,Q}_t:= \{x \in Q: A \times \{0\} \rightsquigarrow (x,t) \text{ inside }Q\},\quad t \ge 0.\]
		Then,~$(\zeta^{A,Q}_t)_{t \ge 0}$ is a contact process on~$Q$ started from~$A$ occupied. Also note that
		\begin{equation}\label{eq_compss}
			\zeta^{1}_{N^{12\varepsilon}} \ge \zeta^{Q,Q}_{N^{12\varepsilon}} \qquad \text{and}\qquad \zeta_{N^{12\varepsilon}} \ge \zeta^{\zeta_0\cap Q,Q}_{N^{12\varepsilon}}.
		\end{equation}
  { That is, at time~$N^{12\varepsilon}$ (and indeed at any time), the process on~$\Z^d_N$ started from full occupancy is larger than the process restricted to~$Q$ started from full occupancy, and the process on~$\Z^d_N$ started from~$\zeta_0$ is larger than the process restricted to~$Q$ started from~$\zeta_0 \cap Q$.}
	
  { For each~$Q \in \mathcal{Q}$, let~$E_1(Q)$ be the event that the process restricted to~$Q$ and started from~$\zeta_0 \cap Q$ is alive at time~$N^{12\varepsilon}$:}
		\[ E_1(Q) := \left\{ \zeta^{\zeta_0 \cap Q,Q}_{N^{12\varepsilon}} \neq \varnothing\right\} ,\quad Q \in \mathcal{Q}.\]
		Given a box~$Q = Q_{\Z^d_N}(x,N^{2\varepsilon}) \in \mathcal{Q}$, the assumption of the proposition implies that every sub-box of radius~$N^\varepsilon$ that is contained in~$Q_{\Z^d_N}(x,\frac12 N^{2\varepsilon})$ intersects~$\zeta_0$. This implies that
		\[|\zeta_0 \cap Q_{\Z^d_N}(x,\tfrac12N^{2\varepsilon})| \ge (\tfrac12 N^{2\varepsilon}/N^{\varepsilon})^d = 2^{-d}N^{d\varepsilon} > N^{d\varepsilon/2} = (N^{2\varepsilon})^{d/4}.\]
		Hence, by Lemma~\ref{lem_numbers} (applied to the process~$(\zeta^{\zeta_0 \cap Q,Q}_t)$), we obtain
		\[\mathbb{P}(E_1(Q)) > \mathbb{P}\left(\zeta^{\zeta_0 \cap Q,Q}_{\exp\{N^{d\varepsilon/20} \}} \neq \varnothing \right) > 1-\exp\{-N^{d\varepsilon/4}\}.\]
		{ Next, for each~$Q \in \mathcal{Q}$, let~$E_2(Q)$ be the event that at time~$N^{12\varepsilon}$, the process restricted to~$Q$ and started from~$\zeta_0 \cap Q$ is either empty or has already coupled with the process restricted to~$Q$ and started from full occupancy:} 
		\[ E_2(Q) :=  \left\{\text{either }\; \zeta^{\zeta_0\cap Q,Q}_{N^{12\varepsilon}} = \varnothing \;\text{ or }\; \zeta^{\zeta_0\cap Q,Q}_{N^{12\varepsilon}} =  \zeta^{Q,Q}_{N^{12\varepsilon}}\right\},\quad Q \in \mathcal{Q}.\]
		By Lemma~\ref{lem_die_or_couple}, we have
		\[\mathbb{P}(E_2(Q)) > 1 - \exp\{-N^{2\varepsilon}\}\]
		for any~$Q \in \mathcal{Q}$, if~$N$ is large enough.

		Finally, for each~$x \in \mathbb{Z}^d_N$, 
  {
let~$E_3(x)$ be the event that either (1) the space-time point~$(x,N^{12\varepsilon})$ cannot be reached by an infection path started from anywhere in the torus at time~$0$, or (2) this space-time point can be reached by an infection path that starts at time 0, and spatially stays inside the box~$Q_{\Z^d_N}(x,N^\varepsilon)$. Formally,
  }
		\begin{align*} E_3(x) := \left\{\begin{array}{l} \text{either }\; \mathbb{Z}^d_N \times \{0\} \not \rightsquigarrow (x,N^{12\varepsilon})\\[.1cm] \text{ or }\; Q_{\Z^d_N}(x,N^{\varepsilon}) \times \{0\} \rightsquigarrow (x,N^{12\varepsilon}) \text{ inside } Q_{\Z^d_N}(x,N^{\varepsilon})\end{array}\right\}.
		\end{align*}
		By self-duality of the contact process and Lemma~\ref{lem_well_behaved}, we have
		\[\mathbb{P}(E_3(x)) > 1 - \exp\{-N^{\varepsilon^2/20}\}.\]
		
		The lower bounds on probabilities we have found imply that, letting
		\[E:= \left(\bigcap_{Q \in \mathcal{Q}} (E_1(Q) \cap E_2(Q))\right) \cap \left(\bigcap_{x \in \mathbb{Z}^d_N} E_3(x)\right),\]
		we have
		\begin{align*}\mathbb{P}(E) &> 1-|\mathcal{Q}|\cdot (\exp\{-N^{d\varepsilon/4}\}+\exp\{-N^{2\varepsilon}\}) - |\mathbb{Z}^d_N|\cdot \exp\{-N^{\varepsilon^2/20}\} \\&> 1 - \exp\{-(\log N)^2\}\end{align*}
		if~$N$ is large enough. 

		We now prove that if $E$ occurs, then $\zeta_{N^{12\varepsilon}} = \zeta^1_{N^{12\varepsilon}}$. The inequality $\zeta_{N^{12\varepsilon}} \le \zeta^1_{N^{12\varepsilon}}$ holds trivially. For the reverse inequality, fix $x \in \mathbb{Z}^d_N$ such that $\zeta^1_{N^{12\varepsilon}}(x) = 1$.  This means that $\mathbb{Z}^d_N \times \{0\} \rightsquigarrow (x,N^{12\varepsilon})$. Since $E_3(x)$ occurs, we obtain that $Q_{\Z^d_N}(x,N^{\varepsilon}) \times \{0\} \rightsquigarrow (x,N^{12\varepsilon})$ inside $Q_{\Z^d_N}(x,N^{\varepsilon})$. Abbreviating $Q = Q_{\Z^d_N}(x,N^{2\varepsilon})$, we also obtain $\zeta^{Q,Q}_{N^{12\varepsilon}}(x) = 1$. Next, since $E_1(Q)$ occurs we have $\zeta^{\zeta_0 \cap Q,Q}_{N^{12\varepsilon}} \neq \varnothing$, and then since $E_2(Q)$ occurs we have $\zeta^{\zeta_0 \cap Q,Q}_{N^{12\varepsilon}} = \zeta^{Q,Q}_{N^{12\varepsilon}}$. In particular, $\zeta^{\zeta_0 \cap Q,Q}_{N^{12\varepsilon}}(x) = 1$. By~\eqref{eq_compss}, we also obtain~$\zeta_{N^{12\varepsilon}}(x) = 1$.
%		\[\zeta_{N^{10\varepsilon}}(x) = \mathds{1}\{Q\times \{0\} \rightsquigarrow (x,N^{10\varepsilon}) \text{ inside }Q\} \ge \mathds{1}\{\mathbb{Z}^d_N\times \{0\} \rightsquigarrow (x,N^{10\varepsilon})\}= \zeta^{\zeta_0 \cap Q,Q}_{N^{10\varepsilon}}(x),\]
	\end{proof}

\subsection{Proofs of results for two-type contact process}\label{s_second_a}

Throughout this section, given rates~$\lambda, \lambda' > 0$, we always assume that, in a two-type contact process,~$\lambda$ denotes the birth rate of type~$1$ and~$\lambda'$ denotes the birth rate of type~$2$. 
 Recall from Section~\ref{ss_competition_lattice} the definition of good boxes for a two-type configuration. We will study the appearance and propagation of good boxes here. \color{black}

{
Define the total order~$\preceq$ on~$\{0,1,2\}$ by setting~$1 \preceq 0 \preceq 2$. Define a partial order on~$\{0,1,2\}^{\Z^d}$ (also denoted~$\preceq$, abusing notation) by 
\[\xi \preceq \xi' \quad \text{if and only if}\quad \xi(x) \preceq \xi'(x)\; \text{for all }x \in \Z^d.\]
Note that~$\xi \preceq \xi'$ if and only if
\[\{x:\xi(x) = 1\} \supseteq \{x:\xi'(x) = 1\} \quad \text{and}\quad \{x:\xi(x) = 2\} \subseteq \{x:\xi'(x) = 2\}.\]
The following is an elementary property of the two-type contact process. We include the proof for completeness. \begin{lemma}
	\label{lem_monotone_coupling}
	If~$(\xi_t)$ and~$(\xi'_t)$ are two two-type contact processes on~$\Z^d$ with same rates and obtained from the same graphical construction, then
\begin{equation}\label{eq_order}
	\xi_0 \preceq \xi_0' \quad \Longrightarrow \quad \xi_t \preceq \xi'_t \quad \text{for all } t \ge 0.
\end{equation}
\end{lemma}
\begin{proof}
Recall the description of the graphical construction of the two-type contact process (including death marks, basic arrows and extra arrows) in Section~\ref{ss_two_type_prelim}. In order to prove the statement of the lemma, it suffices to prove that the effects of death marks, basic arrows and extra arrows are all non-decreasing with respect to~$\preceq$. More precisely, we need to check that given~$\xi \preceq \xi'$, if~$\bar{\xi}$ and~$\bar{\xi}'$ are obtained from~$\xi$ and~$\xi'$, respectively, by performing the same operation (death mark, basic arrow, extra arrow, in the same location), then we have~$\bar{\xi} \preceq \bar{\xi}'$. This is trivial for death marks, so we will only discuss arrows.

	Given~$\xi \in \{0,1,2\}^{\Z^d}$ and~$(x,y) \in \Z^d$, let~$\Gamma^{(x,y)}_\mathrm{basic}(\xi) \in \{0,1,2\}^{\Z^d}$ be the configuration obtained from~$\xi$ by applying the effect of a basic arrow from~$x$ to~$y$. That is:
	\[[\Gamma^{(x,y)}_\mathrm{basic}(\xi)](z) := 
	\begin{cases}
		\xi(z) &\text{if } z \neq y;\\
		\psi_\mathrm{basic}(\xi(x),\xi(y)) &\text{if } z = y,
	\end{cases}
	\]
	where~$\psi_\mathrm{basic}: \{0,1,2\}^2 \to \{0,1,2\}$ is given by
	\[\psi_\mathrm{basic}(a,b) := 
	\begin{cases}
		1&\text{if } b = 1 \text{ or } [a=1,\;b=0];\\
		0&\text{if } a = b = 0;\\
		2&\text{if } b = 2 \text{ or } [a=2,\;b=0].
	\end{cases}\]
	We claim that~$\Gamma^{(x,y)}_\mathrm{basic}$ is non-decreasing with respect to~$\preceq$. This is easily seen by checking that~$\psi_\mathrm{basic}$ is non-decreasing, in the sense that if~$a \preceq a'$ and~$b \preceq b'$, then~$\psi_\mathrm{basic}(a,b) \preceq \psi_\mathrm{basic}(a',b')$. Indeed, in case~$b = 1$ or~$b'=2$, the relation~$\psi_\mathrm{basic}(a,b) \preceq \psi_\mathrm{basic}(a',b')$ is immediate, and in case~$b = b' = 0$, it follows from~$\psi_\mathrm{basic}(a,0) = a \preceq a' = \psi_\mathrm{basic}(a',0)$.

	Next, given~$\xi\in \{0,1,2\}^{\Z^d}$ and~$(x,y) \in \Z^d$, let~$\Gamma^{(x,y)}_\mathrm{extra}(\xi) \in \{0,1,2\}^{\Z^d}$ be the configuration obtained from~$\xi$ by applying the effect of a extra arrow from~$x$ to~$y$:
	\[[\Gamma^{(x,y)}_\mathrm{extra}(\xi)](z) := 
	\begin{cases}
		\xi(z) &\text{if } z \neq y;\\
		\psi_\mathrm{extra}(\xi(x),\xi(y)) &\text{if } z = y,
	\end{cases}
	\]
	where~$\psi_\mathrm{extra}: \{0,1,2\}^2 \to \{0,1,2\}$ is given by
	\[\psi_\mathrm{extra}(a,b) := 
	\begin{cases}
		1&\text{if } b = 1;\\
		0&\text{if } b = 0 \text{ and } a \neq 2;\\
		2&\text{if } b = 2 \text{ or } [a=2,\;b=0].
	\end{cases}\]
	The same argument as for~$\Gamma^{(x,y)}_\mathrm{basic}$ shows that this is non-decreasing with respect to~$\preceq$.
\end{proof}
}

The following lemma can be obtained by the same proof as that of Lemma~3.3 in~\cite{MPV20} (note that the roles we take here for types~$1$ and~$2$ are reversed with respect to this reference). We omit the details, to avoid introducing more notation and terminology. 
\begin{lemma}\label{lem_input1}
	Assume that~$\lambda' > \lambda_c(\Z^d)$ and~$\lambda' > \lambda$. There exist~$\beta > 0$ and~$m_0 \in \N$ such that the following holds for all~$m \ge m_0$. If~$(\xi_t)_{t \ge 0}$ is the two-type contact process on~$\Z^d$ with rates~$\lambda$ and~$\lambda'$ started from a configuration~$\xi_0$ such that~$\xi_0(x) = 2$ for all~$x \in Q_{\Z^d}(o,m)$  (with~$\xi_0(x)$ being arbitrary for any~$x \in \Z^d \setminus Q_{\Z^d}(o,m)$)\color{black}, then
	\[\P(\xi_t(x) \neq 1 \text{ for all }(x,t) \text{ such that }  x \in Q_{\Z^d}(o,\tfrac{m}{2} + \beta t)) > 1 - \exp\{-\sqrt{m}\}.\]
\end{lemma}

\begin{corollary}\label{lem_input15}
	Assume that~$\lambda' > \lambda_c(\Z^d)$ and~$\lambda' > \lambda$, and let~$\beta$ be as in Lemma~\ref{lem_input1}. Then, there exists~$m_1 > 0$ such that the following holds for all~$m \ge m_1$ and~$\ell \ge m$. Let~$(\xi_t)_{t \ge 0}$ be the two-type contact process on~$\Z^d$ with rates~$\lambda$ and~$\lambda'$ started from a configuration~$\xi_0$ such that~$\xi_0(x) = 2$ for all~$x \in Q_{\Z^d}(o,m)$ (with~$\xi_0(x)$ being arbitrary for any~$x \in \Z^d \setminus Q_{\Z^d}(o,m)$)\color{black}. Then, 
	\[\P(Q_{\Z^d}(o,\ell) \text{ is good for } \xi_t) >1- \exp\{-\sqrt{m}\} - \exp\{-\ell^{d/25}\}\]
 { for all~$t > \frac{2\ell - m}{2\beta}$}.
\end{corollary}
\begin{proof} Using elementary stochastic domination considerations involving~\eqref{eq_order}, it is easy to see that it is sufficient to prove the statement under the additional assumptions that~$\lambda > \lambda_c(\Z^d)$ and that the initial configuration is
	\[\xi_0(x) = \begin{cases} 2&\text{if } x \in Q_{\Z^d}(o,m);\\ 1&\text{otherwise.}\end{cases}\]
Let~$m_1$ be a large constant to be chosen later (for now,~$m_1 \ge m_0$, where~$m_0$ is the constant of Lemma~\ref{lem_input1}), and fix~$m,\ell,t$ as in the statement {(so that~$\tfrac{m}{2} + \beta t > \ell$)}.  		Noting that~$Q_{\Z^d}(o,\tfrac{m}{2} + \beta t) \supseteq Q_{\Z^d}(o,\ell)$, Lemma~\ref{lem_input1} implies that the event
	 \[E_1 := \{\xi_{t}(x) \neq 1 \text{ for all } x \in Q_{\Z^d}(o,\ell)\}\]
	 has probability larger than~$1 - \exp\{-\sqrt{m}\}$. Moreover, since the initial configuration~$\xi_0$ has no empty sites and~$\lambda'> \lambda$, we have that the process~$(\{x:\xi_t(x) \neq 0\})_{t \ge 0}$ stochastically dominates a contact process~$(\zeta_t)_{t \ge 0}$ on~$\Z^d$ with rate~$\lambda$ and started from~$\zeta_0 \equiv 1$. Hence, by~\eqref{eq_const_death}  and duality, the event
	 \[E_2 := \{\{x:\xi_{t}(x) \neq 0\} \text{ intersects all boxes of radius $\ell^{1/24}$ inside $Q_{\Z^d}(o,\ell)$}\}\]
	 has probability above~$1 - \ell^d\cdot \exp\{-c_\mathrm{death}(\lambda)\cdot \ell^{d/24}\}$, { (recall that $c_\mathrm{death}$ was introduced in  \eqref{eq_const_death})}. If~$m_1$ is large enough (and hence~$\ell$ is large enough, since~$\ell \ge m \ge m_1$), this probability is above~$1-\exp\{-\ell^{d/25}\}$. To conclude, note that if~$E_1 \cap E_2$ occurs, we have that~$\{x:\xi_{t}(x) = 1\}$ does not intersect~$Q_{\Z^d}(o,\ell)$, and~$\{x:\xi_{t}(x) = 2\}$ intersects all boxes of radius~$\ell^{1/24}$ contained in~$Q_{\Z^d}(o,\ell)$; therefore,~$Q_{\Z^d}(o,\ell)$ is good for~$\xi_{t}$.
\end{proof}

\begin{proof}[Proof of Lemma~\ref{lem_appearance}]
	We let~$\ell_0 := m_1$ and assume that~$\ell \ge \ell_0$. By using the fact that~$\xi_0(o) = 2$ and setting up the Poisson marks in the graphical construction in the time interval~$[0,1]$ in an appropriate way, it is easy to see that there exists~$\alpha_0 > 0$ (depending only on~$\lambda,\lambda',\ell_0$) such that
	\[\P(\xi_1(x) = 2 \text{ for all } x \in Q_{\Z^d}(o,\ell_0)) > \alpha_0.\]
	Increasing~$\ell_0$ if necessary, we have~$\frac{\ell_0}{2} + \beta \cdot (\ell^2 - 1) > \ell$ for any~$\ell \ge \ell_0$.  Then, by Corollary~\ref{lem_input15} and the Markov property,
	\begin{align*}&\P(Q_{\Z^d}(o,\ell) \text{ is good for } \xi_{\ell^2} \mid \xi_1(x) = 2 \text{ for all } x \in Q_{\Z^d}(o,\ell_0)) \\
	&\hspace{3cm}> 1-\exp\{-\sqrt{\ell_0}\}-\exp\{-\ell^{d/25}\}.\end{align*}
	The result now follows, with~$\sigma_0 := \alpha_0 \cdot (1-\exp\{-\sqrt{\ell_0}\}-\exp\{-\ell_0^{d/25}\})$.
\end{proof}

The following corollary has a very similar proof to that of Corollary~\ref{lem_input15}, so we omit the details.
\begin{corollary}\label{lem_input2}
	Assume that~$\lambda' > \lambda_c(\Z^d)$ and~$\lambda' > \lambda$. For any~$\varepsilon > 0$, the following holds for~$\ell$ large enough. Assume that~$(\xi_t)_{t \ge 0}$ is the two-type contact process on~$\Z^d$ with rates~$\lambda$ and~$\lambda'$ and started from a configuration~$\xi_0 $ such that~$\xi_0(x) = 2$ for all~$x \in Q_{\Z^d}(o,\ell)$ { (with~$\xi_0(x)$ being arbitrary for any~$x \in \Z^d \backslash Q_{\Z^d}(o,\ell)$) and~$\bar{t} \ge \ell^{1+\varepsilon}$.} Then,
	\[\mathbb{P}(\text{the boxes } \{Q_{\Z^d}(x,\ell):\|x\|_\infty \le 10\ell\} \text{ are all good for $\xi_{ \bar{t}}$}) > 1- \exp\{-\ell^{d/25}\}.\]
\end{corollary}

\begin{proof}[Proof of Lemma~\ref{lem_propagation}]
	Fix~$\varepsilon > 0$ and let~$\ell > 0$ be large, to be chosen during the proof. {Also fix~$\bar{t} \ge \ell^{1+\varepsilon}$, as in the statement of the lemma.}
	Let~$H$ be a graphical construction for the two-type contact process on~$\Z^d$ with rates~$\lambda$ and~$\lambda'$. We assume that this graphical construction is used for the process~$(\xi_t)_{t \ge 0}$ that appears in the statement of the lemma (started from a configuration~$\xi_0$ for which~$Q_{\Z^d}(o,\ell)$ is good), as well as all other processes that appear in the proof.

	Let~$(\zeta_t)_{t \ge 0}$ be the process started from
	\[\zeta_0(x) := \begin{cases}1&\text{if } x \notin Q_{\Z^d}(o,\ell);\\ 2&\text{if } x \in Q_{\Z^d}(o,\ell/4);\\ 0&\text{otherwise.}\end{cases}\]
Define the events
	\[E_1:= \{\text{the boxes }\{Q_{\Z^d}(x,\ell): \|x\|_\infty \le  \ell\} \text{ are all good for $\zeta_{\bar{t}}$}\}\]
	and
	\[E_2:= \{\zeta_{\sqrt{\ell}} \preceq \xi_{\sqrt{\ell}}\}.\]
	Note that, by~\eqref{eq_order}, on~$E_2$ we have~$\zeta_t \preceq \xi_t$ for all~$t \ge \sqrt{\ell}$; in particular, this holds for~$t = { \bar{t}}$. Hence, on~$E_1 \cap E_2$ we have that the boxes~$\{Q_{\Z^d}(x,\ell):\|x\|_\infty \le \ell\}$ are all good for~$\xi_{ \bar{t}}$. 

	Hence, the proof will be complete once we show that~$\P(E_1^c)$ and~$\P(E_2^c)$ are both much smaller than~$\exp\{-(\log \ell)^2\}$. For~$\P(E_1^c)$, this follows readily from Corollary~\ref{lem_input2}. In treating the bound for~$\P(E_2^c)$, we will only outline the main steps, since very similar ideas already appeared in Section~\ref{s_fixation}: box insulation arguments and bounds on the speed of the contact process. In what follows, we will say that an event occurs ``\textit{with high probability}'' when its complement occurs with probability much smaller than~$\exp\{-(\log \ell)^2\}$ as~$\ell \to \infty$.

	We claim that
	\begin{equation}\label{eq_for_three_steps}\{x: \xi_{\sqrt{\ell}}(x) = 2\} \supseteq \{x: \zeta_{\sqrt{\ell}}(x) = 2\} \quad \text{with high probability}.\end{equation}
	To argue for this, we follow three steps. First, consider the process~$(\xi'_t)_{t \ge 0}$ started from~$\xi_0'$ defined by
	\[\xi_0'(x) = \begin{cases}\xi_0(x)&\text{if } x \in Q_{\Z^d}(o,\ell);\\0&\text{otherwise.}\end{cases}\]
		By using a box insulation argument, we can show that, with high probability,
		\begin{equation}\label{eq_first_step}\{x: \xi_{\sqrt{\ell}}(x) = 2\} \cap Q_{\Z^d}(o,\ell/2) = \{x: \xi'_{\sqrt{\ell}}(x) = 2\} \cap Q_{\Z^d}(o,\ell/2).\end{equation}
		Second, let the process~$(\zeta_t')_{t \ge 0}$ be started from
		\[\zeta_0'(x)= \begin{cases} 2&\text{if } x \in Q_{\Z^d}(o,\ell/4);\\0&\text{otherwise.}\end{cases}\]
			Using a speed bound for the contact process (cf.~\eqref{eq_constant_speed})\color{black}, we have that with high probability,
			\[\{x:\zeta_{\sqrt{\ell}}'(x) = 2\} \subseteq Q_{\Z^d}(o,\ell/2).\]
			Then, using the fact that~$Q_{\Z^d}(o,\ell)$ is good for~$\xi'_0$, together with the coupling in Proposition~\ref{prop_coupling}, we obtain that with high probability,
	\begin{equation}\label{eq_second_step}\{x:\zeta'_{\sqrt{\ell}}(x) = 2\} \subseteq \{x:\xi'_{\sqrt{\ell}}(x) = 2\} \cap Q_{\Z^d}(o,\ell/2).\end{equation}
		Third, by another box insulation argument, we have that with high probability,
		\begin{equation}\label{eq_third_step} \{x:\zeta'_{\sqrt{\ell}}(x) = 2\} = \{x:\zeta_{\sqrt{\ell}}(x) = 2\}.\end{equation}
			Now, if~\eqref{eq_first_step},~\eqref{eq_second_step} and~\eqref{eq_third_step} all occur, then we have
			\[ \{x:\xi_{\sqrt{\ell}}(x) = 2\} \supseteq \{x:\zeta_{\sqrt{\ell}}(x) = 2\},\]
			proving~\eqref{eq_for_three_steps}.

			We now claim that
			\begin{equation} \label{eq_again_steps}
			\{x:\xi_{\sqrt{\ell}}(x) = 1\} \subseteq \{x: \zeta_{\sqrt{\ell}}(x) = 1\} \quad \text{with high probability}.\end{equation}
	To prove this, we define yet another process~$(\zeta_t'')_{t \ge 0}$, started from
	\[\zeta_0''(x) = \begin{cases}1&\text{if } x \notin Q_{\Z^d}(o,\ell);\\0&\text{otherwise.}\end{cases}\]
		By stochastic domination considerations we have that
		\[\{x:\xi_{\sqrt{\ell}}(x) = 1\} \subseteq \{x:\zeta''_{\sqrt{\ell}}(x) = 1\}.\]
		Next, by a box insulation argument, with high probability we have
		\[\{x:\zeta''_{\sqrt{\ell}}(x) = 1\} = \{x: \zeta_{\sqrt{\ell}}(x) = 1\}.\]
		Putting these facts together gives~\eqref{eq_again_steps}. Taking~\eqref{eq_for_three_steps} and~\eqref{eq_again_steps} together completes the proof.
\end{proof}

%%%%%%%%%%%%%%%%%%%%%%%%%%%%%%%%%%%%%%%%%%%%%%
%% Single Appendix:                         %%
%%%%%%%%%%%%%%%%%%%%%%%%%%%%%%%%%%%%%%%%%%%%%%
%\begin{appendix}
%\section*{???}%% if no title is needed, leave empty \section*{}.
%\end{appendix}
%%%%%%%%%%%%%%%%%%%%%%%%%%%%%%%%%%%%%%%%%%%%%%
%% Multiple Appendixes:                     %%
%%%%%%%%%%%%%%%%%%%%%%%%%%%%%%%%%%%%%%%%%%%%%%
%\begin{appendix}
%\section{???}
%
%\section{???}
%
%\end{appendix}

%%%%%%%%%%%%%%%%%%%%%%%%%%%%%%%%%%%%%%%%%%%%%%
%% Support information, if any,             %%
%% should be provided in the                %%
%% Acknowledgements section.                %%
%%%%%%%%%%%%%%%%%%%%%%%%%%%%%%%%%%%%%%%%%%%%%%
\begin{acks}[Acknowledgments]
	{
 The authors would like to thank the three anonymous referees for their careful reading of this work and suggestions.
 
 The authors thank Leonardo Rolla for giving a suggestion that led to the proof of Proposition~\ref{prop_leo}.}

The research of AT was partially supported by the ERC Synergy under
Grant No.\ 810115 - DYNASNET.

The research of AGC was partially supported by CONACYT Ciencia Basica (CB-A1-S-14615).
\end{acks}

\end{document}